\documentclass[english]{smfart}
\usepackage[francais,english]{babel}
\usepackage{smfthm}

\usepackage{amssymb}
\usepackage{euscript}
\usepackage[all]{xypic}
\usepackage{url}
\newcommand{\scr}[1]{\EuScript{#1}}
\newcommand{\ip}[1]{\langle #1 \rangle}
\newcommand{\Hom}{\mathop{\mathrm{Hom}}\nolimits}
\newcommand{\End}{\mathop{\mathrm{End}}\nolimits}
\newcommand{\Spf}{\mathop{\mathrm{Spf}}\nolimits}
\newcommand{\Sp}{\mathop{\mathrm{Sp}}\nolimits}
\newcommand{\Spec}{\mathop{\mathrm{Spec}}\nolimits}
\newcommand{\SL}{\mathop{\mathrm{SL}}\nolimits}
\newcommand{\ord}{\mathop{\mathrm{ord}}\nolimits}
\renewcommand{\div}{\mathrm{div}}
\def\Q{\mathbb{Q}}
\def\Z{\mathbb{Z}}
\def\C{\mathbb{C}}
\def\F{\mathbb{F}}
\def\A{\mathbb{A}}
\def\OO{\EuScript{O}}
\def\Tate{\underline{\mathrm{Tate}}}
\def\an{\mathrm{an}}
\newcommand{\E}{{\mathbf{E}}}
\renewcommand{\H}{{\mathbf{H}}}
\newcommand{\Fr}{{\mathrm{Fr}}}

\newcommand{\rig}{{\mathrm{rig}}}
\newcommand{\spe}{{\mathrm{sp}}}
\newcommand{\q}{{\mathbf{q}}}
\newcommand{\m}{{\mathbf{m}}}

\author{Nick Ramsey} \address{Department of Mathematics, 
  University of Michigan}
\email{naramsey@umich.edu} 
\title{The half-integral weight eigencurve} 
\thanks{This research is supported in part by NSF Grant DMS-0503264}

\begin{document}
\frontmatter

\begin{abstract}
  In this paper we define Banach spaces of overconvergent
  half-integral weight $p$-adic modular forms and Banach modules of
  families of overconvergent half-integral weight $p$-adic modular
  forms over admissible open subsets of weight space.  Both spaces are
  equipped with a continuous Hecke action for which $U_{p^2}$ is
  moreover compact.  The modules of families of forms are used to
  construct an eigencurve parameterizing all finite-slope systems of
  eigenvalues of Hecke operators acting on these spaces.  We also
  prove an analog of Coleman's theorem stating that overconvergent
  eigenforms of suitably low slope are classical.
\end{abstract}

\maketitle

\tableofcontents

\theoremstyle{plain}
\newtheorem{conv}[smfthm]{Convention}

\theoremstyle{remark}
\newtheorem{rems}[smfthm]{Remarks}

\section{Introduction}

In \cite{mfhi}, the author set up a geometric theory of modular forms
of weight $k/2$ for odd positive integers $k$, complete with
geometrically defined Hecke operators.  This approach naturally led
to a theory of overconvergent $p$-adic modular forms of such weights
equipped with a Hecke action for which $U_{p^2}$ is compact.

In this paper we define overconvergent half-integral weight $p$-adic
modular forms of general $p$-adic weights, as well as rigid-analytic
families thereof over admissible open subsets of weight space.  We use
the latter spaces and Buzzard's eigenvariety machine
(\cite{buzzardeigenvarieties}) to construct a half-integral weight
eigencurve parameterizing all systems of eigenvalues of Hecke
operators occurring on spaces of half-integral weight overconvergent
eigenforms of finite slope.  In contrast to the integral weight
situation, this space \emph{does not} parameterize actual forms
because a half-integral weight form that is an eigenform for all Hecke
operators is not always characterized by its weight and collection of
Hecke eigenvalues.  We also prove an analog of Coleman's result that
overconvergent eigenforms of suitably low slope are classical.

This paper lays the foundation for a forthcoming paper of the author
in which we construct a map from our half-integral weight eigencurve
to its integral weight counterpart (at least after passage to the
underlying reduced spaces) that rigid-analytically interpolates the
classical Shimura lifting introduced in \cite{shimura}.\\

\noindent {\bf Acknowledgments}\\

The author extends his thanks to Brian Conrad for writing the
appendix, as well as for numerous helpful discussions and suggestions
about the technical issues in Section \ref{sec:prelims}.  The author
would also like to thank the referee for several suggestions on the
manuscript and for directing him to some good references to help deal
with the case $p=2$.

\section{Preliminaries}\label{sec:prelims}

\subsection{General Notation}

Fix a prime number $p$. The symbol $K$ will always denote a complete
and discretely-valued field extension of $\Q_p$.  For such $K$ we
denote the ring of integers by $\OO_K$ and the maximal ideal therein by
$m_K$.  The absolute value on $K$ will always be normalized by $|p|=
1/p$.

\subsection{Modular Curves}\label{sec:modcurves}

For positive integers $N$ and $n$, $X_1(N)$ and $X_1(N,n)$ will denote
the usual moduli stacks of generalized elliptic curves with level
structure.  The former classifies generalized elliptic curves with a
point $P$ of order $N$ while the latter classifies generalized
elliptic curves with a pair $(P,C)$ consisting of a point $P$ of order
$N$ and a cyclic subgroup $C$ of order $n$ meeting the subgroup
generated by $P$ trivially (plus a certain ampleness condition for
non-smooth curves). This level structure will always be taken to be
the Drinfeld-style level structure found in \cite{katzmazur},
\cite{conradamgec}, and the appendix to this paper, and in all cases
the base ring will be a $\Z_{(p)}$-algebra.

Throughout this paper we will make extensive use of certain admissible
opens in rigid spaces associated to some of these modular curves.
Traditionally these opens were defined using the Eisenstein series
$E_{p-1}$, but this requires that we pose unfavorable restrictions on
$p$ and $N$.  Fortunately, more recent papers of Buzzard
(\cite{buzzard}) and Goren-Kassaei (\cite{gorenkassaei}) define these
opens and explore their properties in greater generality using
alternative techniques.
These authors define a ``measure of singularity'' $v(E)\in \Q^{\geq
  0}$ associated to an elliptic curve over a complete extension of
$\Q_p$.  In case $v(E)\leq p/(p+1)$, one may associate to $E$ a
canonical subgroup $H_p(E)$ of order $p$ in an appropriately
functorial manner.  Moreover, one understands $v(E/C)$ for finite
cyclic subgroups $C\subseteq E$ as well as the canonical subgroup of
$E/C$ when it exists.  Inductively applying this with $C=H_p(E)$, one
can define (upon further restricting $v(E)$) canonical subgroups
$H_{p^m}(E)$ of higher $p$-power order.  For details regarding these
constructions and facts, we refer the reader to Section 3 of
\cite{buzzard} and Section 4 of \cite{gorenkassaei}.

We will denote the Tate elliptic curve over $\Z(\!(q)\!)$ by
$\Tate(q)$ (see \cite{katz}).  Our notational conventions concerning
the Tate curve differ from those often found in the literature as
follows.  In the presence of, for example, level $N$ structure,
previous authors (e.g.,\ \cite{katz}) have preferred to consider the
curve $\Tate(q^N)$ over the base $\Z(\!(q)\!)$.  Points of order $N$
on this curve are used to characterize the behavior of a modular form
at the cusps, and are all defined over the fixed ring
$\Z(\!(q)\!)[\zeta_N]$ (where $\zeta_N$ is some primitive
$N^{\small\mathrm{th}}$ root of $1$).  We prefer to fix the curve
$\Tate(q)$ and instead consider \emph{extensions} of the base.  Thus,
in the presence of level $N$ structure, we introduce the formal
variable $q_N$, and \emph{define} $q=q_N^N$.  Then the curve
$\Tate(q)$ is defined over the sub-ring $\Z(\!(q)\!)$ of
$\Z(\!(q_N)\!)$ and all of its $N$-torsion is defined over the ring
$\Z(\!(q_N)\!)[\zeta_N]$.  To be precise, the $N$-torsion is given by
$$\zeta_N^iq_N^j,\ \ 0\leq i,j\leq N-1.$$
Cusps will always be
referred to by specifying a level structure on the Tate curve.

Suppose that $N\geq 5$ so that we have a fine moduli scheme
$X_1(N)_{\Q_p}$ and let $K/\Q_p$ be a finite extension (which will
generally be fixed in applications).  If $r\in [0,1]\cap \Q$, then the
region in the rigid space $X_1(N)^\an_K$ whose points correspond to
pairs $(E,P)$ with $v(E)\leq r$ is an admissible affinoid open.  We
denote by $X_1(N)^\an_{\geq p^{-r}}$ the connected component of this
region that contains the cusp associated to the datum
$(\Tate(q),\zeta_N)$ for some (equivalently, any) choice of primitive
$N^{\small\mathrm{th}}$ root of unity $\zeta_N$.  Similarly,
$X_1(N,n)^\an_{\geq p^{-r}}$ will denote the connected component of
the region defined by $v(E)\leq r$ in $X_1(N,n)^\an_K$ containing the
cusp associated to $(\Tate(q),\zeta_N,\ip{q_n})$ for any such
$\zeta_N$.  For smaller $N$ one defines these spaces by first adding
prime-to-$p$ level structure to rigidify the moduli problem and
proceeding as above, and then taking invariants.  Similarly, the space
$X_0(N)^\an_{\geq p^{-r}}$ is defined as the quotient of
$X_1(N)^\an_{\geq p^{-r}}$ by the action of the diamond operators.
The reader may wish to consult Section 6 of
\cite{buzzardeigenvarieties} for a more detailed discussion of these
quotients.

\subsection{Norms}\label{norms}

If $\mathfrak{X}$ is an admissible formal scheme over $\OO_K$ (in the
sense of \cite{formalrigid}), we will denote its (Raynaud) generic
fiber by $\mathfrak{X}_\rig$ and its special fiber by
$\mathfrak{X}_0$.  In case $\mathfrak{X} = \Spf(\scr{A})$ is a formal
affine we have $\mathfrak{X}_\rig = \Sp(\scr{A}\otimes_{\OO_K}K)$ and
$\mathfrak{X}_0 = \Spec(\scr{A}/\pi\scr{A})$ where $\pi\in \OO_K$ is a
uniformizer.  We recall for later use that the natural specialization
map
$$\spe:\mathfrak{X}_\rig \longrightarrow \mathfrak{X}_0$$
is
surjective on the level of closed points (see Proposition 3.5 of
\cite{formalrigid}).

Assume that $\mathfrak{X}$ is reduced and let $\mathfrak{L}$ be an
invertible sheaf on $\mathfrak{X}$ (that is to say, a sheaf of modules
on this ringed space that is Zariski-locally free of rank one).  For a
point $x\in \mathfrak{X}_\rig(L)$ let
$$\widehat{x}:\Spf(\OO_L)\longrightarrow \mathfrak{X}$$
denote the
unique extension of $x$ to the formal model.  Then the canonical
identification
$$H^0(\Sp(L),x^*\mathfrak{L}_\rig) =
H^0(\Spf(\OO_L),\widehat{x}^*\mathfrak{L})\otimes_{\OO_L}L$$
furnishes a
norm $|\!\cdot\!|_x$ on this one-dimensional vector space by declaring
the formal sections on the right to be the unit ball.  Now for any
admissible open $\scr{U}\subseteq\mathfrak{X}_\rig$ and any $f\in
H^0(\scr{U},\mathfrak{L}_\rig)$ we define $$\|f\|_{\scr{U}} =
\sup_{x\in\scr{U}}|x^*f|_x.$$
Note that, in case, $\mathfrak{L} =
\OO_{\mathfrak{X}}$, this is simply the usual supremum norm on
functions.

There is no reason for $\|f\|_{\scr{U}}$ to be finite in general, but
in case $\scr{U}$ is affinoid then this is indeed finite and endows
$H^0(\scr{U},\mathfrak{L}_\rig)$ with the structure of a Banach space
over $K$ as we now demonstrate.
\begin{lemm}\label{banachspaces}
  Let $\mathfrak{X}$ be a reduced quasi-compact admissible formal
  scheme over $\OO_K$, let $\mathfrak{L}$ be an invertible sheaf on
  $\mathfrak{X}$, and let $\scr{U}$ be an admissible affinoid open in
  $\mathfrak{X}_\rig$.  Then $H^0(\scr{U},\mathfrak{L}_\rig)$ is a
  $K$-Banach space with respect to $\|\!\cdot\!\|_{\scr{U}}$ .
\end{lemm}
\begin{proof}
  By Raynaud's theorem there is quasi-compact admissible formal blowup
  $\pi:\mathfrak{X}'\longrightarrow \mathfrak{X}$ and an admissible
  formal open $\mathfrak{U}$ in $\mathfrak{X}'$ with generic fiber
  $\scr{U}$.  For $x\in\scr{U}$ let $\widehat{x}'$ denote the unique
  extension to an $\OO_L$-valued point of $\mathfrak{U}$ and let
  $\widehat{x}$ denote its image in $\mathfrak{X}$ (which is the same
  $\widehat{x}$ as above by uniqueness).  Then we have
  $$H^0(\Spf(\OO_L), \widehat{x}'^*\pi^*\mathfrak{L}) =
  H^0(\Spf(\OO_L),\widehat{x}^*\mathfrak{L})$$
  as lattices in
  $H^0(\Sp(L),\mathfrak{L}_\rig)$.  It follows that $|f|_x =
  |\pi^*f|_x$ and we may compute $\|f\|_{\scr{U}}$ using the models
    $\mathfrak{X}'$ and $\pi^*\mathfrak{L}$, and hence we may as well
    assume that $\scr{U}$ is the generic fiber of an admissible formal
    open $\mathfrak{U}$ in $\mathfrak{X}$.  Furthermore, we may just
    well replace $\mathfrak{X}$ by $\mathfrak{U}$ and assume that
    $\mathfrak{U}$ is the generic fiber of $\mathfrak{X}$ itself.
  
  Cover $\mathfrak{X}$ by a finite collection of admissible formal
  affine opens $\mathfrak{U}_i$ trivializing $\mathfrak{L}$ and pick a
  trivializing section $\ell_i$ of $\mathfrak{L}$ on $\mathfrak{U}_i$.
  Let $\scr{U}_i= (\mathfrak{U}_i)_\rig$, so that the $\scr{U}_i$ form
  an admissible cover of $\scr{U}$ by admissible affinoid opens.
  Then, for any section $f\in H^0(\scr{U},\mathfrak{L}_\rig)$, we may
  write $f|_{\scr{U}_i} = a_i\ell_i$ for a unique $a_i\in
  \OO(\scr{U}_i)$, and one easily checks that $$\|f\|_{\scr{U}} =
  \max_i\|a_i\|_{\sup}.$$
  The desired assertion now follows easily
  from the analogous assertion about the supremum norm on a reduced
  affinoid.
\end{proof}

The following lemma and its corollary establish a sort of maximum
modulus principle for these norms.
\begin{lemm}\label{supmaxprinc}
  Let $\mathfrak{X}= \Spf(\scr{A})$ be a reduced admissible affine
  formal scheme over $\OO_K$ and let $U\subseteq \mathfrak{X}_0$ be a
  Zariski-dense open subset of the special fiber.  Suppose that the
  generic fiber $X=\Sp(\scr{A}\otimes_{\OO_K}K)$ is equidimensional.
  Then, for any $a\in \scr{A}\otimes_{\OO_K}K$, the supremum norm of $a$
  over $X$ is achieved on $\spe^{-1}(U)$.
\end{lemm}
\begin{proof}
  Let us first prove the lemma in the case that $\scr{A}$ is normal.
  First note that if $\|a\|_{\sup}=0$, then the result is obvious.
  Otherwise, since the supremum norm is power-multiplicative we may
  assume that $\|a\|_{\sup}$ is a norm from $K$ and scale to reduce to
  the case $\|a\|_{\sup}=1$.  By Theorem 7.4.1 of \cite{dejong} it
  follows that $a\in \scr{A}$ (this is where normality is used).  If
  the reduction $a_0\in \scr{A}_0 = \scr{A}/\pi\scr{A}$ vanishes at
  every closed point of $U$, then it vanishes everywhere by density,
  so $a_0^n = 0$ in $\scr{A}_0$ for some $n$, which is to say that
  $\pi|a^n$ in $\scr{A}$.  But this is impossible because by
  power-multiplicativity we have $\|a^n\|_{\sup}=1$ for all $n\geq 1$.
  Thus $a_0$ must be non-vanishing at some point of $U$. By the
  surjectivity of the specialization map we can find a point $x$
  reducing to this point.  Clearly then $|a(x)|=1$, which establishes
  the normal case.
  
  Suppose that $X$ is equidimensional of dimension $d$.  We claim that
  it follows that the special fiber $\mathfrak{X}_0$ must be
  equidimensional of dimension $d$ as well.  Indeed, inside each
  irreducible component of this special fiber we can find a nonempty
  Zariksi-open subset $V$ that does not meet any of the other
  irreducible components.  The generic fiber $V_\rig$ is an admissible
  open in $X$ and therefore has dimension $d$.  It follows that $V$
  has dimension $d$, and the claim follows.
  
  Let $f:\widetilde{\mathfrak{X}}\longrightarrow \mathfrak{X}$ be the
  normalization map (meaning $\Spf$ applied to the normalization map
  on algebras) and note that this map is finite by general excellence
  considerations.  By Theorem 2.1.3 of \cite{conradirredcpnts} the
  generic fiber of this map coincides with the normalization of $X$.
  Thus $\widetilde{\mathfrak{X}}_\rig$ is also equidimensional of
  dimension $d$ and the argument above shows that
  $\widetilde{\mathfrak{X}}_0$ is equidimensional of dimension $d$ as
  well.  Now since $f$ is finite it follows that $f_0$ carries
  generic points to generic points.  In particular we see that
  $f_0^{-1}(U)$ is Zariski-dense in $\widetilde{\mathfrak{X}}_0$.
  Thus by the normal case proven above there exists $x\in
  \widetilde{\mathfrak{X}}_\rig$ reducing to $f_0^{-1}(U)$ at which
  $a$ (thought of as an element of
  $\widetilde{\scr{A}}\otimes_{\OO_K}K$) attains its supremum norm.
  But then $f(x)$ is a point in $X$ reducing to $U$ with the same
  property, since the supremum norm of $a$ is the same thought of on
  $X$ or on $\widetilde{X}$ (since $\widetilde{X}\longrightarrow X$ is
  surjective).
\end{proof}

\begin{rema}
  Note that the proof in the normal case did not use the
  equidimensionality hypothesis.  This hypothesis may not be required
  in the general case, but the above proof breaks down without it
  since it is not clear how to control the special fiber under
  normalization in general, especially if $\mathfrak{X}_0$ is
  non-reduced (as is often the case for us).
\end{rema}

\begin{coro}\label{genmaxprinc}
  Let $\mathfrak{X}$ be a reduced quasi-compact admissible formal
  scheme over $\OO_K$, let $U\subseteq \mathfrak{X}_0$ be a
  Zariski-dense open, and let $\mathfrak{L}$ be an invertible sheaf on
  $\mathfrak{X}$.  Assume that $\mathfrak{X}_{\rig}$ is
  equidimensional.  Then, for any $f\in
  H^0(\mathfrak{X}_\rig,\mathfrak{L}_\rig)$ we have
  $$\|f\|_{\mathfrak{X}_\rig} = \sup_{x\in \spe^{-1}(U)}|x^*f|_x =
  \max_{x\in \spe^{-1}(U)}|x^*f|_x.$$   
\end{coro}
\begin{proof}
  Cover $\mathfrak{X}$ be a finite collection of admissible formal
  affine opens trivializing $\mathfrak{L}$ and apply Lemma
  \ref{supmaxprinc} on each such affine separately.
\end{proof}

The invertible sheaves whose sections we will be taking norms of in
this paper will all be of the form $\OO_X(D)$ for some divisor $D$ on
$X=X_1(N)_K$ or $X_1(N,n)_K$ supported on the cusps.  In the end, the
main consequence of the previous Corollary (namely, Lemma
\ref{ignorecusps}) will be that these norms are equal to the supremum
norm of the restriction of the section in question to the complement
of the residue disks around the cusps (where it is simply an analytic
function).  We feel that it is worthwhile to give more
natural definitions using the above norm machinery in the cases that it
applies to (those where we have nice moduli schemes to work with) in
the hopes that the techniques used and the above Corollary will be
useful in other similar situations.  The reader who is content with
this equivalent ``ad hoc'' definition (that is, the supremum norm on
the complement of the residue disks around the cusps) can skip to
Section \ref{weightspace} and ignore the appendix all together.

In order to endow spaces of sections of a line bundle as in the
previous paragraph with norms using the techniques above, we need
formal models of the spaces $X$ and sheaves $\OO(D)$.  For technical
reasons (involving regularity of certain moduli stacks) we are forced
to work over $\Z_p$ in going about this. The formal models over $\OO_K$
will then be obtained by extension of scalars.  The general procedure
for obtaining formal models over $\Z_p$ goes as follows.  Let $X$
denote one the stacks $X_1(N)$ or $X_1(N,n)$ over $\Z_p$ and assume
that the generic fiber $X_{\Q_p}$ is a scheme.  Let $D$ be a divisor
on $X_{\Q_p}$ that is supported on the cusps.  If the closure
$\overline{D}$ of $D$ in $X$ lies in the maximal open subscheme $X^{sch}$ of
$X$ and this subscheme is moreover regular along $\overline{D}$, then
this closure is Cartier and we may associate to it the invertible
sheaf $\OO(\overline{D})$ on $X^{sch}$.  Let $(X^{sch})^{\widehat{}}$
and $\OO(\overline{D})^{\widehat{}}$ denote the formal completions of
these objects along the special fiber.

In case $X=X_1(N)$ or $X_1(N,n)$ with $p\nmid n$, assume that $N$ has
a divisor that is prime to $p$ and at least $5$.  Then $X^{sch} =X$ by
Theorem 4.2.1 of \cite{conradamgec}, and $X$ is moreover regular (at
least over $\Z_{(p)}$) by Theorem 4.1.1 of \cite{conradamgec}.  That
passage to $\Z_p$ preserves regularity follows by excellence
considerations from the fact that $\Z_{(p)}\longrightarrow \Z_p$ is
geometrically regular.  Strictly speaking the results of
\cite{conradamgec} do not apply to $X_1(N,n)$ as stated, but since
$p\nmid n$ the proofs of these results are still valid over
$\Z_{(p)}$, as is observed in the appendix.  Since $X$ is proper over
$\Z_p$, we have $\widehat{X}_{\rig} = X^\an_{\Q_p}$ (the
analytification of the algebraic generic fiber of $X$) and hence we
have a formal model $(\widehat{X},\OO(\overline{D})^{\widehat{}}\ )$ of
$(X^\an_{\Q_p},\OO(D))$.

Now suppose that $X = X_1(Mp,p^2)$ for an integer $M\geq 5$ prime to
$p$.  Let $D$ be any divisor supported on the cusps in the connected
component $X_1(Mp,p^2)^\an_{\geq 1}$ of the ordinary locus.  By
Theorem \ref{app2} of the appendix, the closure $\overline{D}$ of $D$
in $X$ lies in $X^{sch}$ and is Cartier.  Thus we obtain a formal
model $((X^{sch})^{\widehat{}},\OO(\overline{D})^{\widehat{}}\ )$ of
$((X^{sch})^{\widehat{}}_\rig,\OO(D))$.  Observe that, by Lemma
\ref{scheme} and the comments that follow it, $X^{sch}$ is simply the
complement of a finite collection of cusps on the characteristic $p$
fiber (namely, the ones with nontrivial automorphisms).  It follows
that the open immersion
\begin{equation}\label{immersion}
(X^{sch})^{\widehat{}}_\rig\hookrightarrow (X^{sch}_{\Q_p})^\an\cong
X_{\Q_p}^\an
\end{equation}
identifies the Raynaud generic fiber on the left with the complement
of the residue disks around the cusps in the analytification on the
right that reduce to the missing points in characteristic $p$.  Thus
(\ref{immersion}) is an isomorphism when restricted to any connected
component of the locus defined by $v(E)\leq r$ that contains
no such cusps.  In particular, it is an isomorphism when restricted to 
$X_1(Mp,p^2)^\an_{\geq p^{-r}}$ by Theorem \ref{app2} of the
appendix.  

Given a complete discretely-valued extension $K/\Q_p$, we may extend
scalars on the formal models of $\OO(D)$ we have obtained to arrive at
norms on the following spaces.
\begin{itemize}
\item sections of $\OO(D)$ over any admissible open $\scr{U}$ in
  $X=X_1(N)^\an_K$ (resp. $X_1(N,n)^\an_K$ with $p\nmid n$), where $D$
  is (the scalar extension of) a divisor on $X_1(N)_{\Q_p}$
  (resp. $X_1(N,n)_{\Q_p}$) and $N$ is divisible by an integer
  that is prime to $p$ and at least $5$
\item sections of $\OO(D)$ over any admissible open $\scr{U}$ in
  $X=X_1(Mp,p^2)^\an_{\geq p^{-r}}$, where $D$ is (the scalar
  extension of) a divisor supported on the cusps in
  $X_1(Mp,p^2)^\an_{\Q_p}$ and $M$ is an integer that is prime to $p$
  and at least $5$
\end{itemize}

\begin{lemm}\label{ignorecusps}
  Let $X$, $D$, and $\scr{U}$ be as in either of the two cases above
  and assume that $\scr{U}$ contains every component of the ordinary
  locus that it meets.  Let $\scr{U}'$ denote the complement of the
  residue disks around the cusps in $\scr{U}$.  Then for any $f\in
  H^0(\scr{U},\OO(D))$ we have $$\|f\|_{\scr{U}} = \| f|_{\scr{U}'}
  \|_{\sup}.$$
\end{lemm}
\begin{proof}
  We will treat the case of $X=X_1(N)^\an_K$; the other cases are
  proven in exactly the same manner.  First note that, since points in
  $\scr{U}'$ reduce to points outside of the support of
  $\overline{D}$, the claim is equivalent to the assertion that
  $$\|f\|_{\scr{U}} = \|f|_{\scr{U}'}\|_{\scr{U}'}.$$ That is, the
  norm on $\scr{U}'$ that we have defined using formal models happens
  to be equal to the supremum norm on $\scr{U}'$.
  
  Note that the supersingular loci of $\scr{U}$ and $\scr{U}'$
  coincide, so the contributions to the above norms over this locus
  are equal, and it suffices to check the assertion upon restriction
  to the ordinary locus.  By assumption, the ordinary locus in
  $\scr{U}$ is a finite union of connected components of the ordinary
  locus in $X_1(N)^\an_K$.  Each such component corresponds via
  reduction to an irreducible component of the special fiber.  Let
  $\mathfrak{X}$ denote the admissible formal open in
  $X_1(N)^{\widehat{}}$ given by the union of the components so
  obtained with the supersingular points removed.  Then
  $\mathfrak{X}_\rig$ is precisely the ordinary locus in $\scr{U}$,
  and the result now follows from Corollary \ref{genmaxprinc} with $U$
  equal to the complement of the cusps in $\mathfrak{X}_0$.
\end{proof}

\begin{rema}
  There remain some curves on which we will need to have norms for
  sections of $\OO(D)$ but to which the norm machinery as set up here does
  not apply.  Namely, for $p\neq 2$ we have the curves
  $X_1(4p^m)^\an_K$ and $X_1(4p^m,p^2)^\an_K$, while for $p=2$ we have
  $X_1(2^{m+1}N)^\an_K$ and $X_1(2^{m+1}N,4)^\an_K$, where $m\geq 1 $
  and $N\in \{1,3\}$.  The previous lemma suggests an \emph{ad hoc}
  workaround to this problem.  In case we are working with sections of
  $\OO(D)$ for a cuspidal divisor on one of these curves, we simply
  \emph{define} the norm to be the supremum norm of the restriction of
  our section to the complement of the residue disks about the
  cusps. A more natural definition would likely result from
  considerations of ``formal stacks,'' but this norm would
  surely turn out to be equal to ours by an analog of Lemma
  \ref{ignorecusps}.
\end{rema}

\subsection{Weight space}\label{weightspace}  Throughout most of this
paper $\scr{W}$ will denote $p$-adic weight space (everywhere except
for the beginning of Section \ref{sec:eigencurve} where it is allowed
to be a general reduced rigid space for the purpose of reviewing a
general construction).  That is, $\scr{W}$ is a rigid space over
$\Q_p$ whose points with values in an extension $K/\Q_p$ are
$$\scr{W}(K) = \Hom_{\small\rm cont}(\Z_p^\times ,K^\times).$$ Define
$\q=p$ if $p\neq 2$ and $\q=4$ if $p=2$.  Let
$$\tau:\Z_p^\times\longrightarrow (\Z/\q\Z)^\times \longrightarrow
\Q_p^\times$$ denote reduction composed with the Teichmuller lifting,
and let $\ip{x} = x/\tau(x)\in 1+\q\Z_p$.  For a weight $\kappa$ we
have $$\kappa(x) = \kappa(\ip{x})\kappa(\tau(x)) =
\kappa(\ip{x})\tau(x)^i$$ for a unique integer $i$ with $0\leq i
<\varphi(\q)$ (where $\varphi$ denotes Euler's function).  Moreover,
this breaks up the space $\scr{W}$ as the admissible disjoint union of
$\varphi(\q)$ admissible opens $\scr{W}^i$, each of which is
isomorphic to a one-dimensional open ball.

For each positive integer $n$, let $\scr{W}_n$ denote the admissible
open subspace of $\scr{W}$ whose points are those $\kappa$ with
$$|\kappa(1+\q)^{p^{n-1}} - 1| \leq |\q|.$$
Then $\scr{W}_n^i :=
\scr{W}^i\cap \scr{W}_n$ is an affinoid disk in $\scr{W}^i$ and the
$\{\scr{W}^i_n\}_n$ form a nested admissible cover of $\scr{W}^i$.

To each integer $\lambda$ we may associate the weight $x\mapsto
x^\lambda$.  This weight, which by abuse of notation we simply refer
to as $\lambda$, lies in $\scr{W}^i$ for the unique $i\equiv
\lambda\pmod{\varphi(\q)}$.  Also, if $\lambda$ is an integer and
$\psi:(\Z/\q p^{n-1}\Z)^\times\longrightarrow \C_p^\times$ is a
character, then $x\mapsto x^\lambda\psi(x)$ is a point in $\scr{W}$
(with values in $\Q_p(\mu_{p^{n-1}})$) which lies in $\scr{W}_n$, as
standard estimates for $|\zeta-1|$ for roots of unity $\zeta$
demonstrate.

\section{Some modular functions}\label{sec:somefunctions}

Our definition of the spaces of half-integral weight modular forms
will follow the general approach of \cite{colmaz} (in the integral
weight $p$-adic situation) and \cite{mfhi} (in the half-integral
weight situation).  The motivating idea behind this approach is to
reduce to weight zero by dividing by a well-understood form of the
same weight.  For example, if $f$ is a half-integral weight $p$-adic
modular form of weight $k/2$, $\theta$ is the usual Jacobi theta
function of weight $1/2$, and $E_\lambda$ is the weight
$\lambda=(k-1)/2$ Eisenstein series introduced below, then
$f/(E_{\lambda}\theta)$ should certainly be a meromorphic modular
function of weight zero.  As we have no working notion of
``half-integral weight $p$-adic modular form'' we simply use the
weight zero forms so obtained as the \emph{definition} of this notion.
One must of course work out issues such as exactly what kind of poles
are introduced, how dividing by $\theta E_{\lambda}$ affects the
nebentypus character, and how to translate the classical Hecke action
into an action on these new forms.  The precise definition will be
given in the next section.

We remark that this was carried out by the author in \cite{mfhi} by
dividing by $\theta^k$ instead of $\theta E_{\lambda}$.  That approach
had the disadvantage of limiting us to \emph{classical} weights $k/2$,
whereas the current approach will work for more general $p$-adic
weights (and indeed, for families of modular forms) since
$E_{\lambda}$ interpolates nicely in the variable $\lambda$. 
  
This technique of division to reduce to weight zero in order to define
modular forms forces us to modify the usual construction of the Hecke
operators using the Hecke correspondences on the curve $X_1(N)$ by
multiplying by certain functions on the source spaces of these
correspondences.  Our first task is to define these functions and to
establish their overconvergence properties.  Since we are dividing by
$E_\lambda\theta$ to reduce to weight zero, we will require, for each
prime number $\ell$, a modular function whose $q$-expansion (at the
appropriate cusp, on the appropriate space, which depends on whether
or not $\ell=p$) is
$$\frac{E_\lambda(q_{\ell^2})\theta(q_{\ell^2})}{E_\lambda(q)\theta(q)}.$$
Factoring this into its Eisenstein part and theta part we split the
problem into two problems, the first of which is nearly done in the
integral-weight literature (see
\cite{buzzardeigenvarieties},\cite{colmaz}), and the second of which
is done in an earlier paper of the author (\cite{mfhi}).  We briefly
review both here.  See the aforementioned references for additional
details.  Note that all analytic spaces in this section are taken over
$\Q_p$.

Let $\mathbf{c}$ denote the cusp on $X_1(4)_{\Q}$ corresponding to the
point $\zeta_4q_2$ of order $4$ on the Tate curve.
Define a $\Q$-divisor $\Sigma_{4N}$ on the curve $X_1(4N)_{\Q}$
by
$$\Sigma_4 := \frac{1}{4}\pi^*[\mathbf{c}]$$ where
$$\pi:X_1(4N)_\Q\longrightarrow X_1(4)_\Q$$
is the obvious degeneracy
map.  This divisor is set up to look like the divisor of zeros of the
pullback of the Jacobi theta function $\theta$ to $X_1(4N)_{\Q}$ and
will later be used to control poles introduced in dividing by
$E_\lambda\theta$.

In \cite{mfhi} we defined a rational function $\Theta_{\ell^2}$ on
$X_1(4,\ell^2)_{\Q}$ with divisor
$$\div(\Theta_{\ell^2}) = \pi_2^*\Sigma_4 - \pi_1^*\Sigma_4$$
such
that
$$\Theta_{\ell^2}(\Tate(q),\zeta_4,\ip{q_{\ell^2}}) =
\frac{\sum_{n\in\Z}q_{\ell^2}^{n^2}} {\sum_{n\in\Z} q^{n^2}} =
\frac{\theta(q_{\ell^2})} {\theta(q)}.$$ Here $\pi_1$ and $\pi_2$ are
the maps comprising the $\ell^2$ Hecke correspondence on $X_1(4)$ and
are defined in Section \ref{heckefixed}.  Strictly speaking, we had
assumed $\ell\neq 2$ in the arguments in \cite{mfhi}, but if one is
only interested in the result above, then one can easily check that
the arguments work for $\ell=2$ verbatim.

Let us now turn to the Eisenstein part of the above functions.  For
further details and proofs of the assertions in this paragraph, we
refer the reader to Sections 6 and 7 of \cite{buzzardeigenvarieties}.
Let
$$E(q) := 1 +\frac{2}{\zeta_p(\kappa)}\sum_n \left(\sum_{d|n\ \!,\ \!
  p\nmid d}\kappa(d)d^{-1}\right)q^n \in \scr{O}(\scr{W}^0)[\![q]\!]$$
be the $q$-expansion of the $p$-deprived Eisenstein family over
$\scr{W}^0$.  Note that there are no problems with zeros of $\zeta_p$
since we are restricting our attention to $\scr{W}^0$.  For a
particular choice of $\kappa\in\scr{W}^0$, we denote by $E_\kappa(q)$
the expansion obtained by evaluating all of the coefficients at
$\kappa$.  In particular, for a positive integer $\lambda\geq 2$
divisible by $\varphi(\q)$, $E_\lambda(q)$ is the $q$-expansion of the
usual $p$-deprived classical Eisenstein series of weight $\lambda$ and
level $p$.

Let $\ell$ be a prime number.  If $\ell\neq p$, then there exists a
rigid analytic function $\E_\ell$ on $X_0(p\ell)^\an_{\geq
  1}\times\scr{W}^0$ whose $q$-expansion at $(\Tate(q),\mu_{p\ell})$
is $E(q)/E(q^\ell)$.  If $\ell=p$, then the same holds with
$X_0(p\ell)^\an_{\geq 1}$ replaced by $X_0(p)^\an_{\geq 1}$ and
$\mu_{p\ell}$ replaced by $\mu_p$.  In \cite{buzzardeigenvarieties} it
is shown that there exists a sequence of rational numbers
$$\frac{1}{p+1}>r_1\geq r_2\geq\cdots\geq r_n\geq\cdots >0$$ with
$r_i< p^{2-i}/\q(1+p)$ such that, when restricted to
$X_0(p\ell)^\an_{\geq 1}\times\scr{W}^0_n$ (respectively,
$X_0(p)^\an_{\geq 1}\times\scr{W}^0_n$ if $\ell=p$), $\E_\ell$
analytically continues to an invertible function on
$X_0(p\ell)^\an_{\geq p^{-r_n}}\times \scr{W}^0_n$ (respectively,
$X_0(p)^\an_{\geq p^{-r_n}}\times \scr{W}^0_n$ if $\ell=p$).  Fix such
a sequence once and for all.  Let us first extend these results to
square level.
\begin{lemm}\label{lemma1}
  Let $\ell\neq p$ be a prime number.  There exists an invertible
  function $\E_{\ell^2}$ on $X_0(p\ell^2)^\an_{\geq 1}\times\scr{W}^0$
  whose $q$-expansion at $(\Tate(q),\mu_{p\ell^2})$
  is
  $E(q)/E(q^{\ell^2})$.  Moreover, the function $\E_{\ell^2}$, when
  restricted to $\scr{W}^0_n$, analytically continues to an invertible
  function on $X_0(p\ell^2)^\an_{\geq p^{-r_n}}\times\scr{W}^0_n$.
  
  There exists an invertible function $\E_{p^2}$ on $X_0(p)^\an_{\geq
    1}\times \scr{W}^0$ whose $q$-expansion at $(\Tate(q),\mu_p)$ is
  $E(q)/E(q^{p^2})$.  Moreover, the function $\E_{p^2}$, when
  restricted to $\scr{W}^0_n$, analytically continues to an invertible
  function on $X_0(p)^\an_{\geq p^{-r_n/p}}\times \scr{W}^0_n$.
\end{lemm}
\begin{proof}
Let $\ell$ be a prime different from $p$. There are two natural maps
$$X_0(p\ell^2)^\an_{\Q_p}\longrightarrow X_0(p\ell)^\an_{\Q_p},$$
namely those
given on noncuspidal points by
\begin{eqnarray*}
  (E,C) & \stackrel{d_{\ell,1}}{\longmapsto} & (E,\ell C)\\
  (E,C) & \stackrel{d_{\ell,2}}{\longmapsto} & (E/p\ell C,C/p\ell C)
\end{eqnarray*}
Both of these restrict to maps
$$d_{\ell,1},d_{\ell,2}: X_0(p\ell^2)^\an_{\geq
  p^{-r_n}}\longrightarrow X_0(p\ell)^\an_{\geq p^{-r_n}}.$$
We define
$\E_{\ell^2}$ to be the invertible function
\begin{equation}\label{Edefinition}
\E_{\ell^2} :=
  d_{\ell,1}^*\E_\ell \cdot d_{\ell,2}^*\E_\ell \in
  \OO(X_0(p\ell^2)^\an_{\geq p^{-r_n}}\times \scr{W}^0_n)^\times.
\end{equation}
The $q$-expansion of $\E_{\ell^2}$ at $(\Tate(q),\mu_{p\ell^2})$ is 
\begin{eqnarray*}
\lefteqn{ \E_\ell(d_{\ell,1}(\Tate(q),\mu_{p\ell^2}))
\E_\ell(d_{\ell,2}(\Tate(q),\mu_{p\ell^2}))  } \\ &=&
\E_\ell(\Tate(q),\mu_{p\ell})
\E_\ell(\Tate(q)/\mu_\ell,\mu_{p\ell^2}/\mu_\ell)  \\ &=&
\E_\ell(\Tate(q),\mu_{p\ell}) \E_\ell(\Tate(q^\ell),\mu_{p\ell}) \\
&=&
\frac{E(q)}{E(q^\ell)}\frac{E(q^\ell)}{E(q^{\ell^2})} =
\frac{E(q)}{E(q^{\ell^2})} 
\end{eqnarray*}

One must take additional care if $\ell=p$.   Then
there is a well-defined map
\begin{eqnarray*}
d: X_0(p)^\an_{\geq p^{-r_n/p}} &\longrightarrow& X_0(p)^\an_{\geq
  p^{-r_n}} \\ (E,C) & \longmapsto & (E/C,H_{p^2}/C)
\end{eqnarray*}
where $H_{p^2}$ is the canonical subgroup of $E$ of order $p^2$.  This
follows form the fact that $X_0(p)^\an_{\geq p^{-r_n/p}}$ consists of
pairs $(E,C)$ with $C$ equal to the canonical subgroup of $E$ of order
$p$, and standard facts about quotienting by such subgroups (see for
example Theorem 3.3 of \cite{buzzard}).  We define an invertible
function by $$\E_{p^2} := \E_p \cdot d^*\E_p\in \OO(X_0(p)^\an_{\geq
  p^{-r_n/p}}\times \scr{W}^0_n)^\times$$
where we have implicitly
restricted $\E_p$ to
$$X_0(p)^\an_{\geq p^{-r_n/p}}\times \scr{W}^0_n\subseteq X_0(p)^\an_{\geq
  p^{-r_n}}\times \scr{W}^0_n.$$
The $q$-expansion of $\E_{p^2}$ at $(\Tate(q),\mu_p)$ is
\begin{eqnarray*}
  \E_p(\Tate(q),\mu_p)\E_p(d(\Tate(q),\mu_p)) &=&
  \E_p(\Tate(q),\mu_p)\E_p(\Tate(q)/\mu_p,\mu_{p^2}/\mu_p) \\ &=&
  \E_p(\Tate(q),\mu_p)\E_p(\Tate(q^p),\mu_p) \\ &=&
  \frac{E(q)}{E(q^p)} \frac{E(q^p)}{E(q^{p^2})} = \frac{E(q)} {E(q^{p^2})}
\end{eqnarray*}
\end{proof}

Let $$\pi: X_1(p,\ell^2)^\an_{\Q_p}\longrightarrow \left\{\begin{array}{ll}
    X_0(p\ell^2)^\an_{\Q_p} & \ell\neq p \\ X_0(p)^\an_{\Q_p} &
    \ell=p\end{array}\right.$$
denote the map given on noncuspidal
points by $$(E,P,C) \longmapsto \left\{
  \begin{array}{ll} (E/C,(\ip{P}+E[\ell^2])/C) & \ell\neq p \\
    (E/C,\ip{P}/C) & \ell=p\end{array}\right.$$   Note that we have
\begin{equation}\label{domains}
\pi(\Tate(q),\zeta_p, \ip{q_{\ell^2}}) =
\left\{\begin{array}{ll} (\Tate(q_{\ell^2}),\mu_{p\ell^2})  &
    \ell\neq p \\ (\Tate(q_{p^2}),\mu_p) & \ell= p\end{array}\right.
\end{equation}
This observation suggests that perhaps the components
$X_1(p,\ell^2)^\an_{\geq p^{-r}}$ should be related to (via $\pi$) the
components $X_0(p\ell^2)^\an_{\geq p^{-r}}$.
\begin{lemm}
  If $\ell\neq p$, then the map $\pi$ restricts to $$\pi:
  X_1(p,\ell^2)^\an_{\geq p^{-r}}\longrightarrow
  X_0(p\ell^2)^\an_{\geq p^{-r}}$$
  for all $r< p/(1+p)$.

 In case $\ell=p$, the map $\pi$
  restricts to $$X_1(p,p^2)^\an_{\geq p^{-p^2r}}\longrightarrow
  X_0(p)^\an_{\geq p^{-r}}$$ for all $r<1/p(1+p)$.
\end{lemm}
\begin{proof}
  First suppose $\ell\neq p$.  Let $\scr{U}$ denote the entirety of
  the locus in $X_0(p\ell^2)^\an_{\Q_p}$ defined by $v(E)\leq r$.
  First note that, since quotienting by a subgroup of order prime to
  $p$ does not change its measure of singularity, the map $\pi$
  restricts to a map
  $$X_1(p,\ell^2)^\an_{\geq p^{-r}} \longrightarrow \scr{U}.$$
  The inverse
  images of the two connected components of $\scr{U}$ under this map are
  disjoint admissible opens that admissibly cover a connected space,
  and $\pi^{-1}(X_0(p\ell^2)^\an_{\geq p^{-r}})$ is nonempty by
  (\ref{domains}), so this must be all of $X_1(p,\ell^2)^\an_{\geq
    p^{-r}}$, and the result follows.
  
  Now suppose that $\ell=p$.  Let $\scr{U}$ denote the entirety of the
  locus in $X_0(p)^\an_{\Q_p}$ defined by $v(E)\leq r$.  Once we
  verify that $\pi$ restricts to $$X_1(p,p^2)^\an_{\geq
    p^{-p^2r}}\longrightarrow \scr{U},$$ the argument may proceed
  exactly as above.  We claim, moreover, that if $(E,P,C)$ is a point
  in $X_0(p,p^2)^\an_{\geq p^{-p^2r}}$, then $v(E/C) = v(E)/p^2$.
  This would follow if we knew that $C$ met the canonical subgroup of
  $E$ trivially (again by standard facts about quotienting by
  canonical and non-canonical subgroups of order $p$, as in Section 3
  of \cite{buzzard}), so it suffices to prove that $\ip{P}$ \emph{is}
  the canonical subgroup of $E$.

  The natural map 
  \begin{eqnarray*}
    X_1(p,p^2) & \longrightarrow & X_0(p) \\
    (E,P,C) &\longmapsto & (E,\ip{P})
  \end{eqnarray*}
  restricts to $$X_1(p,p^2)^\an_{\geq p^{-r}}\longrightarrow
  X_0(p)^\an_{\geq p^{-r}}$$  by the same connectivity argument used
  in the $\ell\neq p$ case (since this map clearly doesn't change
  $v(E)$).  But the locus $X_0(p)^\an_{\geq p^{-r}}$ is
  well-known to consist of pairs $(E,C)$ with $C$ equal to the
  canonical subgroup of $E$. 
\end{proof}

We may pull back the Eisenstein family of Lemma \ref{lemma1} for
$\ell\neq p$ through the map $\pi$ to arrive at an invertible function
on $X_1(p,\ell^2)^\an_{\geq p^{-r_n}}\times\scr{W}^0_n$.  By the
previous lemma, we may also pull back the family for $\ell=p$ through
$\pi$ to arrive at an invertible function on $X_0(p,p^2)^\an_{\geq
  p^{-pr_n}}\times \scr{W}^0_n$.  For any $\ell$, it follows from
(\ref{domains}) that the function $\pi^*\E_{\ell^2}$ satisfies
$$\pi^*\E_{\ell^2}(\Tate(q),\zeta_p,\ip{q_{\ell^2}}) =
\frac{E(q_{\ell^2})} {E((q_{\ell^2})^{\ell^2})} =
\frac{E(q_{\ell^2})} {E(q)}.$$

To arrive at the functions that we need, we simply multiply
$\pi^*\E_{\ell^2}$ and $\Theta_{\ell^2}$ (which is constant in the
weight).  Of course, to do so we must first pull these functions back
so that they lie on a common curve.  The natural (``smallest'') curve
to use depends on whether or not $p=2$, since $2$ already lies in the
$\Gamma_1$ part of the level of $\Theta_{\ell^2}$.  The following
proposition summarizes the properties of the resulting functions.

\begin{prop}\label{twistingfunction}
  Let $p$ be and $\ell$ be primes.  There exists an element
  $\H_{\ell^2}$ of
$$\left\{\!\begin{array}{ll} H^0(X_1(4p,\ell^2)^\an_{\geq
    1}\times\scr{W}^0,\scr{O}(\pi_1^*\Sigma_{4p} -
  \pi_2^*\Sigma_{4p})) & p\neq 2 \\ H^0(X_1(4,\ell^2)^\an_{\geq
    1}\times\scr{W}^0,\scr{O}(\pi_1^*\Sigma_{4} - \pi_2^*\Sigma_{4}))
  & p=2 \end{array}\right.$$ whose $q$-expansion at
  $$\left\{\!\begin{array}{ll} (\Tate(q),\mu_{4p},\ip{q_{\ell^2}})) &
  p\neq 2 \\ (\Tate(q),\mu_{4},\ip{q_{\ell^2}})) & p=2
  \end{array}\right.$$
  is equal to
  $$\frac{E(q_{\ell^2})\theta(q_{\ell^2})} {E(q)\theta(q)}.$$
  Moreover, there exists a sequence of rational numbers $r_n$ such
  that $$\frac{1}{1+p}> r_1 \geq r_2 \geq \cdots > 0$$
  with $r_i <
  p^{2-i}/\q(1+p)$ such that $\H_{\ell^2}$, when restricted to
  $\scr{W}^0_n$, analytically continues to the region
  $$\left\{\!\begin{array}{ll} X_1(4p,\ell^2)^\an_{\geq
    p^{-r_n}}\times\scr{W}^0_n & p\neq 2, \ell\neq
  p\\ X_1(4p,p^2)^\an_{\geq p^{-pr_n}}\times\scr{W}^0_n &
  p\neq 2, \ell=p\\ X_1(4,\ell^2)^\an_{\geq 2^{-r_n}}\times\scr{W}^0_n &
  p=2,\ell\neq 2\\ X_1(4,4)^\an_{\geq 2^{-2r_n}}\times\scr{W}^0_n &
  p=\ell=2
\end{array}\right.$$
\end{prop}

Finally, we wish to extend $\H_{\ell^2}$ and $E(q)$ to all of
$\scr{W}$.  To do this, we simply pull back through the natural map
\begin{eqnarray}\label{wmap}
 \scr{W} & \longrightarrow & \scr{W}^0 \\ \nonumber \kappa & \longmapsto &
 \kappa\circ\ip{} 
\end{eqnarray}
When restricted to $\scr{W}^i$, this map is simply the isomorphism
$\kappa \mapsto \kappa/\tau^i$.

\begin{rema}\label{diamondtrivial}
  We have chosen in the end to use $\Gamma_1$-structure on the curves
  on which the $\mathbf{H}_{\ell^2}$ lie both to rigidify the
  associated moduli problems over $\Q_p$ as well as because these are
  the curves that will actually turn up in the sequel.  We note,
  however, that the $\mathbf{H}_{\ell^2}$ are invariant under all
  diamond automorphisms.
\end{rema}

\section{The spaces of forms}

In this section we define spaces of overconvergent $p$-adic modular
forms as well as families thereof over admissible open subsets of
$\scr{W}$.  Again, the motivating idea behind these definitions is
that we have reduced to weight $0$ via division by the well-understood
forms $E_{\lambda}\theta$.  By ``well-understood'' we essentially mean
two things here.  The first is that we understand their zeros once we
eliminate part of the supersingular locus (and thereby \emph{remove}
the zeros of the Eisenstein part).  The second is that, by the
previous section, we know that there are modular functions with
$q$-expansions
$$\frac{E_\lambda(q_{\ell^2})\theta(q_{\ell^2})}
{E_\lambda(q)\theta(q)}$$
that interpolate rigid-analytically in
$\lambda$, a fact that we will need to define Hecke operators on
families in the next section.

Before defining the spaces of forms, we need to make a couple of
remarks about diamond automorphisms.  For a positive integer $N$ and
an element $d\in(\Z/N\Z)^\times$, let $\ip{d}$ denote the usual
diamond automorphism of $X_1(N)$ given on (noncuspidal) points by
$(E,P)\mapsto (E,dP)$.  Now suppose we are given a factorization 
$N=N_1N_2$ into relatively prime factors, so the natural reduction map
$$(\Z/N\Z)^\times \stackrel{\sim}{\longrightarrow} (\Z/N_1\Z)^\times
\times (\Z/N_2\Z)^\times$$
is an isomorphism.  For
$a\in(\Z/N_1\Z)^\times$ and $b\in(\Z/N_2\Z)^\times$ we let
$(a,b)\in(\Z/N\Z)^\times$ denote the inverse image of the pair $(a,b)$
under the this map.  For $a\in(\Z/N_1\Z)^\times$, we define
$\ip{a}_{N_1} := \ip{(a,1)}$, and we refer to these automorphisms as
\emph{the diamond automorphisms at $N_1$}.  The diamond automorphisms
at $N_2$ are defined similarly, and we have a factorization
$$\ip{d} = \ip{d}_{N_1}\circ \ip{d}_{N_2}.$$
Finally, we observe that
the diamond operators on $X_1(4N)_K^\an$ preserve the subspaces
$X_1(4N)^\an_{\geq p^{-r}}$ and the divisor $\Sigma_{4N}$ in the
sense that $\ip{d}^{-1}(X_1(4N)^\an_{\geq p^{-r}}) = X_1(4N)^\an_{\geq
  p^{-r}}$ and $\ip{d}^*\Sigma_{4N}=\Sigma_{4N}$, respectively.

\begin{conv}
By the symbol $\OO(\Sigma)$ for a $\Q$-divisor $\Sigma$ we shall always
mean $\OO(\lfloor \Sigma\rfloor)$, where $\lfloor \Sigma\rfloor$ is the
divisor obtained by taking the floor of each coefficient occurring in
$\Sigma$.  
\end{conv}

First we define the spaces of forms of fixed weight.  Let $N$ be a
positive integer and suppose that either $p\nmid 4N$ or that $p=2$ and
$p\nmid N$.
\begin{defi}
  Let $\kappa \in\scr{W}^i(K)$ and pick $n$ such that
  $\kappa\in\scr{W}^i_n$.  Then, for any rational number $r$ with
  $0\leq r\leq r_n$, we define the space of $p$-adic half-integral
  weight modular forms of weight $\kappa$, tame level $4N$ (or rather
  $N$ if $p=2$) , and growth condition $p^{-r}$ over $K$ to be
  $$ \widetilde{M}_\kappa(4N,K,p^{-r}) := \left\{\!\begin{array}{ll}
  H^0(X_1(4Np)^\an_{\geq
    p^{-r}},\OO(\Sigma_{4Np}))^{\tau^i}\times\{\kappa\} & p\neq 2
  \\ H^0(X_1(4N)^\an_{\geq
    2^{-r}},\OO(\Sigma_{4N}))^{(-1/\cdot)^i\tau^i}\times\{\kappa\} & p=2
  \end{array}\right.$$
where $()^{\tau^i}$ denotes the $\tau^i$ eigenspace for the action of
the diamond automorphisms at $p$, and similarly for
$(-1/\cdot)^i\tau^i$ if $p=2$.
\end{defi}

\begin{rems}\label{remsfixedwt}\
  \begin{itemize}
  \item For $p\neq 2$, we have chosen to remove $p$ from the level and
    only indicate the tame level in the notation because, as we will
    see, these spaces contain forms of all $p$-power level.  However,
    for $p=2$, we have left the $4$ in as a reminder that the forms
    have at least a $4$ in the level, as well as for some uniformity
    in notation.
  \item Note that this space has been ``tagged'' with the weight
    $\kappa$ because the actual space has only a rather trivial
    dependence on $\kappa$ ($\kappa$ serves only to restrict the
    admissible $K$ and $r$ and to determine $i$).  The point is that,
    as we will see, the Hecke action on this space is very sensitive
    to $\kappa$.  The tag will generally be ignored in what follows as
    the weight will be clear from the context.
  \item This space is endowed with a norm which is defined as in
    Subsection \ref{norms} and is a Banach space over $K$ with respect
    to this norm.
  \item  We call the forms belonging to spaces with
    $r>0$ \emph{overconvergent}.  The space of all overconvergent forms
    (of this weight and level) is the inductive limit
    $$\widetilde{M}_\kappa^\dagger(4N,K) = \lim_{r\to 0}
    \widetilde{M}_\kappa(4N,K,p^{-r}).$$   
  \item In case $\kappa$ is the character associated to an integer
    $\lambda\geq 0$, the space of forms defined above would
    classically be thought of having weight $\lambda+1/2$.  Our choice
    of $p$-adic weight character book-keeping seems to be the most
    natural one (the Shimura lifting has the effect of squaring the
    weight character, for example).
  \item In case $\kappa$ is the weight associated to an integer
    $\lambda\geq 0$, then the definition here is somewhat less general
    than the definition of the space of forms of weight $\lambda+1/2$
    contained in the author's previous paper (\cite{mfhi}) due to the
    need to eliminate enough of the supersingular locus to get rid of
    the Eisenstein zeros.  The two definitions are
    (Hecke-equivariantly) isomorphic whenever they are both defined,
    as we will see in Proposition \ref{comparisonofdefs}.
  \item The tilde is an homage to the metaplectic literature and will
    be used in forthcoming work on all half-integral weight objects in
    order to distinguish them from their integral weight counterparts.
  \end{itemize}
\end{rems}

We now turn to the spaces of families of modular forms.
\begin{defi}\label{families}
  Let $X$ be a connected affinoid subdomain of $\scr{W}$.  Then
  $X\subseteq \scr{W}^i$ for some $i$ since $X$ is connected and
  moreover $X\subseteq \scr{W}^i_n$ for some $n$ since $X$ is affinoid.
  For any rational number $r$ with $0\leq r\leq r_n$, we define the
  space of families of half-integral weight modular forms of tame level $4N$
  and growth condition $p^{-r}$ on $X$ to be
  $$\widetilde{M}_X(4N,K,p^{-r}) := \left\{\begin{array}{ll}
  H^0(X_1(4Np)^\an_{\geq
    p^{-r}},\OO(\Sigma_{4Np}))^{\tau^i}\widehat{\otimes}_K \OO(X) & p\neq 2 \\
  H^0(X_1(4N)^\an_{\geq
    2^{-r}},\OO(\Sigma_{4N}))^{(-1/\cdot)^i\tau^i}\widehat{\otimes}_K 
  \OO(X) & p=2\end{array}\right.$$
\end{defi}

\begin{rems}\label{remsfamilies}\
  \begin{itemize}
  \item We endow $\widetilde{M}_X(4N,K,p^{-r})$ with the completed
    tensor product norm obtained from the norms we have defined in
    Section \ref{norms} and the supremum norm
    on $\OO(X)$.  The space $\widetilde{M}_X(4N,K,p^{-r})$ with this
    norm is a Banach module over the Banach algebra $\OO(X)$.
  \item As in the case of fixed weight, the definition depends rather
    trivially on $X$ but the Hecke action will be very sensitive to $X$.
  \item In general, if $X$ is an affinoid subdomain of $\scr{W}$, we
    define $\widetilde{M}_X$ to be the direct sum of the spaces
    corresponding to the connected components of $X$.  Also, just as
    for particular weights, we can talk about the space of all
    overconvergent families of forms on $X$, namely
    $$\widetilde{M}^\dagger_X(4N,K) = \lim_{r\to 0}
    \widetilde{M}_X(4N,K,p^{-r}).$$
  \item Using a simple projector argument, one sees easily that we have
    a canonical identification 
    \begin{eqnarray*}
      \lefteqn{H^0(X_1(4Np)^\an_{\geq
        p^{-r}},\OO(\Sigma_{4Np}))^{\tau^i}\widehat{\otimes}_K\OO(X)} && \\
        && \cong  (H^0(X_1(4Np)^\an_{\geq
        p^{-r}},\OO(\Sigma_{4Np}))\widehat{\otimes}_K\OO(X))^{\tau^i}, 
    \end{eqnarray*}
    and similarly at level $4N$ if $p=2$, a comment that will prove to
    be useful in the next section.
  \end{itemize}
\end{rems}

For each $X$ as above and each $L$-valued point $\kappa\in X$,
evaluation at $x$ induces a specialization map
$$\widetilde{M}_X(4N,K,p^{-r})\longrightarrow
\widetilde{M}_\kappa(4N,L,p^{-r}).$$
In the next section we will define
a Hecke action on both of these spaces for which such specialization
maps are equivariant and which recover the usual Hecke operators on
the right side above (in the sense that they are given by the usual
formulas on $q$-expansions).

Each of the spaces of forms that we have defined has a cuspidal
subspace consisting of forms that ``vanish at the cusps.''  This
notion is a little subtle in half-integral weight because there are
often cusps at which \emph{all} forms are \emph{forced} to vanish.  To
explain this comment and motivate the subsequent definition of the
space of cusp forms, let us go back to the motivation behind our
definitions of the spaces of forms.  If $F$ is a form of half-integral
weight in our setting, then $F\theta E$ (where $E$ is an appropriate
Eisenstein series) is what we would ``classically'' like to think of
as a half-integral weight form.  Indeed, in case $F$ is classical
(this notion is defined in Section \ref{classicality}) then $F\theta
E$ can literally be identified with a classical holomorphic modular
form of half-integral weight over $\C$.  The condition $\div(F)\geq
-\Sigma_{4Np}$ (we are assuming $p\neq 2$ for the sake of this
motivation) in our definition is exactly the condition that $F\theta
E$ be holomorphic at all cusps.  Likewise, the condition that this
inequality be strict at all cusps is the condition that $F\theta E$ be
cuspidal.  But since $\div(F)$ has integral coefficients, the
non-strict inequality \emph{implies} the strict inequality at all
cusps where $\Sigma_{4Np}$ has non-integral coefficients.

With this in mind, we are led to the following definition of cusp
forms.  For an integer $M$, let $C_{4M}$ be the divisor on
$X_1(4M)^\an_{\Q_p}$ given by the sum of the cusps at which
$\Sigma_{4M}$ has integral coefficients.  To define the cuspidal
subspace of any of the above spaces of forms, we replace the divisor
$\Sigma_{4Np}$ (resp. $\Sigma_{4N}$ if $p=2$) by the divisor
$\Sigma_{4Np} - C_{4Np}$ (resp. $\Sigma_{4N}-C_{4N}$ if $p=2$).  We
will denote the cuspidal subspaces by the letter $S$ instead of $M$.
Thus, for example, if $\kappa\in\scr{W}_n^i(K)$ and $0\leq r\leq r_n$,
we define
$$\widetilde{S}_\kappa(4N,K,p^{-r}) = \left\{\begin{array}{ll}
H^0(X_1(4Np)^\an_{\geq p^{-r}}, \scr{O}(\Sigma_{4Np} -
C_{4Np}))^{\tau^i}\times\{\kappa\} & p\neq 2 \\ H^0(X_1(4N)^\an_{\geq
  2^{-r}}, \scr{O}(\Sigma_{4N} -
C_{4N}))^{(-1/\cdot)^i\tau^i}\times\{\kappa\} &
p=2\end{array}\right.$$ Remarks \ref{remsfixedwt} and
\ref{remsfamilies} apply equally well to the corresponding spaces of
cusp forms.

\section{Hecke operators}\label{sec:hecke}

Before we construct Hecke operators, we need to make some remarks on
diamond operators and nebentypus.  Since the $p$-part of the
nebentypus character is encoded as part of the $p$-adic weight
character, we need to separate out the tame part of the diamond
action.  Fix a weight $\kappa\in\scr{W}^i(K)$.  In order to define the
tame diamond operators in a manner compatible with the classical
definitions and that in \cite{mfhi} we must twist (at least in the
case $p\neq 2$) those obtained via pull-back from the automorphism
$\ip{}_{4N}$ by $(-1/\cdot)^i$.  That is, for $F\in
\widetilde{M}_\kappa(4N,K,p^{-r})$, we define
$$\ip{d}_{4N,\kappa}F =
\left(\!\frac{-1}{d}\!\right)^i \ip{d}_{4N}^*F\ \ \mbox{if}\ \ p\neq 2$$
and
$$\ip{d}_{N,\kappa}F =\ip{d}_{N}^*F\ \ \mbox{if}\ \ p= 2$$ Without
this twist in the $p\neq 2$ case, the definition would not agree with
the classical one because of the particular nature of the automorphy
factor of the form $\theta$ used in the identification of our forms
with classical forms.  The same formulas define operators
$\ip{}_{4N,X}$ and $\ip{}_{N,X}$ on the space of families of modular
forms over $X\subseteq\scr{W}^i$.  For a more general
$X\subseteq\scr{W}$, we break into the components in $\scr{W}^i$ for
each $i$ and define $\ip{}_{4N,X}$ and $\ip{}_{N,X}$ component by
component.  For a character $\chi$ modulo $4N$ (resp. modulo $N$ if
$p=2$), we define the space of forms of tame nebentypus $\chi$ to be
the $\chi$-eigenspace of $\widetilde{M}_\kappa(4N,K,p^{-r})$ for the
operators $\ip{}_{4N,\kappa}$ (resp. $\ip{}_{N,\kappa}$ if $p=2$).  The
same definition applies to families of forms.  These subspaces are
denoted by appending a $\chi$ to the list of arguments (e.g.
$\widetilde{M}_\kappa(4N,K,p^{-r},\chi)$).

Let $\scr{X}$ and $\scr{Y}$ be rigid spaces equipped with a pair of
maps $$\pi_1,\pi_2:\scr{X}\longrightarrow \scr{Y}$$
and let $D$ be a
$\Q$-divisor on $\scr{Y}$ such that $\pi_1^*D - \pi_2^*D$ has integral
coefficients.  Let $\scr{Z}\subseteq \scr{X}$ be an admissible
affinoid open and let
$$H\in H^0(\scr{Z}, \OO(\pi_1^*D- \pi_2^*D)).$$
Let
$\scr{U},\scr{V}\subseteq \scr{Y}$ be admissible affinoid opens such
that $\pi_1^{-1}(\scr{V})\cap \scr{Z}\subseteq \pi_2^{-1}(\scr{U})\cap
\scr{Z}$, and suppose that
$$\pi_1: \pi_1^{-1}(\scr{V})\cap \scr{Z} \longrightarrow \scr{V}$$
is
finite and flat. Then there is a well-defined map
$$H^0(\scr{U},\OO(D)) \longrightarrow H^0(\scr{V},\OO(D))$$
given by the
composition
$$\xymatrix{ H^0(\scr{U},\OO(D))\ar[r]^-{\pi_2^*} &
  H^0(\pi_2^{-1}(\scr{U})\cap
  \scr{Z},\OO(\pi_2^*D))\ar[r]^{\mathrm{res}}&
  H^0(\pi_1^{-1}(\scr{V})\cap \scr{Z},\OO(\pi_2^*D)) \ar `r[d] `[l]
  `[dll]_{\cdot H} `[dl] [dl] \\ & H^0(\pi_1^{-1}(\scr{V})\cap
  \scr{Z},\OO(\pi_1^*D))\ar[r]^-{\pi_{1*}} & H^0(\scr{V},\OO(D))}$$
where $\pi_{1*}$ is the trace map corresponding to the finite and flat
map $\pi_1$.

\subsection{Hecke operators for a fixed weight}\label{heckefixed}

Let $N$ be as above, let $\ell$ be any prime number, and let
$$\pi_1,\pi_2: \left\{\begin{array}{ll} X_1(4Np,\ell^2)^\an_K
\longrightarrow X_1(4Np)^\an_K & p\neq 2 \\ X_1(4N,\ell^2)^\an_K
\longrightarrow X_1(4N)^\an_K & p=2\end{array}\right.$$ be the maps
defined on noncuspidal points of the underlying moduli problem by
\begin{eqnarray*}
  \pi_1: (E,P,C) & \longmapsto & (E,P) \\
  \pi_2: (E,P,C) & \longmapsto & (E/C,P/C)
\end{eqnarray*}

Suppose that $\ell\neq p$.  Then $$\left\{\!\begin{array}{cc}
\pi_1^{-1}(X_1(4Np)^\an_{\geq p^{-r}}) = \pi_2^{-1}(X_1(4Np)^\an_{\geq
  p^{-r}}) & p\neq 2 \\ \pi_1^{-1}(X_1(4N)^\an_{\geq 2^{-r}}) =
\pi_2^{-1}(X_1(4N)^\an_{\geq 2^{-r}}) & p= 2 \end{array}\right.$$ for
any $r<p/(1+p)$ since quotienting an elliptic curve by a subgroup of
order prime to $p$ does not change its measure of singularity.  Fix a
weight $\kappa\in\scr{W}^i(K)$ and let $\H_{\ell^2}(\kappa)$ denote
the specialization of $\H_{\ell^2}$ to $\kappa\in\scr{W}$ (which,
recall, is defined to be the specialization of $\H_{\ell^2}$ to
$\kappa/\tau^i\in\scr{W}^0$).  Pick $n$ such that
$\kappa\in\scr{W}^i_n$ and suppose $0\leq r\leq r_n$.  Applying the
general construction above with
\begin{center}
\begin{tabular}{c|c|c}
   & $p\neq 2$ & $p=2$ \\
 
\hline $\scr{X}$ & $X_1(4Np,\ell^2)^\an_K$ & $X_1(4N,\ell^2)^\an_K$ \\
$\scr{Y}$ & $X_1(4Np)^\an_K$ & $X_1(4N)^\an_K$ \\
$\scr{Z}$ & $X_1(4Np,\ell^2)^\an_{\geq p^{-r}}$ &
$X_1(4N,\ell^2)^\an_{\geq 2^{-r}}$ \\
$D$ & $\Sigma_{4Np}$ & $\Sigma_{4N}$\\ 
$H$ & $\H_{\ell^2}(\kappa)$ & $\H_{\ell^2}(\kappa)$ \\
$\scr{U} = \scr{V}$ & $X_1(4Np)^\an_{\geq p^{-r}}$ & $X_1(4N)^\an_{\geq
  2^{-r}}$ 
\end{tabular}
\end{center}
we arrive an endomorphism of the $K$-vector space
$$\left\{\!\begin{array}{cc} H^0(X_1(4Np)^\an_{\geq
  p^{-r}},\OO(\Sigma_{4Np})) & p\neq 2 \\ H^0(X_1(4N)^\an_{\geq
  2^{-r}},\OO(\Sigma_{4N})) & p=2\end{array}\right.$$ One checks
easily that since the diamond operators act trivially on $\H_{\ell^2}$
(see Remark \ref{diamondtrivial}), this endomorphism commutes with the
action of the diamond operators, and therefore induces an endomorphism
of $\widetilde{M}_\kappa(4N,K,p^{-r})$.  We define $T_{\ell^2}$ (or
$U_{\ell^2}$ if $\ell\mid 4N$) to be the quotient of this endomorphism
by $\ell^2$.

Now suppose that $\ell=p$.  Note that
$$\left\{\!\begin{array}{cc} \pi_1^{-1}(X_1(4Np)^\an_{\geq p^{-p^2r}})
\subseteq \pi_2^{-1}(X_1(4Np)^\an_{\geq p^{-r}}) & p\neq 2
\\ \pi_1^{-1}(X_1(4N)^\an_{\geq 2^{-2^2r}}) \subseteq
\pi_2^{-1}(X_1(4N)^\an_{\geq 2^{-r}}) & p= 2\end{array}\right.$$ for
any $r< 1/p(1+p)$.  This follows from repeated application of the
observation (made, for example, in \cite{buzzard}, Theorem 3.3 ({\it
  v})) that if $v(E)< p/(1+p)$ and $C$ is a subgroup of order $p$
other than the canonical subgroup, then $v(E/C) = v(E)/p$ and the
canonical subgroup of $E/C$ is $E[p]/C$.

If $\kappa\in\scr{W}^i_n$ and $r$ is chosen so that
$0\leq r\leq r_n$, then we may apply the construction above with
\begin{center}
\begin{tabular}{c|c|c}
  & $p\neq 2$ & $p=2$ \\ \hline $\scr{X}$ & $X_1(4Np,p^2)_K^\an$ &
  $X_1(4N,4)_K^\an$ \\ $\scr{Y}$ & $X_1(4Np)^\an_K$ & $X_1(4N)^\an_K$
  \\ $\scr{Z}$ & $X_1(4Np,p^2)^\an_{\geq p^{-pr}}$ &
  $X_1(4N,4)^\an_{\geq 2^{-2r}}$\\
 $D$ & $\Sigma_{4Np}$ & $\Sigma_{4N}$\\ 
$H$ &  $\H_{p^2}(\kappa)$ & $\H_{4}(\kappa)$\\ 
$\scr{U}$ & $X_1(4Np)^\an_{\geq p^{-r}}$ & $X_1(4N)^\an_{\geq 2^{-r}}$
  \\ $\scr{V}$ & $X_1(4Np)^\an_{\geq p^{-pr}}$ & $X_1(4N)^\an_{\geq 2^{-2r}}$
\end{tabular}
\end{center}
to arrive at a linear map $$\left\{\!\begin{array}{cc}
H^0(X_1(4Np)^\an_{\geq p^{-r}},\OO(\Sigma_{4Np})) \longrightarrow
H^0(X_1(4Np)^\an_{\geq p^{-pr}},\OO(\Sigma_{4Np})) & p\neq
2\\ H^0(X_1(4N)^\an_{\geq 2^{-r}},\OO(\Sigma_{4N})) \longrightarrow
H^0(X_1(4N)^\an_{\geq 2^{-2r}},\OO(\Sigma_{4N})) &
p=2\end{array}\right.$$ This map commutes with the diamond operators
and restricts to a map
  $$\widetilde{M}_\kappa(4N,K,p^{-r})\longrightarrow 
  \widetilde{M}_\kappa (4N,K,p^{-pr}).$$  
When composed with the natural restriction map
\begin{equation}\label{compactrestriction}
  \widetilde{M}_\kappa(4N,K,p^{-pr}) \longrightarrow
  \widetilde{M}_\kappa(4N,K,p^{-r})
\end{equation}
and divided by $p^2$, we arrive at an endomorphism of
$\widetilde{M}_\kappa(4N,K,p^{-r})$ which we denote by $U_{p^2}$.

\begin{prop}
  The Hecke operators defined above are continuous.
\end{prop}
\begin{proof}
  Each of the spaces arising in the construction is a Banach space
  over $K$, so it suffices to show that each of the constituent maps
  of which our Hecke operators are the composition has finite norm.
  By Lemma \ref{ignorecusps} we may ignore the residue disks around
  the cusps when computing norms, thereby reducing ourselves to the
  supremum norm on functions.  It follows easily that the pullback,
  restriction, and trace maps have norm not exceeding $1$ and that
  multiplication by $H$ has norm not exceeding the supremum norm of
  $H$ on the complement of the residue disks around the cusps.  The
  latter is finite since this complement is affinoid.
\end{proof}

\begin{rems}\
  \begin{itemize}
  \item 
    In the overconvergent case, \emph{i.e.} when we have $r>0$, the
    restriction map (\ref{compactrestriction}) is compact (see
    Proposition A5.2 of \cite{coleman}).  It follows that $U_{p^2}$ is compact
    as it is the composition of a continuous map with a compact map.
  \item The Hecke operators $T_{\ell^2}$ and $U_{\ell^2}$ preserve
    the space of cusp forms, as can be seen by simply constructing them
    directly on this space in the same manner as above.  The operator
    $U_{p^2}$ is compact on a space of overconvergent cusp forms.
  \end{itemize}
\end{rems}

\subsection{Hecke operators in families}
Let $X\subseteq\scr{W}$ be a connected admissible affinoid open.  We
wish to define endomorphisms of $\widetilde{M}_X(4N,K,p^{-r})$ that
interpolate the endomorphisms $T_{\ell^2}$ and $U_{\ell^2}$
constructed above for fixed weights $\kappa\in X$.

Suppose that $\ell\neq p$ and let 
\begin{center}
\begin{tabular}{c|c|c}
& $p\neq 2$ & $p=2$ \\ \hline $\scr{U}=\scr{V}$ & $X_1(4Np)^\an_{\geq
    p^{-r}}$ & $X_1(4N)^\an_{\geq 2^{-r}}$ \\ $\scr{Z}$ &
  $X_1(4Np,\ell^2)^\an_{\geq p^{-r}}$ & $X_1(4N,\ell^2)^\an_{\geq
    2^{-r}}$ \\ $\Sigma$ & $\Sigma_{4Np}$ & $\Sigma_{4N}$
\end{tabular}
\end{center}
In the
interest of keeping notation under control, let us for the remainder
of this section assume the following definitions.
\begin{eqnarray*}
  M &=& H^0(\scr{U},\OO(\Sigma)) \\
  N &=& H^0(\pi_2^{-1}(\scr{U})\cap \scr{Z},\OO(\pi_2^*\Sigma)) \\
  L &=& H^0(\pi_1^{-1}(\scr{V})\cap \scr{Z},\OO(\pi_2^*\Sigma)) \\
  P &=& H^0(\pi_1^{-1}(\scr{V})\cap
  \scr{Z},\OO(\pi_1^*\Sigma-\pi_2^*\Sigma)) \\ 
  Q &=& H^0(\pi_1^{-1}(\scr{V})\cap \scr{Z},\OO(\pi_1^*\Sigma)) \\
\end{eqnarray*}
The Hecke operator $T_{\ell^2}$ (or $U_{\ell^2}$ if $\ell\mid 4N$) at
a fixed weight was constructed in the previous section by first taking
the composition of the following continuous maps: a pullback $M\to N$,
a restriction $N\to L$, multiplication by an element of $H\in P$ to
arrive at an element of $Q$, and a trace $Q\to M$, and then
restricting to an eigenspace of the diamond operators at $p$ and
dividing by $\ell^2$.

The module of families of forms on $X$ is an eigenspace of
$M\widehat{\otimes}_K\OO(X)$ (by the final remark in Remarks
\ref{remsfamilies}).  To define $T_{\ell^2}$ (or $U_{\ell^2}$) we
begin as in the fixed weight case by defining an endomorphism of
$M\widehat{\otimes}_K\OO(X)$ and then observing that it commutes with
the diamond automorphisms and therefore restricts to an operator on
families of modular forms.  To define this endomorphism, we modify the
above sequence of maps by first applying $\widehat{\otimes}_K\OO(X)$ to
all of the spaces and taking the unique continuous $\OO(X)$-linear
extension of each map, with the exception of the multiplication step,
where we opt instead to multiply by $\mathbf{H}_{\ell^2}|_X\in
P\widehat{\otimes}_K\OO(X)$.  In so doing we arrive at an
$\OO(X)$-linear endomorphism of $M\widehat{\otimes}_K\OO(X)$ that is
easily seen to commute with the diamond automorphisms, thereby
inducing an endomorphism of the module $\widetilde{M}_X(4N,K,p^{-r})$.

\begin{lemm}
  The Hecke operators defined above for families are continuous.
\end{lemm}
\begin{proof}
  By definition, each map arising in the construction is continuous
  except perhaps for the multiplication map.  The proof of the
  continuity of this map requires several simple facts about completed
  tensor products, all of which can be found in section 2.1.7 of
  \cite{bgr}.
  
  It follows trivially from Lemma \ref{ignorecusps} that the
  multiplication map $$L\times P \longrightarrow Q$$
  is a bounded
  $K$-bilinear map and therefore extends uniquely to a bounded
  $K$-linear map $$L\widehat{\otimes}_K P \longrightarrow Q.$$  Extending
  scalars to $\OO(X)$ and completing we arrive at a bounded
  $\OO(X)$-linear map
  $$(L\widehat{\otimes}_K P)\widehat{\otimes}_K \OO(X)\longrightarrow
  Q\widehat{\otimes}_K\OO(X).$$   There is an isometric isomorphism
  $$(L\widehat{\otimes}_K P)\widehat{\otimes}_K\OO(X) \cong
  (L\widehat{\otimes}_K\OO(X))\widehat{\otimes}_{\OO(X)}
  (P\widehat{\otimes}_K\OO(X))$$ 
  so we conclude that the $\OO(X)$-bilinear multiplication map
  $$(L\widehat{\otimes}_K\OO(X))\widehat{\otimes}_{\OO(X)}
  (P\widehat{\otimes}_K\OO(X)) \longrightarrow
  Q\widehat{\otimes}_K\OO(X)$$
  is bounded.  In particular,
  multiplication by $H\in P\widehat{\otimes}_K\OO(X)$ is a bounded (and
  hence continuous) map
  $$\cdot H: L\widehat{\otimes}_K\OO(X)\longrightarrow
  Q\widehat{\otimes}_K\OO(X)$$
  as desired.
\end{proof}

\begin{rems}\
  \begin{itemize}
  \item 
    The construction of a continuous endomorphism $U_{p^2}$ is
    entirely analogous and once again we find that $U_{p^2}$ is compact
    in the overconvergent case, that is, whenever $r>0$.
  \item The endomorphisms $T_{\ell^2}$ and $U_{\ell^2}$ can be
    extended to $\widetilde{M}_X(4N,K,p^{-r})$ for general admissible
    affinoid opens $X$ in the usual manner working component by
    component.
  \item All of the the Hecke operators defined on families preserve the
    cuspidal subspaces, as a direct construction on these spaces
    demonstrates.  Again, the operator $U_{p^2}$ is compact on a
    module of overconvergent cusp forms.
  \end{itemize}
\end{rems}

\subsection{Effect on $q$-expansions}

In this section we will work out the effect of the Hecke operators
that we have defined on $q$-expansions.  As in \cite{mfhi}, we must
adjust the naive $q$-expansions obtained by literally evaluating our
forms on Tate curves with level structure to get at the classical
$q$-expansions.  In particular, by the $q$-expansion of a form $F\in
\widetilde{M}_\kappa(4N,K,p^{-r})$ at the cusp associated to
$(\Tate(q),\zeta)$ where $\zeta$ is a primitive
$4Np^{\small\mathrm{th}}$ root of unity if $p\neq 2$ and a primitive
$4N^{\small\mathrm{th}}$ root of unity if $p=2$, we mean
$$F(\Tate(q),\zeta)\theta(q)E_\kappa(q) $$
Similarly, for a
family $F\in M_X(4N,K,p^{-r})$ the corresponding $q$-expansion is
$$F(\Tate(q),\zeta)\theta(q)E(q)|_X$$
and has coefficients in the
ring of analytic functions on $X$.

\begin{prop}
  Let $F$ be an element of $\widetilde{M}_\kappa(4N,K,p^{-r})$ or
  $\widetilde{M}_X(4N,K,p^{-r})$ and let $\sum a_n q^n$ be the
  $q$-expansion of $F$ at $(\Tate(q),\zeta)$.  The corresponding
  $q$-expansion of $U_{p^2}F$ is then $\sum a_{p^2n}q^n$.
\end{prop}
\begin{proof}
  We prove the theorem for $U_{p^2}$ acting on
  $\widetilde{M}_\kappa(4N,K,p^{-r})$.  To obtain the result for
  families one could either proceed in the same manner or deduce the
  result for families over $X$ from the result for fixed weight by
  specializing to weights in $X$.  Let $F\in \widetilde{M}_\kappa
  (4N,K,p^{-r})$ and suppose that
  $$F(\Tate(q),\zeta)\theta(q)E_\kappa(q) = \sum a_n q^n.$$
  The
  expansion we seek is $$\frac{1}{p^2}\pi_{1*}(\pi_2^*F \cdot
  \H_{p^2}(\kappa))(\Tate(q),\zeta) \cdot
  \theta(q)E_\kappa(q).$$
  The cyclic subgroups of order $p^2$ that
  intersect the subgroup generated by $\zeta$ trivially are
  exactly those of the form $\ip{\zeta_{p^2}^iq_{p^2}}$, $0\leq i\leq
  p^2-1$.  Thus we have
  \begin{eqnarray*}
    \lefteqn{\pi_{1*}(\pi_2^*F\cdot \H_{p^2}(\kappa))(\Tate(q),\zeta) =
    \sum_{i=0}^{p^2-1} (\pi_2^*F\cdot
    \H_{p^2}(\kappa))(\Tate(q),\zeta,\ip{\zeta_{p^2}^iq_{p^2}}) }
    && \\ &=& 
    \sum_{i=0}^{p^2-1}F(\Tate(q)/\ip{\zeta_{p^2}^iq_{p^2}},
    \zeta/\ip{\zeta_{p^2}^iq_{p^2}})
    \H_{p^2}(\kappa)(\Tate(q),\zeta,\ip{\zeta_{p^2}^iq_{p^2}}) \\ &=& 
    \sum_{i=0}^{p^2-1} F(\Tate(\zeta_{p^2}^iq_{p^2}),\zeta)
    \H_{p^2}(\kappa)(\Tate(q),\zeta,\ip{\zeta_{p^2}^iq_{p^2}}) \\ &=&
    \sum_{i=0}^{p^2-1} \frac{ \sum a_n (\zeta_{p^2}^iq_{p^2})^n
    }{\theta(\zeta_{p^2}^iq_{p^2})E_\kappa(\zeta_{p^2}^iq_{p^2})}
    \frac{\theta(\zeta_{p^2}^iq_{p^2}) E_\kappa(\zeta_{p^2}^iq_{p^2})}
    {\theta(q)E_\kappa(q)}  = p^2\frac{\sum
    a_{p^2n}q^n}{\theta(q)E_\kappa(q)} 
  \end{eqnarray*}
\end{proof}

\noindent The same analysis also proves the following.
\begin{prop}
  Suppose that either $\ell|4N$.  Let $F$ be an element of
  $\widetilde{M}_\kappa(4N,K,p^{-r})$ or
  $\widetilde{M}_X(4N,K,p^{-r})$ and let $\sum a_n q^n$ be the
  $q$-expansion of $F$ at $(\Tate(q),\zeta)$.  Then the corresponding
  $q$-expansion of $U_{\ell^2}F$ is then $\sum a_{\ell^2n}q^n$.
\end{prop}

In order to work out the effect of $T_{\ell^2}$ for $\ell\nmid 4Np$ on
$q$-expansions, we will need several more $q$-expansions of
$\Theta_{\ell^2}$ and $\E_{\ell^2}$.  For the former, we refer the
reader to \cite{mfhi}.  The latter will follow from the following
lemma.  For $x\in\Z_p^\times$, we denote by $[x]$ the analytic
function on $\scr{W}$ defined by $[x](\kappa) = \kappa(x)$.
\begin{lemm}\label{moreqexpansions}
  For $\ell\neq p$ we have $$\E_\ell(\Tate(q),\mu_p+\ip{q_\ell}) =
  [\ip{\ell}] \frac{E(q)}{E(q_\ell)}\ \ \mbox{and}\ \ 
  \E_\ell(\Tate(q),\mu_{p\ell}) = \frac{E(q)}{E(q^\ell)}.$$
\end{lemm}
\begin{proof}
  The second equality is how we chose to characterize $\E_\ell$ in the
  first place.  We will use it to give an alternative
  characterization, which we will in turn use to prove the first
  equality.  
  
  By definition, $\E_\ell$ and the coefficients of $E(q)$ are
  pulled back from their restrictions to $\scr{W}^0$ through the map
  (\ref{wmap}).  Clearly $[\ip{\ell}]$ is the pull-back of $[\ell]$
  through this map, so it suffices to prove that
  $$\E_\ell(\Tate(q),\mu_p+\ip{q_\ell}) = [\ell]\frac{E(q)}
  {E(q_\ell)}$$
  where the coefficients are now though of as function
  only on $\scr{W}^0$.  Moreover, it suffices to prove the equality
  after specialization to integers $\lambda\geq 2$ divisible by $\varphi(\q)$,
  as such integers are Zariski-dense in $\scr{W}^0$.  Let
  $E_\lambda(\tau)$ denote the classical analytic $p$-deprived
  Eisenstein series of weight $\lambda$ and level $p$ (normalized to
  have $q$-expansion $E_\lambda(q)$).  Then
  $$\E^\an_\ell(\lambda) := E_\lambda(\tau)/E_\lambda(\ell\tau)$$
  is a
  meromorphic function on $X_0(p\ell)_\C^\an$ with rational
  $q$-expansion coefficients, and by GAGA and the $q$-expansion
  principle yields a rational function on the algebraic curve
  $X_0(p\ell)_{\Q_p}$.  By comparing $q$-expansions it is
  evident that the restriction of this function to the region
  $X_0(p\ell)^\an_{\geq 1}$ is equal to the specialization,
  $\E_\ell(\lambda)$, of $\E_\ell$ to $\lambda\in\scr{W}^0$.
  
  It follows that $\E_\ell(\lambda)(\Tate(q),\mu_p+\ip{q_\ell}) =
  \E^\an_\ell(\lambda)(\Tate(q),\mu_p+\ip{q_\ell})$. The right side
  can be computed using the usual yoga where one pretends to
  specialize $q$ to $e^{2\pi i\tau}$ and then computes with analytic
  transformation formulas (see Section 5 of \cite{mfhi} for a rigorous
  explanation of this yoga).  So specializing, we get
  $$\E^\an_\ell(\lambda)(\Tate(q),\mu_p+\ip{q_\ell})(\tau) =
  \E^\an_\ell(\lambda)(\C/\ip{1,\tau},\ip{1/p} +\ip{\tau/\ell}).$$
  Choosing a matrix $$\gamma= \left(\begin{matrix}a & b \\ c &
    d\end{matrix}\right)\in\SL_2(\Z)$$
  such that $p|c$ and $\ell|d$ we arrive
  at an isomorphism
  \begin{eqnarray*}
    (\C/\ip{1,\tau},\ip{1/p}+\ip{\tau/\ell}) &
    \stackrel{\sim}{\longrightarrow} &
    (\C/\ip{1,\gamma\tau},\ip{1/p\ell}) \\ 
    z & \longmapsto & \frac{z}{c\tau+d}
  \end{eqnarray*}
  Thus $$\E^\an_\ell(\lambda)(\C/\ip{1,\tau},\ip{1/p}+\ip{\tau/\ell}) =
  \E^\an_\ell(\lambda)(\C/\ip{1,\gamma\tau},\ip{1/p\ell}) =
  \frac{E_\lambda(\gamma\tau) }{E_\lambda(\ell\gamma\tau)}.$$  Now
  $$\ell\gamma\tau =
  \frac{(a\ell)(\tau/\ell)+b}{c(\tau/\ell)+d/\ell},$$ so we have
  $$\frac{E_\lambda(\gamma\tau)}{E_\lambda(\ell\gamma\tau)} =
  \frac{(c\tau+d)^\lambda E_\lambda(\tau)} {((c\tau+d)/\ell)^\lambda
  E_\lambda(\tau/\ell)} = \ell^\lambda \frac{E_\lambda(\tau)}
  {E_\lambda(\tau/\ell)}$$ and the result follows.
\end{proof}

\begin{prop}
  Let $F\in \widetilde{M}_\kappa(4N,K,p^{-r},\chi)$ with
  $\kappa\in\scr{W}^i$ and let $\sum a_nq^n$ be the $q$-expansion of
  $F$ at $(\Tate(q),\zeta)$.
  Then the corresponding $q$-expansion of $T_{\ell^2}F$ is $\sum
  b_nq^n$ where $$b_n = a_{\ell^2n} + \kappa(\ell)\chi(\ell)\ell^{-1}
  \left(\!\frac{(-1)^in}{\ell}\!\right)a_n +
  \kappa(\ell)^2\chi(\ell)^2\ell^{-1}a_{n/\ell^2}.$$
  
  Let $F\in \widetilde{M}_X(4N,K,p^{-r},\chi)$ with $X$ a connected
  affinoid in $\scr{W}^i$, and let the $q$-expansion of $F$ be $\sum
  a_n q^n$ as above.  Then the corresponding $q$-expansion of
  $T_{\ell^2}F$ is $\sum b_nq^n$ where
  $$b_n = a_{\ell^2n} + [\ell]\chi(\ell)\ell^{-1}
  \left(\!\frac{(-1)^in}{\ell}\!\right)a_n +
  [\ell]^2\chi(\ell)^2\ell^{-1}a_{n/\ell^2}.$$
\end{prop}
\begin{proof}
  We prove the first assertion.  The second assertion may either be
  proven directly in the same manner or simply deduced from the first
  via specialization to individual weights in $X$.  Let
  $\kappa\in\scr{W}(K)$, let $F\in \widetilde{M}_\kappa
  (4N,K,p^{-r},\chi)$, and let
  $$F(\Tate(q),\zeta)\theta(q)E_\kappa(q) = \sum a_n q^n$$
  be
  the $q$-expansion of $F$ at $(\Tate(q),\zeta)$.  The
  corresponding $q$-expansion of $T_{\ell^2}F$ is
\begin{equation}\label{qexp}
\frac{1}{\ell^2}\pi_{1*}(\pi_2^*F \cdot
  \H_{\ell^2}(\kappa)) \cdot \theta(q)E_\kappa(q).
\end{equation}
The cyclic subgroups of $\Tate(q)$ of order $\ell^2$ are the subgroups
  $$\mu_{\ell^2},\ \ \  \ip{\zeta_{\ell^2}^iq_{\ell^2}}_{0\leq i\leq
  \ell^2-1},\ \ \mbox{and}\ \   \ip{\zeta_{\ell^2}^jq_\ell}_{1\leq
  j\leq \ell-1}.$$ We examine the contribution of each of these types
  of subgroups to $$\pi_{1*}(\pi_2^*F\cdot \H_{\ell^2}(\kappa))$$
  separately.  

  First, we have 
 \begin{eqnarray*}
    \lefteqn{   F(\Tate(q)/\mu_{\ell^2},\zeta/\mu_{\ell^2})
      \H_{\ell^2}(\kappa)(\Tate(q),\zeta,\mu_{\ell^2})   } \\ &=&
    F(\Tate(q^{\ell^2}),\zeta^{\ell^2})
    \Theta_{\ell^2}(\Tate(q),\zeta_4,\mu_{\ell^2})
    \pi^*\E_{\ell^2}(\kappa)(\Tate(q),\zeta_p,\mu_{\ell^2})  \\ &=&
    F(\Tate(q^{\ell^2}),\zeta^{\ell^2})
    \Theta_{\ell^2}(\Tate(q),\zeta_4,\mu_{\ell^2})\\ &&
    \hspace{9em}\cdot\E_{\ell^2}(\kappa)(\Tate(q)/\mu_{\ell^2},(\mu_p +
    \Tate(q)[\ell^2])/\mu_{\ell^2}) 
        \\ &=&
    F(\Tate(q^{\ell^2}),\zeta^{\ell^2})
    \Theta_{\ell^2}(\Tate(q),\zeta_4,\mu_{\ell^2})
    \E_{\ell^2}(\kappa)(\Tate(q^{\ell^2}), \mu_p +  \ip{q})
  \end{eqnarray*}
  From the definition (\ref{Edefinition}) and Lemma
  \ref{moreqexpansions} we have
\begin{eqnarray*}
\lefteqn{
\E_{\ell^2}(\Tate(q^{\ell^2}),\mu_p+\ip{q})} \\  &=& \E_\ell
(\Tate(q^{\ell^2}),\mu_p+\ip{q^\ell}) \E_\ell
(\Tate(q^{\ell^2})/\ip{q^\ell},(\mu_p+\ip{q})/\ip{q^\ell}) \\
&=& \E_\ell(\Tate(q^{\ell^2}),\mu_p+\ip{q^\ell}) \E_\ell
(\Tate(q^\ell),\mu_p+\ip{q})  \\ &=&
[\ip{\ell}]\frac{E(q^{\ell^2})} {E(q^\ell)} \cdot
[\ip{\ell}] \frac{E(q^\ell)} {(q)} =
[\ip{\ell}]^2\frac{E(q^{\ell^2})} {E(q)}
\end{eqnarray*}
When specialized to $\kappa$, this becomes
$$\kappa(\ip{\ell})^2\frac{E_\kappa(q^{\ell^2})} {E_\kappa(q)}.$$
Referring to  \cite{mfhi} we find
$$\Theta_{\ell^2}(\Tate(q),\zeta_4,\mu_{\ell^2}) =
\ell\frac{\theta(q^{\ell^2})} {\theta(q)}.$$
Thus the contribution of
this first subgroup is
$$\frac{\chi(\ell^2)\tau(\ell^2)^i\sum a_n q^{\ell^2n}}
{\theta(q^{\ell^2})E_\kappa(q^{\ell^2})}\ell\frac{\theta(q^{\ell^2})}
{\theta(q)} \kappa(\ip{\ell})^2\frac{E_\kappa(q^{\ell^2})} {E_\kappa(q)} =
(\kappa(\ip{\ell})\chi(\ell)\tau(\ell)^i)^2\frac{\ell\sum a_nq^{\ell^2n}}
{\theta(q)E_\kappa(q)}$$ 

The subgroups $\ip{\zeta_{\ell^2}^aq_{\ell^2}}$ contribute 
\begin{eqnarray*}
\lefteqn{\sum_{a=0}^{\ell^2-1}  F(\Tate(q)/\ip{\zeta_{\ell^2}^aq_{\ell^2}},
  \zeta/\ip{\zeta_{\ell^2}^aq_{\ell^2}})   
  \mathbf{H}_{\ell^2}(\kappa)(\Tate(q),\zeta,
  \ip{\zeta_{\ell^2}^aq_{\ell^2}}) 
  } \\ &= & 
  \sum_{a=0}^{\ell^2-1}F(\Tate(\zeta_{\ell^2}^aq_{\ell^2}),\zeta)
  \Theta_{\ell^2}(\Tate(q),\zeta_4,\ip{\zeta_{\ell^2}^aq_{\ell^2}})\\
  &&\hspace{9em}\cdot
  \pi^*\E_{\ell^2}(\kappa)(\Tate(q),\zeta_p,\ip{\zeta_{\ell^2}^aq_{\ell^2}})
  \\ &=& \sum_{a=0}^{\ell^2-1}F(\Tate(\zeta_{\ell^2}^aq_{\ell^2}),\zeta)
  \Theta_{\ell^2}(\Tate(q),\zeta_4,\ip{\zeta_{\ell^2}^aq_{\ell^2}})\\
  &&\hspace{9em}\cdot\E_{\ell^2}(\kappa)
  (\Tate(q)/\ip{\zeta_{\ell^2}^aq_{\ell^2}},
  (\mu_p+\Tate(q)[\ell^2])/  \ip{\zeta_{\ell^2}^aq_{\ell^2}}) \\ &=& 
  \sum_{a=0}^{\ell^2-1}F(\Tate(\zeta_{\ell^2}^aq_{\ell^2}),\zeta)
  \Theta_{\ell^2}(\Tate(q),\zeta_4,\ip{\zeta_{\ell^2}^aq_{\ell^2}})\\ &&
  \hspace{9em}\cdot
  \E_{\ell^2}(\kappa)(\Tate(\zeta_{\ell^2}^aq_{\ell^2}), \mu_{p\ell^2})
\end{eqnarray*}
By (\ref{Edefinition}) we have 
\begin{eqnarray*}
  \E_{\ell^2}(\Tate(\zeta_{\ell^2}^aq_{\ell^2}),\mu_{p\ell^2}) &=&
  \E_\ell(\Tate(\zeta_{\ell^2}^aq_{\ell^2}),\mu_{p\ell})\E_\ell
  (\Tate(\zeta_{\ell^2}^aq_{\ell^2})/\mu_\ell,\mu_{p\ell^2}/\mu_\ell)
  \\ &=& \E_\ell(\Tate(\zeta_{\ell^2}^aq_{\ell^2}),\mu_{p\ell})
  \E_\ell (\Tate(\zeta_\ell^aq_\ell),\mu_{p\ell}) \\ &=&
  \frac{E(\zeta_{\ell^2}^aq_{\ell^2})}
  {E(\zeta_\ell^aq_\ell)} \frac{E(\zeta_\ell^aq_\ell)}
  {E (q)} = \frac{E(\zeta_{\ell^2}^aq_{\ell^2})}
  {E(q)} 
\end{eqnarray*}
Referring to \cite{mfhi}, we find
$$\Theta_{\ell^2}(\Tate(q),\zeta_4,\ip{\zeta_{\ell^2}^aq_{\ell^2}}) =
\frac{\theta(\zeta_{\ell^2}^aq_{\ell^2})} {\theta(q)}.$$  Thus the
total contribution of this collection of subgroups is 
$$\sum_{a=0}^{\ell^2-1} \frac{\sum a_n(\zeta_{\ell^2}^aq_{\ell^2})^n}
{\theta(\zeta_{\ell^2}^aq_{\ell^2})E_\kappa(\zeta_{\ell^2}^aq_{\ell^2})}
\frac{\theta(\zeta_{\ell^2}^aq_{\ell^2})} {\theta(q)}
\frac{E_\kappa(\zeta_{\ell^2}^aq_{\ell^2})} {E_\kappa (q)} =
\ell^2\frac{\sum a_{\ell^2n}q^n} {\theta(q)E_\kappa(q)}.$$

The subgroups $\ip{\zeta_{\ell^2}^bq_\ell}$ contribute
\begin{eqnarray*}
  \lefteqn{ \sum_{b=1}^{\ell-1}
  F(\Tate(q)/\ip{\zeta_{\ell^2}^bq_\ell},
  \zeta/\ip{\zeta_{\ell^2}^bq_\ell})
  \H_{\ell^2}(\kappa)(\Tate(q),\zeta,\ip{\zeta_{\ell^2}^bq_\ell}) }
  \\ &=& \sum_{b=1}^{\ell-1} F(\Tate(\zeta_\ell^b q),\zeta^\ell)
  \Theta_{\ell^2}(\Tate(q),\zeta_4,\ip{\zeta_{\ell^2}^bq_\ell})\\ &&
  \hspace{9em}\cdot
  \pi^*\E_{\ell^2}(\kappa) (\Tate(q),\zeta_p, \ip{\zeta_{\ell^2}^bq_\ell}) \\
  &=& \sum_{b=1}^{\ell-1}  F(\Tate(\zeta_\ell^bq),\zeta^\ell)
  \Theta_{\ell^2}(\Tate(q),\zeta_4,\ip{\zeta_{\ell^2}^bq_\ell}) \\ &&
  \hspace{9em}\cdot
  \E_{\ell^2}(\kappa)(\Tate(q)/\ip{\zeta_{\ell^2}^bq_\ell},
  (\mu_p+\Tate(q)[\ell^2])/\ip{\zeta_{\ell^2}^bq_\ell}) \\
       &=& \sum_{b=1}^{\ell-1}  F(\Tate(\zeta_\ell^bq),\zeta^\ell)
  \Theta_{\ell^2}(\Tate(q),\zeta_4,\ip{\zeta_{\ell^2}^bq_\ell})\\ &&
  \hspace{9em}\cdot
  \E_{\ell^2}(\kappa)(\Tate(\zeta_\ell^bq), \mu_p+\ip{q_\ell})
\end{eqnarray*}
By (\ref{Edefinition}) and Lemma \ref{moreqexpansions} we have
\begin{eqnarray*}
\lefteqn{  \E_{\ell^2}(\Tate(\zeta_\ell^bq),\mu_p+\ip{q_\ell}) } \\ &=&
  \E_\ell(\Tate(\zeta_\ell^bq),\mu_p+\ip{q})
  \E_\ell(\Tate(\zeta_\ell^bq)/\mu_\ell,(\mu_p+\ip{q_\ell})/\mu_\ell)
  \\ &=&
  \E_\ell(\Tate(\zeta_\ell^bq),\mu_{p\ell})\E_\ell(\Tate(q^\ell),\mu_p+
  \ip{q}) \\ &=& \frac{E(\zeta_\ell^bq)}
  {E(q^\ell)} \cdot [\ip{\ell}]\frac{E(q^\ell)}
  {E(q)} = [\ip{\ell}]\frac{E(\zeta_\ell^bq)} {E(q)}
\end{eqnarray*}
When specialized to $\kappa$, this becomes
$$\kappa(\ip{\ell})\frac{E_\kappa(\zeta_\ell^bq)} {E_\kappa(q)}.$$
Referring to \cite{mfhi} we find
$$\Theta_{\ell^2}(\Tate(q),\zeta_4,\ip{\zeta_\ell^bq}) =
\left(\!\frac{-1}{\ell}\!\right)\mathfrak{g}_\ell(\zeta_\ell^b)
\frac{\theta(\zeta_\ell^bq)} {\theta(q)}$$
where
$$\mathfrak{g}_{\ell}(\zeta) =
\sum_{m=1}^{\ell-1}\left(\!\frac{m}{\ell}\!\right)\zeta^m$$ is the
Gauss sum associated to the $\ell^{\small\mathrm{th}}$ root of unity
$\zeta$.  
Thus the total contribution of this third
collection of subgroups is
\begin{eqnarray*}
\lefteqn{
\sum_{b=1}^{\ell-1} \frac{\chi(\ell)(-1/\ell)^i\tau(\ell)^i\sum
  a_n(\zeta_\ell^bq)^n}{\theta(\zeta_\ell^bq)E_\kappa(\zeta_\ell^bq)}
\left(\!\frac{-1}{\ell}\!\right)\mathfrak{g}_\ell(\zeta_\ell^b)
\frac{\theta(\zeta_\ell^bq)} {\theta(q)}
\kappa(\ip{\ell})\frac{E_\kappa(\zeta_\ell^bq)} {E_\kappa(q)} } \\ &=& 
\kappa(\ip{\ell})\chi(\ell)\left(\!\frac{-1}{\ell}\!
\right)^{i+1}\tau(\ell)^i  
\frac{\mathfrak{g}_\ell(\zeta_\ell)}
{\theta(q)E_\kappa(q)} \sum_n a_n\left(\sum_{b=1}^{\ell-1}
  \zeta_\ell^{bn}\left(\!\frac{b}{\ell}\!\right)\right)q^n \\ &=&
\kappa(\ip{\ell})\chi(\ell)\left(\!\frac{-1}
  {\ell}\!\right)^{i+1}\tau(\ell)^i
\frac{\mathfrak{g}_\ell(\zeta_\ell)} {\theta(q)E_\kappa(q)}\sum_n a_n
\left(\!\frac{n}{\ell}\!\right)\mathfrak{g}_\ell(\zeta_\ell)q^n \\ &=&
\kappa(\ip{\ell})\chi(\ell)\left(\!\frac{-1}{\ell}\!\right)^i\tau(\ell)^i
\frac{\ell\sum 
  \left(\!\frac{n}{\ell}\!\right) a_n q^n}
{\theta(q)E_\kappa(q)}
\end{eqnarray*}

Adding all this up and plugging into (\ref{qexp}) we see that the
$q$-expansion of $T_{\ell^2}F$ is $\sum b_nq^n$ where 
\begin{eqnarray*}
b_n & = & 
a_{\ell^2 n} +
\kappa(\ip{\ell})\ell^{-1}\chi(\ell)\left(\!\frac{-1}{\ell}\!\right)^i
\tau(\ell)^i\left(\!\frac{n}{\ell}\!\right) a_n
+
\kappa(\ip{\ell})^2\ell^{-1}\chi(\ell)^2\tau(\ell)^{2i}a_{n/\ell^2}
\\ &=& 
a_{\ell^2 n} +
\kappa(\ell)\ell^{-1}\chi(\ell)\left(\!\frac{(-1)^in}{\ell}\!\right) a_n +
\kappa(\ell)^2\ell^{-1}\chi(\ell)^2a_{n/\ell^2}.
\end{eqnarray*}

\end{proof}

\section{Classical weights and classical forms}\label{classicality}

In this section we define classical subspaces of our spaces of modular
forms and prove the following analog of Coleman's theorem on
overconvergent forms of low slope.  Throughout this section $k$ will
denote an odd positive integer and we set $\lambda=(k-1)/2$.
\begin{theo}\label{lowslopeisclassical}
  Let $m$ be a positive integer, let
  $\psi:(\Z/\q p^{m-1}\Z)^\times\longrightarrow K^\times$ be a character, and
  define $\kappa(x)=x^\lambda\psi(x)$. If $F\in
  \widetilde{M}^\dagger_{\kappa}(4N,K)$ satisfies $U_{p^2}F=\alpha F$
  with $v(\alpha)<2\lambda-1$, then $F$ is classical.
\end{theo}
Our proof follows the approach of Kassaei (\cite{kassaei}), which is
modular in nature and builds the classical form by analytic
continuation and gluing.  The term ``analytic continuation'' has
little meaning here since we have only defined our modular forms over
restricted regions on the modular curve, owing to the need to avoid
Eisenstein zeros.  To get around this difficulty, we must invoke the
previous formalism of the author for $p$-adic modular forms of
\emph{classical} half-integral weight (see \cite{mfhi}).

Let $N$ be a positive integer.  In \cite{mfhi} we
defined the space of modular forms of weight $k/2$ and level $4N$ over
a $\Z[1/4N]$-algebra $R$ to be the $R$-module
$$\widetilde{M}'_{k/2}(4N,R) := H^0(X_1(4N)_R,\OO(k\Sigma_{4N})).$$
Note that this space was denoted $M_{k/2}(4N,R)$ and $k\Sigma_{4N}$
was denoted $\Sigma_{4N,k}$ in \cite{mfhi}.  Roughly speaking, in this
space of forms we have divided by $\theta^k$ to reduce to weight zero
instead of $E_\lambda\theta$.  Let $r\in [0,1]\cap \Q$ and define
$$\widetilde{M}'_{k/2}(4Np^m,K,p^{-r}) = H^0(X_1(4Np^m)^\an_{\geq
  p^{-r}},\OO(k\Sigma_{4Np^m})).$$
It is an easy matter to check that the
construction of the Hecke operators $T_{\ell^2}$ and $U_{p^2}$ in
Section \ref{sec:hecke} (using $H = \Theta_{\ell^2}^k$) adapts to this
space of forms and furnishes us with Hecke operators having the
expected effect on $q$-expansions.  We will briefly review the
construction of $U_{p^2}$ in this context later in this section.  

The next proposition relates these spaces of $p$-adic modular forms to
the ones defined in this paper, and will ensure that the latter spaces
(and consequently the eigencurve defined later in this paper) see the
classical half-integral weight modular forms of arbitrary $p$-power
level.  Note that this identification requires the knowledge of the
action of the diamond operators at $p$ because this data is part of
the $p$-adic weight character.

\begin{prop}\label{comparisonofdefs}
  Let $m$ be a positive integer, let
  $\psi:(\Z/\q p^{m-1}\Z)^\times\longrightarrow K^\times$ be a character, and
  define $\kappa(x) = x^\lambda\psi(x)$.  Then, for $0\leq r\leq r_m$,
  the space
  $$\widetilde{M}'(4Np^{m+1}/\q,K,p^{-r})^{\ip{}^*_{\q p^{m-1}}=\psi}=
  \left\{\!\begin{array}{cc}
  \widetilde{M}'_{k/2}(4Np^m,K,p^{-r})^{\ip{}^*_{p^m}=\psi} & p\neq
  2\\ \widetilde{M}'_{k/2}(2^{m+1}N,K,p^{-r})^{\ip{}^*_{2^{m+1}}=\psi} &
  p=2
\end{array}\right.$$
is isomorphic to 
  $\widetilde{M}_{\kappa}(4N,K,p^{-r})$ in a manner
  compatible with the action of the Hecke operators and tame diamond
  operators.
\end{prop}
\begin{proof}
  Let $i$ be such that $\kappa\in \scr{W}^i$.  The complex-analytic
  modular forms $\theta^{k-1}$ and $E_{\kappa\tau^{-i}}$ are each of
  weight $\lambda$.  If $p\neq 2$, then the former is invariant under
  the $\ip{d}^*_{\q p^{m-1}}$ while if $p=2$ it has eigencharacter
  $(-1/\cdot)^i$.  The latter has eigencharacter $\psi\tau^{-i}$ for
  this action in both cases.  Standard arguments using GAGA and the
  $q$-expansion principle show that the ratio
  $\theta^{k-1}/E_{\kappa\tau^{-i}}$ furnishes an algebraic rational
  function on $X_1(4Np^{m+1}/\q)_K$.  Passing to the $p$-adic
  analytification and restricting to $X_1(4Np^{m+1}/\q)^\an_{\geq
    p^{-r}}$, we see that this function has divisor
  $(k-1)\Sigma_{4Np^{m+1}/\q}$, since $E_{\kappa\tau^{-i}}$ is
  invertible in this region for $r$ as in the statement of the
  proposition (because $\kappa\in\scr{W}_m$).

  Let $F'\in \widetilde{M}'_{k/2}(4Np^{m+1}/\q,K,p^{-r})$ be a form
  with eigencharacter $\psi$ for $\ip{}^*_{\q p^{m-1}}$ and let $$F =
  F'\cdot\frac{\theta^{k-1}}{E_{\kappa\tau^{-i}}}$$ Then, for $d\in
  (\Z/\q p^{m-1}\Z)^\times$ we have $\ip{d}_{\q p^{m-1}}^*F =
  \tau(d)^i(-1/\cdot)^iF$.  In particular, $F$ is fixed by
  $\ip{d}_{p^m}^*$ with $d\equiv 1\pmod{\q}$.  The construction of the
  canonical subgroup of order $\q p^{m-1}$ (defined because $r\leq
  r_m<p^{2-m}/\q(1+p)$) ensures that the map
  \begin{equation}\label{quot}
  X_1(4Np^{m+1}/\q)^\an_{\geq p^{-r}}/\{\ip{d}_{\q
    p^{m-1}}\ |\ d\equiv 1\!\!\pmod{\q}\}\longrightarrow
  \left\{\!\begin{array}{cc} X_1(4Np)^\an_{\geq p^{-r}} & p\neq 2
    \\ X_1(4N)^\an_{\geq 2^{-r}} & p=2\end{array}\right.
  \end{equation}
  induced
  by $$(E,P)\longmapsto (E,aP)$$ where the integer $a$ is chosen so that 
  $$\left\{\!\begin{array}{cc}
  a\equiv p^{m-1} \pmod{p^m}\ \ \mbox{and}\ \ a\equiv 1\pmod{4N} & p\neq 2
  \\ a\equiv 2^{m-1}\pmod{2^{m+1}}\ \ \mbox{and}\ \ a\equiv 1\pmod{N} & p=2
  \end{array}\right.$$
  is an isomorphism.  This map pulls the divisor $\Sigma_{4Np}$ (or
  $\Sigma_{4N}$ if $p=2$) back to $\Sigma_{4Np^m}$
  (resp. $\Sigma_{2^{m+1}N}$ if $p=2$), so we conclude that $F$
  descends to a section of $\OO(\Sigma_{4Np})$ on $X_1(4Np)^\an_{\geq
    p^{-r}}$ (resp. a section of $\OO(\Sigma_{4N})$ on
  $X_1(4N)^\an_{\geq 2^{-r}}$) and that this section satisfies
  $\ip{d}_p^*F = \tau(d)^iF$ (resp. $\ip{d}_4^*F =
  \tau(d)^i(-1/d)^iF$) for all $d\in(\Z/\q\Z)^\times$.  Thus we may
  regard $F$ as an element of $\widetilde{M}_{\kappa}(4N,K,p^{-r})$.
  Conversely, for $F\in \widetilde{M}_{\kappa}(4N,K,p^{-r})$, it is
  easy to see that $$F\cdot
  \frac{E_{\kappa\tau^{-i}}}{\theta^{k-1}}\in
  \widetilde{M}'_{k/2}(4Np^{m+1}/\q,K,p^{-r})^{\ip{}_{\q
      p^{m-1}}=\psi}$$ (where $F$ is implicitly pulled back via the
  above map (\ref{quot})) and that this furnishes an inverse to the
  above map $F'\mapsto F$.  That these maps are equivariant with
  respect to the Hecke action is a formal manipulation with the setup
  in Section \ref{sec:hecke} used to define the action on both sides.
  That it is equivariant with respect to tame diamond operators is
  trivial, but relies essentially on the ``twisted'' convention for
  this action on $\widetilde{M}_{\kappa}(4N,K,p^{-r})$ (for $p\neq
  2$).
\end{proof}

In general, if $\scr{U}$ is a connected admissible open in
$X_1(4Np^{m+1}/\q)^\an_K$ containing $X_1(4Np^{m+1}/\q)^\an_{\geq
  p^{-r}}$ and $F\in \widetilde{M}_{\kappa}(4N,K,p^{-r})$ (with
$\kappa$ as in the previous proposition) we will say that $F$
analytically continues to $\scr{U}$ if the corresponding form
$F'\in\widetilde{M}'_{k/2}(4Np^{m+1}/\q,K,p^{-r})$ analytically continues to
an element of
\begin{equation}\label{formsoveru}
H^0(\scr{U},\OO(k\Sigma_{4Np^{m+1}/\q})).
\end{equation}  
Note that, in case $\scr{U}$ is preserved by the diamond operators at
$p$, this analytic continuation automatically lies in the
$\psi$-eigenspace of (\ref{formsoveru}) since
$G-\ip{d}_{\q p^{m-1}}^*G$ vanishes on the nonempty admissible open
$X_1(4Np^{m+1}/\q)^\an_{\geq p^{-r}}$ for all $d$, and hence must
vanish on all of $\scr{U}$.  In particular, in case $\scr{U} =
X_1(4Np^{m+1}/\q)^\an_K$ we make the following definition.
\begin{defi}
  Let $\kappa(x) = x^\lambda\psi(x)$ be as in Proposition
  \ref{comparisonofdefs}.  An element $F\in
  \widetilde{M}_{\kappa}(4N,K)^\dagger$ is called \emph{classical} if
  it analytically continues in the sense described above to all of
  $X_1(4Np^{m+1}/\q)^\an_K$.  That is, if it is in the image of the
  (injective) map 
\begin{eqnarray*} 
   H^0(X_1(4Np^{m+1}/\q)^\an_K,
      \OO(k\Sigma_{4Np}))^{\ip{}_{p^m}=\psi}  &\longrightarrow &
      \widetilde{M}'_{k/2}(4Np^{m+1}/\q,K,
      p^{-r_m})^{\ip{}_{p^m}=\psi} \\ &\cong& 
      \widetilde{M}_{\kappa}(4N,K,p^{-r_m})  \\ &\hookrightarrow &
      \widetilde{M}_{\kappa}(4N,K)^\dagger
\end{eqnarray*}
\end{defi}

The analytic continuation used to prove Theorem
\ref{lowslopeisclassical} will proceed in three steps.  All of them
involve the construction of the operator $U_{p^2}$ on
$\widetilde{M}'_{k/2}(4Np^{m+1}/\q,K,p^{-r})$, which goes as follows.  Let
$$\pi_1,\pi_2 : X_1(4Np^{m+1}/\q,p^2)^\an_K \longrightarrow
X_1(4Np^{m+1}/\q)^\an_K$$ be the usual pair of maps and let
$\Theta_{p^2}$ denote the rational function on $X_1(4,p^2)_{\Q}$ from
Section \ref{sec:somefunctions}.  For any pair of admissible open
$\scr{U}$ and $\scr{V}$ in $X_1(4Np^{m+1}/\q)^\an_K$ with
$$\pi_1^{-1}\scr{V} \subseteq \pi_2^{-1}\scr{U}$$ we have the map 
\begin{eqnarray*}
  H^0(\scr{U},\OO(k\Sigma_{4Np^{m+1}/\q}))  & \longrightarrow &
  H^0(\scr{V},\OO(k\Sigma_{4Np^{m+1}/\q})) \\ F & \longmapsto &
 \frac{1}{p^2} \pi_{1*}(\pi_2^*F   \cdot \Theta_{p^2}^k)
\end{eqnarray*}
Note that there is no need to introduce the space $\scr{Z}$ as in
Section \ref{sec:hecke} since our ``twisting'' section
$\Theta_{p^2}^k$ is defined on all of $X_1(4Np^{m+1}/\q,p^2)^\an_K$.
Also, recall from Section \ref{sec:hecke} that if $0\leq r<1/p(1+p)$ we have
$$\pi_1^{-1} (X_1(4Np^{m+1}/\q)^\an_{\geq p^{-p^2r}})\subseteq
\pi_2^{-1}(X_1(4Np^{m+1}/\q)^\an_{\geq p^{-r}})$$ Thus if $F\in
\widetilde{M}'_{k/2}(4Np^{m+1}/\q,K,p^{-r})$ with $r<1/p(1+p)$ then
$U_{p^2}F$ analytically continues to $X_1(4Np^{m+1}/\q)^\an_{\geq
  p^{-p^2r}}$.  From this simple observation we get the first and
easiest analytic continuation result.
\begin{prop}\label{prop:an1}
  Let $r>0$ and let $F\in \widetilde{M}'_{k/2}(4Np^{m+1}/\q,K,p^{-r})$.
  Suppose that there exists a polynomial $P(T)\in K[T]$ with $P(0)\neq
  0$ such that $P(U_{p^2})F$ analytically continues to
  $X_1(4Np^{m+1}/\q)^\an_{\geq p^{-1/(1+p)}}$.  Then $F$ analytically
  continues to this region as well.
\end{prop}
\begin{proof}
  Write $P(T) = P_0(T)+a$ with $P_0(0)=0$ and $a\neq 0$.  Then $$F =
  \frac{1}{a}\left( P(U_{p^2})F-P_0(U_{p^2})F\right).$$ If
  $0<r<1/p(1+p)$, then the right side analytically continues to
  $X_1(4Np^{m+1}/\q)^\an_{\geq p^{-p^2r}}$, and hence so does $F$.
  Since $r>0$, we may repeat this process until we have analytically
  continued $F$ to $X_1(4Np^{m+1}/\q)^\an_{\geq p^{-s}}$ for some
  $s\geq 1/p(1+p)$.  Now restrict $F$ to $X_1(4Np^{m+1}/\q)^\an_{\geq
    p^{-1/p^2(1+p)}}$ and apply the process once more to get the
  desired result.
\end{proof}

The second analytic continuation step requires that we introduce some
admissible opens in $X_1(4Np^{m+1}/\q)^\an_{\Q_p}$ defined by Buzzard
in \cite{buzzard}.  The use of the letter $\scr{W}$ in this part of
the argument is intended to keep the notation parallel to that in
\cite{buzzard} and should not be confused with weight space.  If
$p\neq 2$, we let $\scr{W}_0\subseteq X_1(4N,p)^\an_{\Q_p}$ denote the
admissible open subspace whose points reduce to the irreducible
component on the special fiber of $X_1(4N,p)$ in characteristic $p$
that contains the cusp associated to the datum $(\Tate(q),P,\mu_p)$
for some (equivalently, any) point of order $4N$ on $\Tate(q)$.
Alternatively, $\scr{W}_0$ can be characterized as the compliment of
the connected component of the ordinary locus in
$X_1(4N,p)_{\Q_p}^\an$ containing the cusp associated to
$(\Tate(q),P,\ip{q_p})$ for some (equivalently, any) choice of $P$.
If $p=2$, we let $\scr{W}_0\subseteq X_1(N,2)^\an_{\Q_p}$ denote the
admissible open subspace whose points reduce to the irreducible
component on the special fiber of $X_1(N,2)$ in characteristic $2$
that contains the cusp associated to the datum $(\Tate(q),P,\mu_2)$
for some (equivalently, any) point of order $N$ on $\Tate(q)$.
Alternatively, $\scr{W}_0$ can be characterized as the compliment of
the connected component of the ordinary locus in $X_1(N,2)_{\Q_p}^\an$
containing the cusp associated to $(\Tate(q),P,\ip{q_2})$ for some
(equivalently, any) choice of $P$.  In particular $\scr{W}_0$ always
contains the entire supersingular locus.  The reader concerned about
problems with small $N$ in these descriptions should focus on the
``alternative'' versions and the remarks in Section
\ref{sec:modcurves} about adding level structure and taking invariants.

In \cite{buzzard}, Buzzard introduces a map
$v':\scr{W}_0\longrightarrow \Q$ defined as follows.  If
$x\in\scr{W}_0$ is a cusp, then set $v'(x)=0$.  Otherwise,
$x\in\scr{W}_0$ corresponds to a triple $(E/L,P,C)$ with $E/L$ an
elliptic curve, $P$ a point of order $4N$ ($N$ if $p=2$) on $E$, and
$C\subset E$ a cyclic subgroup of order $p$.  If $E$ has bad or
ordinary reduction, then set $v'(x)=0$.  Otherwise, if
$0<v(E)<p/(1+p)$, then $E$ has a canonical subgroup $H$ of order $p$, and we
define
$$v'(x) = \left\{\begin{array}{ll} v(E) & H=C \\ 1-v(E/C) & H\neq
C\end{array}\right.$$ Finally, if $v(E)\geq p/(1+p)$ we define $v'(x)
= p/(1+p)$.  Note that $v'$ does not depend on the point $P$.  For a
nonnegative integer $n$, we let $V_n$ denote the region in $\scr{W}_0$
defined by the inequality $v'\leq 1-1/p^{n-1}(1+p)$.  Buzzard proves
that $V_n$ is an admissible affinoid open in $\scr{W}_0$ for each $n$,
and that $\scr{W}_0$ is admissibly covered by the $V_n$.

Let 
$$f:X_1(4Np^{m+1}/\q)^\an_{\Q_p}\longrightarrow
\left\{\!\begin{array}{cc} X_1(4N,p)^\an_{\Q_p} & p\neq 2
\\ X_1(N,2)^\an_{\Q_p} & p=2\end{array}\right.$$ denote the map
characterized by $$(E,P)\longmapsto \left\{\!\begin{array}{cc}
(E/\ip{4NpP}, p^mP/\ip{4NpP},\ip{4NP/\ip{4NpP}} & p\neq 2
\\ (E/\ip{2NP}, 2^{m+1}P/\ip{2NP}, \ip{NP/\ip{2NP}}) &
p=2\end{array}\right.$$ on noncuspidal points.  Define $\scr{W}_1 =
f^{-1}(\scr{W}_0)$ and $Z_n=f^{-1}(V_n)$ for $n\geq 0$.  It follows
from the above that $\scr{W}_1$ is an admissible open in
$X_1(4Np^{m+1}/\q)^\an_K$ and that $\scr{W}_1$ is admissibly covered
by the admissible opens $Z_n$.  The latter are affinoid since $f$ is
finite.
\begin{lemm}\label{lem:inclusion}
  The inclusion $\pi_1^{-1}(Z_{n+2})\subseteq \pi_2^{-1}(Z_n)$ holds for
  all $n\geq 0$.
\end{lemm}
\begin{proof}
  Since the maps $\pi_1$ and $\pi_2$ are finite, the stated inclusion
  is between affinoids and can be checked on noncuspidal points.  Then
  the assertion follows immediately from two applications of Lemma 4.2
  (2) of \cite{buzzard}.  
\end{proof}

We can now state and prove the second analytic continuation result.
\begin{prop}\label{prop:an2}
  Let $r>0$ and let $F\in
  \widetilde{M}'_{k/2}(4Np^{m+1}/\q,K,p^{-r})$.  Suppose that there
  exists a polynomial $P(T)\in K[T]$ with $P(0)\neq 0$ such that
  $P(U_{p^2})F$ extends to $\scr{W}_1$.  Then $F$ extend to this
  region as well.
\end{prop}
\begin{proof}
  Note that $$X_1(4Np^{m+1}/\q)^\an_{\geq p^{-1/(1+p)}} = Z_0\subseteq
  \scr{W}_1$$ so that by Proposition \ref{prop:an1}, $F$ extends to
  $Z_0$.  Now we proceed inductively to extend $F$ to each $Z_n$.  Let
  $P(T) = P_0(T)+a$ with $P_0(0)=0$ and $a\neq 0$.  Then $$F =
  \frac{1}{a}( P(U_{p^2})F - P_0(U_{p^2})F).$$ Suppose $F$ extends to
  $Z_n$ for some $n\geq 0$.  By hypothesis $P(U_{p^2})F$ extends to
  all of $\scr{W}_1$, and by the construction of $U_{p^2}$ and Lemma
  \ref{lem:inclusion}, $P_0(U_{p^2})F$ extends to $Z_{n+2}$, and hence
  so does $F$.  Thus by induction $F$ extends to $Z_n$ for all $n$,
  and since $\scr{W}_1$ is admissibly covered by the $Z_n$, $F$
  extends to $\scr{W}_1$.
\end{proof}

If $p\neq 2$ and $m=1$ (that is, if there is only one $p$ in the level),
then this is the end of the second analytic continuation step.  In all
other cases, Buzzard's techniques in \cite{buzzard} allow us to
analytically continue to more connected components of the ordinary
locus.  Define $$\mathbf{m} = \ord_p(\q p^{m-1}) =
\left\{\!\begin{array}{cc} m & p\neq 2 \\ m+1 &
p=2\end{array}\right.$$ Following Buzzard, for $0\leq r\leq
\mathbf{m}$ let $\scr{U}_r$ denote the admissible open in
$X_1(4Np^{m+1}/\q)^\an_K$ whose non-cuspidal points parameterize pairs
$(E,P)$ that are either supersingular or satisfy
$$H_{p^{\m-r}}(E)=\left\{\!\begin{array}{cc}
H_{p^{m-r}}(E) = \ip{4Np^rP} & p\neq 2 \\
H_{2^{m+1-r}}(E) = \ip{N2^rP} & p=2
\end{array}\right.$$
 We have $$\scr{W}_1=\scr{U}_0\subseteq\scr{U}_1\subseteq\cdots
 \subseteq \scr{U}_{\m}=X_1(4Np^{m+1}/\q)^\an_K$$ The last goal of the
 second step is to analytically continue eigenforms to 
$\scr{U}_{\m-1}$.

\begin{lemm}\label{lem:inclusion2}
For $0\leq r\leq \m-2$ we have  $\pi_1^{-1}(\scr{U}_{r+1})\subseteq
\pi_2^{-1}(\scr{U}_r)$.
\end{lemm}
\begin{proof}
  As usual, it suffices to check this on non-cuspidal points.
  Moreover, it suffices to check it on ordinary points, since the
  entire supersingular locus is contained in each $\scr{U}_r$.  For
  brevity we will assume $p\neq 2$.  The case $p=2$ is proven in
  exactly the same manner.  Let $(E,P,C)\in \pi_1^{-1}(\scr{U}_{r+1})$
  be such a point.  Then $H_{p^{m-r-1}}(E) = \ip{4Np^{r+1}P}$ and
  since $r+1<m$, we conclude that $H_{p^{m-r-1}}(E)\cap C=0$.  Now
  Proposition 3.5 of \cite{buzzard} implies that $H_{p^r}(E/C)$ is
  indeed generated by the image of $4Np^rP$ in $E/C$, so $(E,P,C)\in
  \pi_2^{-1}(\scr{U}_r)$.
\end{proof}

\begin{prop}\label{prop:an2.1}
  Let $r>0$ and let $F\in
  \widetilde{M}'_{k/2}(4Np^{m+1}/\q,K,p^{-r})$.  Suppose that there
  exists a polynomial $P(T)\in K[T]$ with $P(0)\neq 0$ such that
  $P(U_{p^2})F$ extends to $\scr{U}_{\m-1}$.  Then $F$ extend to this
  region as well.
\end{prop}
\begin{proof}
 Since $\scr{U}_0 = \scr{W}_1$, Proposition \ref{prop:an2} ensures
 that $F$ analytically continues to $\scr{U}_0$.  Now we proceed
 inductively to extend $F$ to each $\scr{U}_r$, $0\leq r\leq \m-1$.
 Let $P(T) = P_0(T)+a$ with $P_0(0)=0$ and $a\neq 0$.  Then $$F =
 \frac{1}{a}( P(U_{p^2})F - P_0(U_{p^2})F).$$ Suppose $F$ extends to
 $\scr{U}_r$ for some $0\leq r\leq \m-2$.  By hypothesis $P(U_{p^2})F$
 extends to all of $\scr{U}_{\m-1}$, and by the construction of
 $U_{p^2}$ and Lemma \ref{lem:inclusion2}, $P_0(U_{p^2})F$ extends to
 $\scr{U}_{r+1}$, and hence so does $F$.  Proceeding inductively, we
 see that $F$ can be extended all the way to $\scr{U}_{\m-1}$.
\end{proof}

The third and most difficult analytic continuation step is to continue
to the rest of the curve $X_1(4Np^{m+1}/\q)^\an_K$.  If $p\neq 2$, we
let $\scr{V}_0$ denote the admissible open in $X_1(4N,p)^\an_K$ whose
points reduce to the irreducible component on the special fiber in
characteristic $p$ that contains the cusp associated to
$(\Tate(q),P,\ip{q_p})$ for some (equivalently, any) choice of $P$.
On the other hand, if $p=2$, we let $\scr{V}_0$ denote the admissible
open in $X_1(N,2)^\an_K$ whose points reduce to the irreducible
component on the special fiber in characteristic $2$ that contains the
cusp associated to $(\Tate(q),P,\ip{q_2})$ for some (equivalently,
any) choice of $P$.  Let $\scr{V}$ denote the preimage of $\scr{V}_0$
under the finite map
\begin{eqnarray*}
  g:X_1(4Np^{m+1}/\q)^\an_{\Q_p} & \longrightarrow &
  \left\{\!\begin{array}{cc} X_1(4N,p)^\an_{\Q_p} & p\neq 2
  \\ X_1(N,2)^\an_{\Q_p} & p=2\end{array}\right.
\\ (E,P) & \longmapsto & \left\{\!\begin{array}{cc}
(E,p^mP,\ip{4Np^{m-1}P}) & p\neq 2 \\
(E,2^{m+1}P,\ip{2^mNP}) & p=2\end{array}\right.
\end{eqnarray*}
Note that the preimage under $g$ of the locus that reduces to the
other component of $X_1(4N,p)_{\F_p}$ (or $X_1(N,2)_{\F_2}$ if
$p=2$) is $\scr{U}_{\m-1}$, so in particular $\{\scr{U}_{\m-1},
\scr{V}\}$ is an admissible cover of $X_1(4Np^{m+1}/\q)^\an_{\Q_p}$
and $\scr{U}_{\m-1}\cap \scr{V}$ is the supersingular locus.

For any subinterval $I\subseteq (p^{-p/(1+p)},1]$ let $\scr{V}I$
  (respectively $\scr{U}_{\m-1}I$) denote the admissible open in
  $\scr{V}$ (respectively $\scr{U}_{\m-1}$) defined by the condition
  $p^{-v(E)}\in I$.  Note that the complement of $\scr{U}_{\m-1}$ in
  $X_1(4Np^{m+1}/\q)^\an_K$ is $\scr{V}[1,1]$.  Given a
  $U_{p^2}$-eigenform of suitably low slope we will define a function
  on $\scr{V}[1,1]$ and use the gluing techniques of \cite{kassaei} to
  glue it to the analytic continuation of our eigenform to
  $\scr{U}_{\m-1}$ guaranteed by Proposition \ref{prop:an2}. These
  techniques rely heavily on the norms introduced in Section
  \ref{norms}.  The use of Lemma \ref{ignorecusps} to reduce these
  norms to the supremum norm on the complement of the residue disks
  around the cusps will be implicit in many of the estimates that
  follow.

Over $\scr{V}(p^{-1/p(1+p)},1]$ we have a section $h$ to $\pi_1$ given
on noncuspidal points by 
\begin{eqnarray*}
h:\scr{V}(p^{-1/p(1+p)},1] &\longrightarrow &
  X_1(4Np^{m+1}/\q,p^2)^\an_K \\ (E,P) & \longmapsto & (E,P,H_{p^2})
\end{eqnarray*}
By standard results on quotienting by the canonical subgroup
(\cite{buzzard}, Theorem 3.3), the composition $\pi_2\circ h$
restricts to a map 
\begin{equation}\label{Qrestricts}
Q:\scr{V}(p^{-r},1]\longrightarrow
\scr{V}(p^{-p^2r},1]
\end{equation}
for any $0\leq r\leq 1/p(1+p)$.  Note that since $Q$ preserves the
property of having ordinary or supersingular reduction, $Q$ restricts
to a map $\scr{V}(p^{-r},1)\to\scr{V}(p^{-p^2r},1)$.  Define a
meromorphic function $\vartheta$ on $\scr{V}(p^{-1/p(1+p)},1]$ by
$\vartheta = h^*\Theta_{p^2}$, and note that
\begin{equation}\label{varthetadiv}
\div(\vartheta) = h^*(\pi_2^*\Sigma_{4Np^{m+1}/\q}
-\pi_1^*\Sigma_{4Np^{m+1}/\q}) = Q^*\Sigma_{4Np^{m+1}/\q}
-\Sigma_{4Np^{m+1}/\q}.
\end{equation}

Let $F\in H^0(\scr{U}_{\m-1},\OO(k\Sigma_{4Np^{m+1}/\q}))$ and
suppose that
$$U_{p^2}F = \alpha F+H$$ on $\scr{U}_{\m-1}$ for some \emph{classical}
form $H$ and some $\alpha\neq 0$.  Note that this condition makes
sense because $\pi_1^{-1}(\scr{U}_{\m-1}) \subseteq
\pi_2^{-1}(\scr{U}_{\m-1})$ by Lemma \ref{lem:inclusion2}. For a pair
$(E,P)\in\scr{U}_{\m-1}$ corresponding to a noncuspidal point, we have
\begin{equation}\label{F1extends}
F(E,P) = \frac{1}{\alpha p^2}\sum_C F(E/C,P/C)\Theta_{p^2}^k(E,P,C) -
\frac{1}{\alpha}H(E,P)
\end{equation}
where the sum is over the cyclic subgroups of order $p^2$ having
trivial intersection with the group generated by $P$.  Suppose that
$(E,P)$ corresponds to a point in $\scr{V}(p^{-1/p(1+p)},1)$.  Then
the subgroup generated by $P$ has trivial intersection with the
canonical subgroup $H_{p^2}$, and thus the canonical subgroup is among
the subgroups occurring in the sum above.  One can check using Theorem
3.3 of \cite{buzzard} that $(E/H_{p^2},P/H_{p^2})$ corresponds to a
point of $\scr{V}(p^{-p/(1+p)},1)$, while if $C\neq H_{p^2}$ is a
cyclic subgroup of order $p^2$ with trivial intersection with
$\ip{P}$, then $(E/C,P/C)$ corresponds to a point of
$\scr{U}_{\m-1}(p^{-1/p(1+p)},1]$.  Define $F_1$ on
$\scr{V}(p^{-1/p(1+p)},1)$ by
$$F_1 = F - \frac{1}{\alpha p^2
  }\vartheta^kQ^*(F|_{\scr{V}(p^{-p/(1+p)},1)}).$$

\begin{lemm}
  The function $F_1$ on $\scr{V}(p^{-1/p(1+p)},1)$ extends to an
  element of $H^0(\scr{V}(p^{-1/p(1+p)},1],\OO(k\Sigma_{4Np^{m+1}/\q}))$.
\end{lemm}
\begin{proof}
  Equation (\ref{F1extends}) and the comments that follow it show how
  to define the extension $\widetilde{F}_1$ of $F_1$, at least on
  noncuspidal points.  For a pair $(E,P)$ corresponding to a
  noncuspidal point of $\scr{V}(p^{-1/p(1+p)},1]$, we would like
  $$\widetilde{F}_1(E,P) = \frac{1}{\alpha p^2}\sum_C
  F(E/C,P/C)\Theta_{p^2}^k(E,P,C)-\frac{1}{\alpha}H(E,P)$$
  where the
  sum is over the cyclic subgroups of order $p^2$ of $E$ not meeting
  $\ip{P}$ and not equal to $H_{p^2}(E)$.  We can formalize this as
  follows.

  The canonical subgroup of order $p^2$ furnishes a section to the
  finite map
  $$\pi_1^{-1}(\scr{V}(p^{-1/p(1+p)},1])\stackrel{\pi_1}{
    \longrightarrow} \scr{V}(p^{-1/p(1+p)},1]$$
  and this section is an
  isomorphism onto a connected component of
  $\pi_1^{-1}(\scr{V}(p^{-1/p(1+p)},1])$.  Let $\scr{Z}$ denote the
  compliment of this connected component.  Then $\pi_1$ restricts to a
  finite and flat map $$\scr{Z}\longrightarrow
  \scr{V}(p^{-1/p(1+p)},1].$$
 Note that
 $$\scr{Z}=\pi_1^{-1}(\scr{V}(p^{-1/p(1+p)},1])\cap\scr{Z} \subseteq
 \pi_2^{-1}(\scr{U}_{\m-1}(p^{-1/p(1+p)},1])\cap\scr{Z}$$
 as can be checked
 on noncuspidal points (see the comments following Equation
 (\ref{F1extends})).  Now we may apply the general construction of
 Section \ref{sec:hecke} with this $\scr{Z}$ and define
 $$\widetilde{F}_1 = \frac{1}{\alpha
   p^2}\pi_{1*}(\pi_2^*F\cdot\Theta_{p^2}^k) - \frac{1}{\alpha}H.$$
 Then
 $$\widetilde{F}_1\in
 H^0(\scr{V}(p^{-1/p(1+p)},1],\OO(k\Sigma_{4Np^{m+1}/\q}))$$
 and Equation
 (\ref{F1extends}) shows that $\widetilde{F}_1$ extends $F_1$.
\end{proof}

For $n\geq 1$ we define an element $F_n$ of
$H^0(\scr{V}(p^{-1/p^{2n-1}(1+p)},1],\OO(k\Sigma_{4Np^{m+1}/\q}))$
  inductively, where $F_1$ is as above and for $n\geq 1$ we
  set $$F_{n+1} = F_1 + \frac{1}{\alpha p^2}\vartheta^k
  Q^*(F_n|_{\scr{V}(p^{-1/p^{2n+1}(1+p)},1]}).$$ Note that
    (\ref{Qrestricts}) and (\ref{varthetadiv}) show that the $F_n$ do
    indeed lie in the spaces indicated.  Our goal is to show that the
    sequence $\{F_n\}$, when restricted to $\scr{V}[1,1]$, converges
    to an element of $G$ of
    $H^0(\scr{V}[1,1],\OO(k\Sigma_{4Np^{m+1}/\q}))$ that glues to $F$
    in the sense that there exists a global section of
    $\OO(k\Sigma_{4Np^{m+1}/\q})$ that restricts to $F$ and $G$ on
    $\scr{U}_{\m-1}$ and $\scr{V}[1,1]$, respectively.  To do this we
    will use Kassaei's gluing lemma as developed in \cite{kassaei}.
    The following lemmas furnish some necessary norm estimates.

\begin{lemm}\label{thetaintegral}
  The function $\Theta_{p^2}$ on $Y_1(4,p^2)_{\Q_p}$ is integral.
  That is, it extends to a regular function on the fine moduli scheme
  $Y_1(4,p^2)_{\Z_p}$.
\end{lemm}
\begin{proof}
  Each $\Gamma_1(4)\cap \Gamma^0(p^2)$ structure on the elliptic curve
  $\Tate(q)/\Q_p(\!(q)\!)$ lifts trivially to one over the Tate curve
  thought of over $\Z_p(\!(q)\!)$.  Since the Tate curve is ordinary,
  such a structure specializes to a unique component of the special
  fiber $Y_1(4,p^2)_{\F_p}$.  Since $Y_1(4,p^2)_{\Z_p}$ is
  Cohen-Macaulay, the usual argument used to prove the $q$-expansion
  principal (as in the proof of Corollary 1.6.2 of \cite{katz}) shows
  that $\Theta_{p^2}$ is integral as long as it has integral
  $q$-expansion associated to a level structure specializing to each
  component of the special fiber.  In fact, all $q$-expansion of
  $\Theta_{p^2}$ are computed explicitly in Section 5 of \cite{mfhi},
  and are all integral.
\end{proof}

\begin{lemm}\label{thetabound1}
  Let $R$ be an $\F_p$-algebra, let $E$ be an elliptic curve over $R$,
  and let $E^{(p)}$ denote the base change of $E$ via the absolute
  Frobenius morphism on $\Spec(R)$.  Let $$\Fr :E\longrightarrow
  E^{(p)}$$
  denote the relative Frobenius morphism.  Then for any point $P$
  of order $4$ on $E$ we have
  $$\Theta_{p^2}(E,P,\ker( \Fr^2)) = 0$$
\end{lemm}
\begin{proof}
  In characteristic $p$, the forgetful map $$Y_1(4,p^2)_{\F_p}\longrightarrow
  Y_1(4)_{\F_p}$$
  has a section given on noncuspidal points by
  $$s:(E,P)\longmapsto (E,P,\ker(\Fr^2)).$$
  By Lemma
  \ref{thetaintegral}, we may pull back (the reduction of)
  $\Theta_{p^2}$ through this section to arrive at a regular function
  on the smooth curve $Y_1(4)_{\F_p}$.
  
  The $q$-expansion of $s^*\Theta_{p^2}$ at the cusp associated to
  $(\Tate(q),\zeta_{4})$ is $$s^*\Theta_{p^2}(\Tate(q),\zeta_{4}) =
  \Theta_{p^2}(\Tate(q),\zeta_4,(\ker(\Fr^2))).$$
  Recall that the
  map
  $$\Tate(q) \longrightarrow \Tate(q^p) $$ given by quotienting by
  $\mu_p$ is a lifting of $\Fr$ to characteristic zero (more
  specifically, to the ring $\Z(\!(q)\!)$). Thus the $q$-expansion we
  seek is the reduction of $$\Theta_{p^2}(\Tate(q),\zeta_4,\mu_{p^2})
  = p\frac{\sum_{n\in\Z} q^{p^2n^2}} {\sum_{n\in\Z} q^{n^2}}$$ modulo
  $p$, which is clearly zero.  We refer the reader to Section 5 of
  \cite{mfhi} for the computation of the above $q$-expansion in
  characteristic zero.  It follows from the $q$-expansion principle
  that $s^*\Theta_{p^2}=0$, which implies our claim.
\end{proof}

\begin{lemm}\label{thetabound2}
  Let $0\leq r<1/p(1+p)$.  Then the section $\vartheta$ of
  $$\OO(\Sigma_{4Np^{m+1}/\q} - Q^*\Sigma_{4Np^{m+1}/\q})$$
  satisfies
  $$\|\vartheta\|_{\scr{V}[p^{-r},1]}\leq p^{pr-1}.$$
\end{lemm}
\begin{proof}
  By Lemma \ref{ignorecusps}, we may ignore points reducing to cusps
  in computing the norm.  Let $x\in\scr{V}[p^{-r},1]$ be outside of
  this collection of points, so $x$ corresponds to a pair $(E,P)$ with
  good reduction.  Let $H_{p^i}$ denote the canonical subgroup of $E$
  of order $p^i$ (for whichever $i$ this is defined).  Let
  $\mathbf{E}$ be a smooth model of $E$ over $\OO_L$ and let
  $\mathbf{P}$ and $\mathbf{H}_{p^2}$ be the extensions of $P$ and
  $H_{p^2}$ to $\mathbf{E}$, respectively (these $\mathbf{E}$ and
  $\mathbf{H}$ should not be confused with the functions by the same
  name introduced in Section \ref{sec:somefunctions}).

  By Theorem 3.10 of \cite{gorenkassaei}, $\mathbf{H}_{p}$ reduces
  modulo $p/p^{v(E)}$ to $\ker(\Fr)$.  Applying this to
  $\mathbf{E}/\mathbf{H}_p$ we see that
  $\mathbf{H}_{p^2}/\mathbf{H}_p$ reduces modulo $p/p^{v(E/H_p)}$ to
  $\ker(\Fr)$ on the corresponding reduction of $E/H_p$.  By Theorem
  3.3 of \cite{buzzard}, we know that $v(E/H_p) = pv(E)$, so
  $p^{1-v(E/H_p)}\mid p^{1-v(E)}$ and we may combine these statements
  to conclude that $\mathbf{H}_{p^2}$ reduces modulo $p^{1-pv(E)}$ to
  $\ker(\Fr^2)$ on the reduction of $E$.

  Combining this with the integrality of $\Theta_{p^2}$ (from Lemma
  \ref{thetaintegral}), we have $$h(x)=
  \Theta_{p^2}(E,P,H_{p^2})\equiv \Theta_{p^2}(E,P,\ker(\Fr^2)) \pmod{
    p^{1-pv(E)}}.$$ This is zero by Lemma
  \ref{thetabound1}, so $$|h(x)|\leq |p^{1-pv(E)}| =
  p^{pv(E)-1}\leq 
  p^{pr-1}$$ as desired.
\end{proof}

\begin{prop}\label{prop:an3}
  Let $F\in H^0(\scr{U}_{\m-1},\OO(k\Sigma_{4Np^{m+1}/\q}))$ and suppose
  that $U_{p^2}F - \alpha F$ is classical for some $\alpha\in K$ with
  $v(\alpha)< 2\lambda -1$.  Then $F$ is classical as well.
\end{prop}
\begin{proof}
  Define $F_n$ as above.  We first show that the sequence
  $F_n|_{\scr{V}[1,1]}$ converges.  Note that over $\scr{V}[1,1]$ we
  have
  \begin{eqnarray*}
    F_{n+2}-F_{n+1} &=&
        \left(F_1 + 
        \frac{1}{\alpha 
        p^2}\vartheta^k Q^*F_{n+1}\right)  -
        \left( F_1 + 
      \frac{1}{\alpha p^2}\vartheta^k
        Q^*F_n\right)\\ & =&  \frac{1}{\alpha 
      p^2}\vartheta^k
        Q^*(F_{n+1}-F_n).
  \end{eqnarray*}
  By Lemma \ref{thetabound2} (with $r=0$) we have
  $$\|F_{n+2} - F_{n+1}\|_{\scr{V}[1,1]} \leq
  \frac{p^{2-k}}{|\alpha|}\|F_{n+1} - F_n\|_{\scr{V}[1,1]}.$$
  The
  hypothesis on $\alpha$ ensures that
  $$\left(\frac{p^{2-k}}{|\alpha|}\right)^n\longrightarrow 0\ \ \ 
  \mbox{as}\ \ \ n\longrightarrow \infty$$
  and hence that the sequence
  has successive differences that tend to zero.  As
  $H^0(\scr{V}[1,1],\OO(k\Sigma_{4Np^{m+1}/\q}))$ is a Banach algebra with
  respect to $\|\!\cdot\!\|_{\scr{V}[1,1]}$ by Lemma
  \ref{banachspaces}, it follows that the sequence converges.  Set $$G
  = \lim_{n\to \infty} F_n|_{\scr{V}[1,1]}.$$
  
  Next we apply Kassaei's gluing lemma (Lemma 2.3 of \cite{kassaei})
  to glue $G$ to $F$ as sections of the line bundle $\OO(\lfloor
  k\Sigma_{4Np^{m+1}/\q}\rfloor)$.  So that we are gluing over an affinoid as
  required in the hypotheses of the gluing lemma, we first restrict
  $F$ to $\scr{V}[p^{-1/p(1+p)},1)$ and glue $G$ to this restriction
  to get a section over the smooth affinoid
  $\scr{V}[p^{-1/p(1+p)},1]$.  Since the pair
  $\{\scr{V}[p^{-1/p(1+p)},1], \scr{U}_{\m-1}\}$ is an admissible cover of
  $X_1(4Np^{m+1}/\q)^\an_K$, this section glues to $F$ to give a global
  section.
  
  The ``auxiliary'' approximating sections that are required in the
  hypotheses of this lemma (denoted $F_n$ in \cite{kassaei}) are the
  $F_n$ introduced above.  So that the $F_n$ live on affinoids (as in
  the hypotheses of the gluing lemma) we simply restrict $F_n$ to
  $\scr{V}[p^{-1/p^{2n}(1+p)},1]$.  The two conditions to be verified
  are
  $$\|F_n-F\|_{\scr{V}[p^{-1/p^{2n}(1+p)},1)}\to 0\ \ \mbox{and}\ \ 
  \|F_n-G\|_{\scr{V}[1,1]}\to 0.$$
  The second of these is simply the
  definition of $G$.  As for the first, it is not even clear that the
  indicated norms are finite (since the norms are over non-affinoids).
  To see that these norms are finite and that the ensuing estimates
  make sense, we must show that $F$ has finite norm over
  $\scr{V}[p^{-1/p^2(1+p)},1)$.  It suffices to show that the norms of
  $F$ over the affinoids $$\scr{V}_n=
  \scr{V}[p^{-1/p^{2n}(1+p)},p^{-1/p^{2n+2}(1+p)}]$$ are uniformly
  bounded for $n\geq 1$.  The key is that the map $Q$ restricts to a
  map $$Q: \scr{V}_n\longrightarrow \scr{V}_{n+1}$$ for each $n\geq
  1$.  Since $F_1$ extends to the affinoid
  $\scr{V}[p^{-1/p^2(1+p)},1]$, its norms over the $\scr{V}_n$ are
  certainly uniformly bounded, say, by $M$.  We have 
  \begin{eqnarray*}
    \|F\|_{\scr{V}_n} & \leq & \max\left( \|F_1\|_{\scr{V}_n},
    \left\| \frac{1}{\alpha p^2}\vartheta^k Q^*F
    \right\|_{\scr{V}_n}\right) \\ &\leq & \max\left( M , 
    \frac{p^2}{|\alpha|}\|\vartheta^k\|_{\scr{V}_n}
    \|Q^*F\|_{\scr{V}_n}\right)  
    \\ &\leq & \max\left( M,
    \frac{p^2}{|\alpha|}\left(p^{\frac{1}{p^{2n-1}(1+p)}-1}
    \right)^k\|Q^*F\|_{\scr{V}_n} \right) \\ &\leq & \max\left( M,
    \frac{p^{2-k}}{|\alpha|}p^{\frac{k}{p^{2n-1}(1+p)}}\|F\|_{\scr{V}_{n-1}}
    \right) 
  \end{eqnarray*}
  Iterating this, we see that $\|F\|_{\scr{V}_n}$ does not exceed the
  maximum of
  $$\max_{0\leq m\leq n-2} \left(M\left(\frac{p^{2-k}}{|\alpha|}\right)^m
    p^{\frac{k}{1+p}\left(\frac{1}{p^{2n-1}}+\cdots
        +\frac{1}{p^{2(n-m)+1}}\right)}\right)$$
  and
  $$\left(\frac{p^{2-k}}{|\alpha|}\right)^{n-1}
  p^{\frac{k}{1+p}\left(\frac{1}{p^{2n-1}}+\cdots
      +\frac{1}{p^3}\right)}\|F\|_{\scr{V}_1}.$$
  The sums in the
  exponents of are geometric and do not exceed $1/(p^3-p)$.  Moreover,
  the hypothesis on $\alpha$ ensures that $p^{2-k}/|\alpha| <1$.  Thus we
  have $$\|F\|_{\scr{V}_n} \leq \max\left(
  Mp^{\frac{k}{1+p}\frac{1}{p^3-p}},
  p^{\frac{k}{1+p}\frac{1}{p^3-p}}\|F\|_{\scr{V}_1} \right), $$ which
  is independent of $n$, as desired.  This ensures that all of the
  norms encountered below are indeed finite.
  
  From the definition of the $F_n$, we have
 \begin{eqnarray*}
    F_{n+1}-F & = & F_1 +
    \frac{1}{\alpha p^2}\vartheta^kQ^*F_n - F \\
    &=& F-\frac{1}{\alpha p^2}\vartheta^k Q^*F + \frac{1}{\alpha
    p^2}\vartheta^k Q^*F_n -F \\
  &=& \frac{1}{\alpha p^2}\vartheta^k Q^*(F_n-F).
  \end{eqnarray*}
  Taking supremum norms over the appropriate admissible opens, we
  see
  \begin{eqnarray*}
    \lefteqn{
      \|F_{n+1}-F\|_{\scr{V}[p^{-1/p^{2n+2}(1+p)},1)}  } &&  \\ & \leq & 
    \frac{p^2}{|\alpha|}
    \|\vartheta\|^k_{\scr{V}[p^{-1/p^{2n+2}(1+p)},1)}
    \|Q^*(F_n-F)\|_{\scr{V}[p^{-1/p^{2n+2}(1+p)},1)} \\ &\leq& 
    \frac{p^2}{|\alpha|}\left(p^{\frac{1}{p^{2n+1}(1+p)}-1}\right)^k
      \|F_n-F\|_{\scr{V}[p^{-1/p^{2n}(1+p)},1)} \\ &=&
      \frac{p^{2-k}}{|\alpha|}p^{\frac{k}{p^{2n+1}(1+p)}}\|F_n -
      F\|_{\scr{V}[p^{-1/p^{2n}(1+p)},1)} 
  \end{eqnarray*}
  Iterating this we find that
  $$\|F_n-F\|_{\scr{V}[p^{-1/p^{2n}(1+p)},1)} \leq
  \left(\frac{p^{2-k}}{|\alpha|}\right)^{n-1}p^{\frac{k}{1+p}\left(
      \frac{1}{p^3}+\frac{1}{p^5}+\cdots+\frac{1}{p^{2n-1}} \right)}
  \|F_1 - F\|_{\scr{V}[p^{-1/p^2(1+p)},1)}.$$
  Again the sum in the
  exponent is less than $1/(p^3-p)$ for all $n$, so the hypothesis on
  $\alpha$ ensures that the above norm tends to zero as $n\to\infty$,
  as desired
\end{proof}

We are now ready to prove the main result, which is a mild
generalization of Theorem \ref{lowslopeisclassical} stated at the
beginning of this section.

\begin{theo}\label{genlowslopeisclassical}
  Let $m$ be a positive integer, let
  $\psi:(\Z/\q p^{m-1}\Z)^\times\longrightarrow K^\times$ be a character, and
  define $\kappa(x) = x^\lambda\psi(x)$.  Let $P(T)\in K[T]$ be a
  monic polynomial all roots of which have valuation less than
  $2\lambda-1$.  If $F\in \widetilde{M}^\dagger_{\kappa}(4N,K)$ and
  $P(U_{p^2})F$ is classical, then $F$ is classical as well.
\end{theo}
\begin{proof}
  Pick $0<r<r_m$ such that $F\in \widetilde{M}_{\kappa}(4N,K,p^{-r})$
  and let let $F'\in \widetilde{M}_{k/2}(4Np^{m+1}/\q,K,p^{-r})$ be
  the form corresponding to $F$ under the isomorphism of Proposition
  \ref{comparisonofdefs}.  We must show that $F'$ is classical in the
  sense that it analytically continues to all of
  $X_1(4Np^{m+1}/\q)^\an_K$.  Note that $P(0)\neq 0$ for such a
  polynomial, so by Proposition \ref{prop:an2.1}, $F'$ analytically
  continues to an element of $H^0(\scr{U}_{\m-1},
  \OO(k\Sigma_{4Np^{m+1}/\q}))$.  Now we proceed by induction on the
  degree $d$ of $P$.  The case $d=1$ is Proposition \ref{prop:an3}.
  Suppose the result holds for some degree $d\geq 1$ and let $P(T)$ be
  a polynomial of degree $d+1$ as above.  We may pass to a finite
  extension and write $$P(T) = (T-\alpha_1)\cdots(T-\alpha_{d+1}).$$
  The condition that $P(U_{p^2})F'$ is classical implies by the
  inductive hypothesis that $(U_{p^2}-\alpha_{d+1})F'$ is classical.
  This implies that $F'$ is classical by the case $d=1$.
\end{proof}

\begin{rema}
  The results of this section likely also follow from the very general
  classicality machinery developed in the recent paper
  \cite{kassaeiclassicality} of Kassaei, though we have not checked
  the details.
\end{rema}

\section{The half-integral weight eigencurve}\label{sec:eigencurve}

To construct our eigencurve, we will use the axiomatic version of
Coleman and Mazur's Hecke algebra construction, as set up by Buzzard
in his paper \cite{buzzardeigenvarieties}.  We briefly recall some
relevant details.

Let us for the moment allow $\scr{W}$ to be any reduced rigid space
over $K$.  Let $\mathbf{T}$ be a set with a distinguished element
$\phi$.  Suppose that, for each admissible affinoid open
$X\subseteq\scr{W}$, we are given a Banach module $M_X$ over $\OO(X)$
satisfying a certain technical hypothesis (called ({\it Pr}) in
\cite{buzzardeigenvarieties}) and a map
\begin{eqnarray*}
\mathbf{T} &\longrightarrow & \End_{\OO(X)}(M_X) \\
t & \longmapsto & t_X
\end{eqnarray*}
whose image consists of commuting endomorphisms and such that $\phi_X$
is compact for each $X$.  Assume that, for admissible affinoids
$X_1\subseteq X_2\subseteq \scr{W}$, we are given a continuous
injective $\OO(X_1)$-linear map $$\alpha_{12}: M_{X_1}\longrightarrow
M_{X_2}\widehat{\otimes}_{\OO(X_2)}\OO(X_1)$$
that is a ``link'' in the
sense of \cite{buzzardeigenvarieties} and such that
$(t_{X_2}\widehat{\otimes} 1)\circ\alpha_{12} = \alpha_{12}\circ
t_{X_1}$.  Assume moreover that, if $X_1\subseteq X_2\subseteq
X_3\subseteq \scr{W}$ are admissible affinoids, then $\alpha_{13} =
\alpha_{23}\circ\alpha_{12}$ with the obvious notation.  Note that the
link condition ensures that the characteristic power series $P_X(T)$
of $\phi_X$ acting on $M_X$ is independent of $X$ in the sense that
the image of $P_{X_2}(T)$ under the natural map
$\OO(X_2)[\![T]\!]\to\OO(X_1)[\![T]\!]$ is $P_{X_1}(T)$ (see
\cite{buzzardeigenvarieties}).

Out of this data, Buzzard constructs rigid analytic spaces $D$ and
$Z$, called the \emph{eigenvariety} and \emph{spectral variety},
respectively, equipped with canonical maps
\begin{equation}\label{eigenspecweight}
D\longrightarrow Z\longrightarrow\scr{W}.
\end{equation}
The points of $D$ parameterize systems of eigenvalues of $\mathbf{T}$
acting on the $\{M_X\}$ for which the eigenvalue of $\phi$ is nonzero,
in a sense that will be made precise in Lemma \ref{eigenpoints}, while
the image of such a point in $Z$ simply records the inverse of the
$\phi$ eigenvalue and a point of $\scr{W}$.  If $\scr{W}$ is
equidimensional of dimension $d$, then the same is true of both of the
spaces $D$ and $Z$.

As the details of this construction will be required in the next
section, we recall them here.  The following is Theorem 4.6 of
\cite{buzzardeigenvarieties}, and is the deepest part of the
construction.
\begin{theo}\label{cover}
  Let $R$ be a reduced affinoid algebra over $K$, let $P(T)$ be a
  Fredholm series over $R$, and let $Z\subset \Sp(R)\times \A^1$
  denote the hypersurface cut out by $P(T)$ equipped with the
  projection $\pi: Z\longrightarrow \Sp(R)$.  Define $\scr{C}(Z)$ to
  be the collection of admissible affinoid opens $Y$ in $Z$ such that
  \begin{itemize}
  \item $Y'=\pi(Y)$ is an admissible affinoid open in $\Sp(R)$,
  \item $\pi: Y\longrightarrow Y'$ is finite, and
  \item there exists $e\in \OO(\pi^{-1}(Y'))$ such that $e^2=e$ and $Y$
  is the zero locus of $e$.
  \end{itemize}
  Then $\scr{C}(Z)$ is an admissible cover of $Z$.
\end{theo}
\noindent We will generally take $Y'$ to be connected in what follows.  This is
not a serious restriction, since $Y$ is the disjoint union of the
parts lying over the various connected components of $Y'$.  We also
remark that the third of the above conditions follows from the first
two (this is observed in \cite{buzzardeigenvarieties} where references
to the proof are supplied).

To construct $D$, first fix an admissible affinoid open
$X\subseteq\scr{W}$.  Let $Z_X$ denote the zero locus of $P_X(T) =
\det(1-\phi_X T\ |\ M_X)$ in $X\times\A^1$ and let $\pi:Z_X\to X$
denote the projection onto the first factor.  Let $Y\in \scr{C}(Z_X)$
and let $Y'=\pi(Y)$ as above and assume that $Y'$ is connected.  We
wish to associate to $Y$ a polynomial factor of $P_{Y'}(T) =
\det(1-(\phi_{X}\widehat{\otimes}1)T\ |\ M_{X}\widehat{\otimes}_{\OO(X)}\OO(Y'))$. 
Since the algebra $\OO(Y)$ is a finite and locally free module over
$\OO(Y')$, we may consider the characteristic polynomial $Q'$ of $T\in
\OO(Y)$.  Since $T$ is a root of its characteristic polynomial, we have
a map
\begin{equation}\label{structureofOY}
\OO(Y')[T]/(Q'(T))\longrightarrow \OO(Y).
\end{equation}
It is shown in Section 5 of \cite{buzzardeigenvarieties} that this map
is surjective and therefore an isomorphism since both sides are
locally free of the same rank.

Now since the natural map
$$\OO(Y')[T]/(Q'(T))\longrightarrow \OO(Y')\{\!\{T\}\!\}/(Q'(T))$$
is an
isomorphism, it follows that $Q'(T)$ divides $P_{Y'}(T)$ in
$\OO(Y')\{\!\{T\}\!\}$.  If $a_0$ is the constant term of $Q'(T)$, then
this divisibility implies that $a_0$ is a unit.  We set $Q(T) =
a_0^{-1}Q'(T)$.  The spectral theory of compact operators on Banach
modules (see Theorem 3.3 of \cite{buzzardeigenvarieties}) furnishes a
unique decomposition
$$M_X\widehat{\otimes}_{\OO(X)}\OO(Y') \cong N\oplus F$$
into closed
$\phi$-invariant $\OO(Y')$-submodules such that $Q^*(\phi)$ is zero on
$N$ and invertible on $F$.  Moreover, $N$ is projective of rank equal
to the degree of $Q$ and the characteristic power series of $\phi$ on
$N$ is $Q(T)$.  The projector
$M_X\widehat{\otimes}_{\OO(X)}\OO(Y')\longrightarrow N$ is in the
closure of $\OO(Y')[\phi]$, so $N$ is stable under all of the
endomorphisms associated to elements of $\mathbf{T}$.  Let
$\mathbf{T}(Y)$ denote the $\OO(Y')$-subalgebra of $\End_{\OO(Y')}(N)$
generated by these endomorphisms.  Then $\mathbf{T}(Y)$ is finite over
$\OO(Y')$ and hence affinoid, so we we may set $D_Y =
\Sp(\mathbf{T}(Y))$.  Because the leading coefficient of $Q$ (= the
constant term of $Q^*$) is a unit there is an isomorphism
\begin{eqnarray*}
  \OO(Y')[T]/(Q(T)) & \longrightarrow & \OO(Y')[S]/(Q^*(S)) \\
  T &\longmapsto & S^{-1}
\end{eqnarray*}
Thus we obtain a canonical map $D_Y\longrightarrow Y$, namely, the one
corresponding to the map $$\OO(Y)\cong \OO(Y')[T]/(Q(T))\cong
\OO(Y')[S]/(Q^*(S)) \stackrel{S\mapsto \phi}{\longrightarrow
}\mathbf{T}(Y)$$
of affinoid algebras.

For general $Y\in \scr{C}(Z_X)$, we define $D_Y$ be the disjoint union
of the affinoids defined above from the various connected components
of $Y'$.  We then glue the affinoids $D_Y$ for $Y\in\scr{C}(Z_X)$ to
obtain a rigid space $D_X$ equipped with maps $$D_X\longrightarrow
Z_X\longrightarrow X.$$
Finally, we vary $X$ and glue the desired
spaces and maps above to obtain the spaces and maps in
(\ref{eigenspecweight}).  This final step is where the links
$\alpha_{ij}$ above come into play.  We refer the reader to
\cite{buzzardeigenvarieties} for further details.

\begin{defi}
  Let $L$ be a complete discretely-valued extension of $K$.  An
  $L$-valued system of eigenvalues of $\mathbf{T}$ acting on
  $\{M_X\}_X$ is a pair $(\kappa,\gamma)$ consisting of a map of sets
  $\gamma:\mathbf{T}\longrightarrow L$ and a point $\kappa\in
  \scr{W}(L)$ such that there exists an affinoid $X\subseteq \scr{W}$
  containing $\kappa$ and a nonzero element $m\in
  M_X\widehat{\otimes}_{\OO(X),\kappa} L$ such that
  $(t_X\widehat{\otimes} 1)m = \gamma(t)m$ for all $t\in \mathbf{T}$.
  Such a system of eigenvalues is called $\phi$-finite if
  $\gamma(\phi)\neq 0$.
\end{defi}
Let $x$ be an $L$-valued point of $D$.  Then $x$ lies over a point in
$\kappa_x\in\scr{W}(L)$ which lies in $X$ for some affinoid $X$, and $x$
moreover lies in $D_Y(L)$ for some $Y\in \scr{C}(Z_X)$.  Thus to $x$
and the choice of $X$ and $Y$ corresponds a map
$\mathbf{T}(Y)\longrightarrow L$, and in particular a map of sets
$\lambda_x:\mathbf{T}\longrightarrow L$. 
In \cite{buzzardeigenvarieties}, Buzzard proves the following
characterization of the points of $D$.
\begin{lemm}\label{eigenpoints}
  The correspondence $x\longmapsto (\kappa_x,\lambda_x)$ is a
  well-defined bijective correspondence between $L$-valued points of
  $D$ and $\phi$-finite $L$-valued systems of eigenvalues of
  $\mathbf{T}$ acting on the $\{M_X\}$.
\end{lemm}

In our case, we let $\scr{W}$ be weight space over $\Q_p$ as in Section
\ref{weightspace}, and let $\mathbf{T}$ be the set of symbols
$$\left\{\!\begin{array}{cc}
\{T_{\ell^2}\}_{\ell \not\ \!| 4Np}\cup\{U_{\ell^2}\}_{\ell|4Np}
\cup \{\ip{d}_{4N}\}_{d\in(\Z/4N\Z)^\times} & p\neq 2 \\
\{T_{\ell^2}\}_{\ell \not\ \!| 4N}\cup\{U_{\ell^2}\}_{\ell|4N}
\cup \{\ip{d}_{N}\}_{d\in(\Z/N\Z)^\times} & p=2 \end{array}\right.$$
For an admissible affinoid open $X\subseteq\scr{W}$ we let $$M_X =
\widetilde{M}_X(4N,\Q_p,p^{-r_n})$$
where $n$ is the smallest positive
integer such that $X\subseteq\scr{W}_n$.  This module is a direct
summand of the $\Q_p$-Banach space
$$\left\{\!\begin{array}{cc} H^0(X_1(4Np)^\an_{\geq
  p^{-r_n}},\OO(\Sigma_{4Np}))\widehat{\otimes}_{\Q_p}\OO(X) & p\neq 2
\\ H^0(X_1(4N)^\an_{\geq
  2^{-r_n}},\OO(\Sigma_{4N}))\widehat{\otimes}_{\Q_p}\OO(X) &
p=2\end{array}\right.$$ and therefore satisfies property ({\it Pr})
since this latter space is potentially orthonormalizable in the
terminology of \cite{buzzardeigenvarieties} by the discussion in
Section 1 of \cite{serrefredholm}.  We take the map
$$\mathbf{T}\longrightarrow
\End_{\OO(X)}(M_X)$$ to be the one sending each symbol to the endomorphism
by that name defined in Section \ref{sec:hecke}.

Let $X_1\subseteq X_2\subseteq\scr{W}$ be admissible affinoids and let
$n_i$ be the smallest positive integer with $X_i\subseteq
\scr{W}_{n_i}$.  Then $n_1\leq n_2$ so that $r_{n_2}\leq r_{n_1}$ and
we have an inclusion
$$\widetilde{M}_{X_1}(4N,\Q_p,p^{-r_{n_1}})\longrightarrow
\widetilde{M}_{X_1}(4N,\Q_p,p^{-r_{n_2}})$$
given by restriction.  We define
the required continuous injection $\alpha_{12}$ via the diagram
$$\xymatrix{\widetilde{M}_{X_1}(4N,\Q_p,p^{-r_{n_1}})\ar[r]
  \ar[dr]_{\alpha_{12}} & \widetilde{M}_{X_1}(4N,\Q_p,p^{-r_{n_2}})
  \\ &
  \widetilde{M}_{X_2}(4N,\Q_p,p^{-r_{n_2}})\widehat{\otimes}_{\OO(X_2)}
  \OO(X_1)\ar[u]_\sim}$$
and note that the required compatibility
condition is satisfied.  To see that these maps are links, choose
numbers $$r_{n_1}=s_0\geq s_1> s_2 > \cdots > s_{k-1} \geq s_k =
r_{n_2}$$
with the property that $p^2s_{i+1} >s_i$ for all $i$.  Then
the map $\alpha_{12}$ factors as the composition the maps
$$\widetilde{M}_{X_1}(4N,\Q_p,p^{-s_i}) \longrightarrow
\widetilde{M}_{X_1}(4N,\Q_p,p^{-s_{i+1}})$$ for $0\leq i\leq k-2$ and the
map $$\widetilde{M}_{X_1}(4N,\Q_p,p^{-s_{k-1}})\longrightarrow
\widetilde{M}_{X_2}(4N,\Q_p,p^{-s_k})\widehat{\otimes}_{\OO(X_2)} \OO(X_1).$$
Each of these maps is easily seen to be a primitive link from the
construction of $U_{p^2}$.

The result is that we obtain rigid analytic spaces $\widetilde{D}$ and
$\widetilde{Z}$ which we call \emph{the half-integral weight
  eigencurve} and \emph{the half-integral weight spectral curve},
respectively, as well as canonical  maps
$$\widetilde{D}\longrightarrow \widetilde{Z} \longrightarrow
\scr{W}.$$
As usual, the tilde serves to distinguish these spaces from
their integral weight counterparts first constructed in level 1 by
Coleman and Mazur and later constructed for general level by Buzzard
in \cite{buzzardeigenvarieties}.

If instead of using the full spaces of forms we use only the cuspidal
subspaces everywhere, then we obtain  cuspidal versions of all of
the above spaces, which we will delineate with a superscript $0$.
Thus we have $\widetilde{D}^0$ and $\widetilde{Z}^0$ with the usual
maps,  and the points of these spaces parameterize systems of eigenvalues
of the Hecke operators acting on the spaces of cusp forms by Lemma
\ref{eigenpoints}.  We remark that there is a commutative diagram
$$
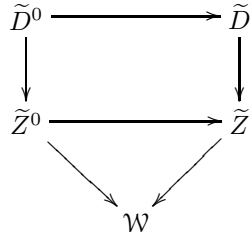
\xymatrix{ \widetilde{D}^0\ar[rr]\ar[d] & & \widetilde{D}\ar[d] \\
  \widetilde{Z}^0\ar[dr]\ar[rr] & & \widetilde{Z}\ar[dl] \\ & \scr{W}
  &}$$
where the horizontal maps are injections that identify the
cuspidal spaces on the left with unions of irreducible components of
the spaces on the right.  This is an exercise in the linear algebra
that goes into the construction of these eigenvarieties and basic
facts about irreducible components of rigid spaces found in
\cite{conradirredcpnts}, and is left to the reader.

For $\kappa\in \scr{W}(K)$, let $\widetilde{D}_\kappa$ and
$\widetilde{D}^0_\kappa$ denote the fibers $\widetilde{D}$ and
$\widetilde{D}^0$ over $\kappa$.  The following
theorem summarizes the basic properties of these eigencurves.
\begin{theo}
  Let $\kappa\in \scr{W}(K)$.  For a complete extension $L/K$, the
  correspondence $x\longmapsto \lambda_x$ is a bijection between the
  $L$-valued points of the fiber $\widetilde{D}_\kappa(L)$ and the set
  of finite-slope systems of eigenvalues of the Hecke
  operators and tame diamond operators occurring on the space
  $\widetilde{M}^\dagger_{\kappa}(4N,L)$ of overconvergent forms of
  weight $\kappa$ defined over $L$.  The same statement holds with
  $\widetilde{D}$ replaced by $\widetilde{D}^0$ and
  $\widetilde{M}^\dagger_{\kappa}(4N,L)$ replaced by
  $\widetilde{S}^\dagger_{\kappa}(4N,L)$.
\end{theo}
\begin{proof}
  We prove the statement for the full space of forms.  The proof for cuspidal
  forms is identical.  Fix $\kappa\in \scr{W}(K)$.  Once we establish
  that the $L$-valued systems of eigenvalues of the form
  $(\kappa,\gamma)$ occurring on the $\{M_X\}_X$ as defined above are
  exactly the systems of eigenvalues the Hecke and tame diamond
  operators that occur on $\widetilde{M}_{\kappa}^\dagger(4N,L)$, the
  result is simply Lemma \ref{eigenpoints} ``collated by weight.''  To
  see this one simply notes that, for any $f\in
  \widetilde{M}^\dagger_\kappa(4N,L)$, we have both $f\in
  \widetilde{M}_\kappa(4N,L,p^{-r_n})$ and $\kappa\in \scr{W}_n$ for
  $n$ sufficiently large.  In particular, if $f$ is a nonzero
  eigenform for the Hecke and tame diamond operators, then the system
  of eigenvalues associated to $f$ occurs in the module $M_{\scr{W}_n}$
  for $n$ sufficiently large.
\end{proof}

We remark that the classicality result of Section \ref{classicality}
has the expected consequence that the collection of points of
$\widetilde{D}$ corresponding to systems of eigenvalues occurring on
classical forms is Zariski-dense in $\widetilde{D}$.  This result is
contained in the forthcoming paper \cite{rigidshimura}.

\appendix\label{appendix}

\section{Properties of the stack $X_1(Mp,p^2)$ over $\Z_{(p)}$ \\
 {\rm by  Brian Conrad}}

\maketitle

\newcommand{\mathscr}{\scr}

In this appendix, we establish some geometric properties concerning
the cuspidal locus in compactified moduli spaces for level structures
on elliptic curves.  We are especially interested in the case of
non-\'etale $p$-level structures in characteristic $p$, so it is not
sufficient to cite the work in \cite{dr} (which requires \'etale level
structures in the treatment of moduli problems for generalized
elliptic curves) or \cite{katzmazur} (which works with Drinfeld structures
over arbitrary base schemes but avoids non-smooth generalized elliptic
curves).  The viewpoints of these works were synthesized in the study
of moduli stacks for Drinfeld structures on generalized elliptic
curves in \cite{conradamgec}, and we will use that as our foundation in
what follows.

Motivated by needs in the main text, for a prime $p$ and an integer $M
\ge 4$ not divisible by $p$ we wish to consider the moduli stack
$X_1(Mp^r,p^e)$ over $\Z_{(p)}$ that classifies triples $(E, P, C)$
where $E$ is a generalized elliptic curve over a $\Z_{(p)}$-scheme
$S$, $P \in E^{\rm{sm}}(S)$ is a Drinfeld $\Z/Mp^r\Z$-structure on
$E^{\rm{sm}}$, and $C \subseteq E^{\rm{sm}}$ is a cyclic subgroup with
order $p^e$ such that some reasonable ampleness and compatibility
properties for $P$ and $C$ are satisfied.  (See Definition
\ref{leveldef} for a precise formulation of these additional
properties.)  The relevant case for applications to $p$-adic modular
forms with half-integer weight is $e = 2$, but unfortunately such
moduli stacks were only considered in \cite{conradamgec} when either $r \ge
e$ or $r = 0$.  (This is sufficient for applications to Hecke
operators, and avoids some complications.)  We now need to allow $1
\le r < e$, and the purpose of this appendix is to explain how to
include such $r$ and to record some consequences concerning the cusps
in these cases.  The consequence that is relevant the main text is
Theorem \ref{app2}.  To carry out the proofs in this appendix we
simply have to adapt some proofs in \cite{conradamgec} rather than develop
any essentially new ideas.  For the convenience of the reader we will
usually use the single paper \cite{conradamgec} as a reference, though it
must be stressed that many of the key notions were first introduced in
the earlier work \cite{dr} and \cite{katzmazur}.  In the context of subgroups
of the smooth locus on a generalized elliptic curve, we will refer to
a Drinfeld $\Z/N\Z$-structure (resp. a Drinfeld $\Z/N\Z$-basis) as a
$\Z/N\Z$-structure (resp. $\Z/N\Z$-basis) unless some confusion is
possible.

\subsection{Definitions}

We refer the reader to \cite[\S2.1]{conradamgec} for the definitions of a
generalized elliptic curve $f:E \rightarrow S$ over a scheme $S$ and
of the closed subscheme $S^{\infty} \subseteq S$ that is the ``locus
of degenerate fibers'' for such an object.  (It would be more accurate
to write $S^{\infty,f}$, but the abuse of notation should not cause
confusion.)  Roughly speaking, $E \rightarrow S$ is a proper flat
family of geometrically connected and semistable curves of arithmetic
genus 1 that are either smooth or are so-called N\'eron polygons, and
the relative smooth locus $E^{\rm{sm}}$ is endowed with a commutative
$S$-group structure that extends (necessarily uniquely) to an action
on $E$ such that whenever $E_s$ is a polygon the action of
$E_s^{\rm{sm}}$ on $E_s$ is via rotations of the polygon.  Also,
$S^{\infty}$ is a scheme structure on the set of $s \in S$ such that
$E_s$ is not smooth.  The definition of the degeneracy locus
$S^{\infty}$ (as given in \cite[2.1.8]{conradamgec}) makes sense for any
proper flat and finitely presented map $C \rightarrow S$ with fibers
of pure dimension 1, and if $S'$ is any $S$-scheme and then there is
an inclusion $S' \times_S S^{\infty} \subseteq {S'}^{\infty}$ as
closed subschemes of $S'$ (with ${S'}^{\infty}$ corresponding to the
$S'$-curve $C \times_S S'$), but this inclusion can fail to be an
equality even when each geometric fiber $C_s$ is smooth of genus 1 or
a N\'eron polygon \cite[Ex.~2.1.11]{conradamgec}.  Fortunately, if $C$
admits a structure of generalized elliptic curve over $S$ then this
inclusion is always an equality \cite[2.1.12]{conradamgec}, so the
degeneracy locus makes sense on moduli stacks for generalized elliptic
curves (where it defines the cusps).

We wish to study moduli spaces for generalized elliptic curves
$E_{/S}$ equipped with certain ample level structures defined by
subgroups of $E^{\rm{sm}}$.  Of particular interest are those subgroup
schemes $G \subseteq E^{\rm{sm}}$ that are not only finite locally
free over the base with some constant order $n$ but are even {\em
  cyclic} in the sense that {\em fppf}-locally on the base we can
write $G = \langle P \rangle := \sum_{j \in \Z/n\Z} [jP]$ in
$E^{\rm{sm}}$ as Cartier divisors for some $n$-torsion point $P$ of
$E^{\rm{sm}}$.  By \cite[2.3.5]{conradamgec}, if $P$ and $P'$ are two such
points for the same $G$ then for any $d|n$ the points $(n/d)P$ and
$(n/d)P'$ are $\Z/(n/d)\Z$-generators of the same $S$-subgroup of $G$,
so by descent this naturally defines a cyclic $S$-subgroup $G_d
\subseteq G$ of order $d$ even if $P$ does not exist over the given
base scheme $S$.  We call $G_d$ the {\em standard} cyclic subgroup of
$G$ with order $d$.  For example, if $d = d' d''$ with $d', d'' \ge 1$
and $\gcd(d', d'') = 1$ then $G_{d'} \times G_{d''} \simeq G_d$ via
the group law on $G$.

\begin{defi}\label{leveldef}  Let $N, n \ge 1$ be integers.   
  
  A {\em $\Gamma_1(N)$-structure} on a generalized elliptic curve
  $E_{/S}$ is an $S$-ample $\Z/N\Z$-structure on $E^{\rm{sm}}$, which
  is to say an $N$-torsion point $P \in E^{\rm{sm}}(S)$ such that the
  relative effective Cartier divisor $D = \sum_{j \in \Z/N\Z} [jP]$ on
  $E^{\rm{sm}}$ is an $S$-subgroup and $D_s$ is ample on $E_s$ for all
  $s \in S$.
  
  A {\em $\Gamma_1(N,n)$-structure} on $E_{/S}$ is a pair $(P,C)$
  where $P$ is a $\Z/N\Z$-structure on $E^{\rm{sm}}$ and $C \subseteq
  E^{\rm{sm}}$ is a cyclic $S$-subgroup with order $n$ such that the
  relative effective Cartier divisor $D = \sum_{j \in \Z/N\Z} (jP +
  C)$ on $E$ is $S$-ample and there is an equality of closed
  subschemes
\begin{equation}\label{244}
  \sum_{j \in \Z/p^{e_p} \Z} (j(N/p^{e_p})P + C_{p^{e_p}}) = E^{\rm{sm}}[p^{e_p}]
\end{equation}
for all primes $p|\gcd(N,n)$, with $e_p = \ord_p(\gcd(N,n)) \ge 1$.
\end{defi}

\begin{exem} Obviously a $\Gamma_1(N,1)$-structure is the same thing
  as a $\Gamma_1(N)$-structure.  If $N = 1$ then we refer to
  $\Gamma_1(1)$-structures as $\Gamma(1)$-structures, and such a
  structure on a generalized elliptic curve $E_{/S}$ must be the
  identity section.  Thus, by the ampleness requirement, the geometric
  fibers $E_s$ must be irreducible.  Hence, the moduli stack
  $\mathscr{M}_{\Gamma(1)}$ of $\Gamma(1)$-structures on generalized
  elliptic curves classifies generalized elliptic curves with
  geometrically irreducible fibers.
\end{exem}  

In \cite[2.4.3]{conradamgec} the notion of $\Gamma_1(N,n)$-structure is
defined as above, but with the additional requirement that $\ord_p(n)
\le \ord_p(N)$ for all primes $p|\gcd(N,n)$.  This requirement always
holds when $n = 1$ and whenever it holds the standard subgroup
$C_{p^{e_p}}$ in (\ref{244}) is the $p$-part of $C$, but it turns out
to be unnecessary for the proofs of the basic properties of
$\Gamma_1(N,n)$-structures and their moduli, as we shall explain in
\S\ref{mstack}.  For example, the proof of \cite[2.4.4]{conradamgec}
carries over to show that we can replace (\ref{244}) with the
requirement that
$$\sum_{j \in \Z/d\Z} (j(N/d)P + C_d) = E^{\rm{sm}}[d]$$
in $E$ for $d
= \gcd(N,n)$.  Another basic property that carries over to the general
case is that if $(P,C)$ is a $\Gamma_1(N,n)$-structure on $E$ then the
relative effective Cartier divisor $\sum_{j \in \Z/N\Z} (jP + C)$ on
$E^{\rm{sm}}$ is an $S$-subgroup; the proof is given in
\cite[2.4.5]{conradamgec} under the assumption $\ord_p(n) \le \ord_p(N)$
for every prime $p|\gcd(N,n)$, but the argument works in general once
it is observed that after making an {\em fppf} base change to acquire
a $\Z/n\Z$-generator $Q$ of $C$ we can use symmetry in $P$ and $Q$ in
the rest of the argument so as to reduce to the case considered in
\cite{conradamgec}.

\subsection{Moduli stacks}\label{mstack}

As in \cite[2.4.6]{conradamgec}, for $N, n \ge 1$ we define the moduli
stack $\mathscr{M}_{\Gamma_1(N,n)}$ to classify
$\Gamma_1(N,n)$-structures on generalized elliptic curves over
arbitrary schemes, and we let $\mathscr{M}_{\Gamma_1(N,n)}^{\infty}
\hookrightarrow \mathscr{M}_{\Gamma_1(N,n)}$ denote the closed
substack given by the degeneracy locus for the universal generalized
elliptic curve.  The arguments in \cite[\S3.1--\S3.2]{conradamgec} carry
over verbatim (i.e., without using the condition $\ord_p(n) \le
\ord_p(N)$ for all primes $p|\gcd(N,n)$) to prove the following
result.

\begin{theo}\label{327}  The stack $\mathscr{M}_{\Gamma_1(N,n)}$ is an
  Artin stack that is proper over $\Z$.  It is smooth over $\Z[1/Nn]$,
  and it is Deligne--Mumford away from the open and closed substack in
  $\mathscr{M}_{\Gamma_1(N,n)}^{\infty}$ classifying degenerate
  triples $(E,P,C)$ in positive characteristics $p$ such that the
  $p$-part of each geometric fiber of $C$ is non-\'etale and
  disconnected.
\end{theo}

The proof of \cite[3.3.4]{conradamgec} does not use the condition
$\ord_p(n) \le \ord_p(N)$ for all primes $p|\gcd(N,n)$ (although this
condition is mentioned in the proof), so that argument gives:

\begin{lemm}\label{334} The open substack $\mathscr{M}_{\Gamma_1(N,n)}^0 =
  \mathscr{M}_{\Gamma_1(N,n)} - \mathscr{M}^{\infty}_{\Gamma_1(N,n)}$
  classifying elliptic curves endowed with a $\Gamma_1(N,n)$-structure
  is regular and $\Z$-flat with pure relative dimension $1$.
\end{lemm}

We are interested in the structure of $\mathscr{M}_{\Gamma_1(N,n)}$
around its cuspidal substack, especially determining whether it is
regular or a scheme near such points.  Our analysis of
$\mathscr{M}_{\Gamma_1(N,n)}^{\infty}$ rests on the following theorem.

\begin{theo}\label{331}
  The map $\mathscr{M}_{\Gamma_1(N,n)} \rightarrow \Spec(\Z)$ is flat
  and Cohen-Macaulay with pure relative dimension $1$.
\end{theo}

\begin{proof}   By Lemma \ref{334}, we just have to work along the cusps.
  Also, it suffices to check the result after localization at each
  prime $p$, and if $p \nmid \gcd(N,n)$ or $1 \le \ord_p(n) \le
  \ord_p(N)$ then \cite[3.3.1]{conradamgec} gives the result over
  $\Z_{(p)}$.  It therefore remains to study the cusps in positive
  characteristic $p$ when $1 \le \ord_p(N) < \ord_p(n)$.  As in the
  cases treated in \cite{conradamgec}, the key is to study the deformation
  theory of a related level structure on generalized elliptic curves
  called a {\em $\widetilde{\Gamma}_1(N,n)$-structure}: this is a pair
  $(P,Q)$ where $P$ is a $\Z/N\Z$-structure on the smooth locus and
  $Q$ is a $\Z/n\Z$-structure on the smooth locus such that $(P,
  \langle Q \rangle)$ is a $\Gamma_1(N,n)$-structure.  The same
  definition is given in \cite[3.3.2]{conradamgec} with the unnecessary
  restriction $\ord_p(n) \le \ord_p(N)$ for all primes $p|\gcd(N,n)$,
  and the argument in \cite{conradamgec} immediately following that
  definition works without such a restriction to show that the moduli
  stack $\mathscr{M}_{\widetilde{\Gamma}_1(N,n)}$ of
  $\widetilde{\Gamma}_1(N,n)$-structures is a Deligne--Mumford stack
  over $\Z$ that is a finite flat cover of the proper Artin stack
  $\mathscr{M}_{\Gamma_1(N,n)}$.
  
  By the Deligne--Mumford property, any
  $\widetilde{\Gamma}_1(N,n)$-structure $x_0 = (E_0, P_0, Q_0)$ over
  an algebraically closed field $k$ admits a universal deformation
  ring. Since $\mathscr{M}_{\widetilde{\Gamma}_1(N,n)}$ is a finite
  flat cover of $\mathscr{M}_{\Gamma_1(N,n)}$, as in the proof of
  \cite[3.3.1]{conradamgec} it suffices to assume ${\rm{char}}(k) = p > 0$
  and to exhibit the deformation ring at $x_0$ as a finite flat
  extension of $W(k)[\![x]\!]$ when $E_0$ is a standard polygon, $n =
  p^e$, and $N = Mp^r$ with $p \nmid M$ and $e, r \ge 1$.  The case $e
  \le r$ was settled in \cite{conradamgec}, and we will adapt that argument
  to handle the case $1 \le r < e$.  By the ampleness condition at
  least one of $MP_0$ or $Q_0$ generates the $p$-part of the component
  group of $E_0^{\rm{sm}}$, and moreover $\{MP_0, p^{e-r}Q_0\}$ is a
  Drinfeld $\Z/p^r\Z$-basis of $E^{\rm{sm}}_0[p^r]$.  We shall break
  up the problem into three cases, and it is only in Case 3 that we
  will meet a situation essentially different from that encountered in
  the proof for $1 \le e \le r$ in \cite{conradamgec}.
  
  {\sc Case 1}: We first assume that $MP_0$ generates the $p$-part of
  the component group, so by the Drinfeld $\Z/p^r\Z$-basis hypothesis
  this point is a basis of $E^{\rm{sm}}_0(k)[p^{\infty}]$ over
  $\Z/p^r\Z$ (as we are in characteristic $p$ and $E_0$ is a polygon).
  Hence, $Q_0 = j MP_0$ for a unique $j \in \Z/p^r\Z$ (so $p^{e-r}Q_0
  = p^{e-r}jMP_0$).  Since $n$ is a $p$-power, it also follows that
  $\langle P_0 \rangle$ is ample.  In particular, $(E_0,P_0)$ is a
  $\Gamma_1(N)$-structure.  Thus, the formation of an infinitesimal
  deformation $(E,P,Q)$ of $(E_0,P_0,Q_0)$ can be given in three
  steps: first give an infinitesimal deformation $(E,P)$ of
  $(E_0,P_0)$ as a $\Gamma_1(N)$-structure, then give a Drinfeld
  $\Z/p^r\Z$-basis $(MP,Q')$ of $E^{\rm{sm}}[p^r]$ with $Q'$ deforming
  $p^{e-r}Q_0$, and finally specify a $p^{e-r}$th root $Q$ of $Q'$
  lifting $Q_0 = jMP_0$.  The one aspect of this description that
  merits some explanation is to justify that such a $p^{e-r}$th root
  $Q$ of $Q'$ must be a $\Z/p^e\Z$-structure on $E^{\rm{sm}}$.  The
  point $Q$ is clearly killed by $p^e$, so the Cartier divisor $D =
  \sum_{j \in \Z/p^e\Z} [jQ]$ in $E^{\rm{sm}}$ makes sense and we have
  to check that it is automatically a subgroup scheme.
  
  The identification $(E^{\rm{sm}}_0)^0[p^t] = \mu_{p^t}$ uniquely
  lifts to an isomorphism $(E^{\rm{sm}})^0[p^t] \simeq \mu_{p^t}$ for
  any $t \ge 0$.  In particular, if $p^{\nu}$ is the order of the
  $p$-part of the cyclic component group of $E_0^{\rm{sm}}$ (with $\nu
  \ge r$) then $E^{\rm{sm}}[p^e]$ is an extension of $\Z/p^j\Z$ by
  $\mu_{p^e}$ where $j = \min(\nu,e)$.  The image of $\langle Q_0
  \rangle$ in the component group can be uniquely identified with
  $\Z/p^i\Z$ (for some $i \le j$) such that $Q_0 \mapsto 1$, and this
  $\Z/p^i\Z$ has preimage $G$ in $E^{\rm{sm}}[p^e]$ that is a
  $p^e$-torsion commutative extension of $\Z/p^i\Z$ by $\mu_{p^e}$
  with $0 \le i \le e$.  Since $Q$ is a point of $G$ over the (artin
  local) base, it follows from \cite[2.3.3]{conradamgec} that $Q$ is a
  $\Z/p^e\Z$-structure on $E^{\rm{sm}}$ if and only if the point $p^i
  Q$ in $\mu_{p^{e-i}}$ is a $\Z/p^{e-i}\Z$-generator of
  $\mu_{p^{e-i}}$.  The case $i = e$ is therefore settled, so we can
  assume $i < e$ (i.e., $\langle Q_0 \rangle$ is not \'etale, or
  equivalently $p^{e-1}Q_0 = 0$).  By hypothesis $p^{e-r}Q = Q'$ is a
  $\Z/p^r\Z$-structure on $E^{\rm{sm}}$ with $1 \le r < e$, so
  $p^{e-1}Q = p^{r-1}Q'$ is a $\Z/p\Z$-structure on $E^{\rm{sm}}$.
  This $\Z/p\Z$-structure must generate the subgroup $\mu_p \subseteq
  E^{\rm{sm}}[p^e]$ since $p^{e-1}Q$ lies in $(E^{\rm{sm}})^0$ (as
  $p^{e-1}Q_0 = 0$).  Hence, $Q'' = p^iQ$ is a point of
  $\mu_{p^{e-i}}$ such that $p^{e-i-1}Q''$ is a $\Z/p\Z$-generator of
  $\mu_p$.  Since $\Z/m\Z$-generators of $\mu_m$ are simply roots of
  the cyclotomic polynomial $\Phi_m$ \cite[1.12.9]{katzmazur}, our problem is
  reduced to the assertion that if $s$ is a positive integer (such as
  $e-i$) then an element $\zeta$ in a ring is a root of the cyclotomic
  polynomial $\Phi_{p^s}$ if $\zeta^{p^{s-1}}$ is a root of $\Phi_p$.
  This assertion is obvious since $\Phi_{p^s}(T) =
  \Phi_p(T^{p^{s-1}})$, and so our description of the infinitesimal
  deformation theory of $(E_0, P_0, Q_0)$ is justified.
  
  The torsion subgroup $E^{\rm{sm}}[p^r]$ is uniquely an extension of
  $\Z/p^r\Z$ by $\mu_{p^r}$ deforming the canonical such description
  for $E_0^{\rm{sm}}[p^r]$, so the condition on $Q'$ is that it has
  the form $\zeta + p^{e-r} j MP$ for a point $\zeta$ of the scheme of
  generators $\mu_{p^r}^{\times}$ of $\mu_{p^r} =
  (E^{\rm{sm}})^0[p^r]$.  Thus, to give $Q$ is to specify a
  $p^{e-r}$th root of $\zeta$ in $E^{\rm{sm}}$ deforming the identity,
  which is to say a point of $\mu_{p^e}^{\times}$.  It is shown in the
  proof of \cite[3.3.1]{conradamgec} that the universal deformation ring
  $A$ for $(E_0,P_0)$ is finite flat over $W(k)[\![x]\!]$, and the
  specification of $\zeta$ amounts to giving a root of the cyclotomic
  polynomial $\Phi_{p^e}$, so the case when $MP_0$ generates the
  $p$-part of the component group of $E_0^{\rm{sm}}$ is settled (with
  deformation ring $A[T]/(\Phi_{p^e}(T))$).
  
  {\sc Case 2}: Next assume that $Q_0$ generates the $p$-part of the
  component group and that $\langle Q_0 \rangle$ is \'etale (i.e.,
  $Q_0 \in E_0^{\rm{sm}}(k)$ has order $p^e$).  The point $Q_0$ must
  generate $E_0^{\rm{sm}}(k)[p^{\infty}]$ over $\Z/p^e\Z$, and the
  \'etale hypothesis ensures that $Q_0$ is a $\Z/p^e\Z$-basis of
  $E_0^{\rm{sm}}(k)[p^{\infty}]$.  Thus, $MP_0 = p^{e-r} j Q_0$ for
  some (unique) $j \in \Z/p^r\Z$.  By replacing $P$ with $P - M^{-1}
  p^{e-r} j Q$ for any infinitesimal deformation $(E,P,Q)$ of $(E_0,
  P_0, Q_0)$ we can assume that the $p$-part of $P_0$ vanishes.  The
  $p$-part of $P$ must therefore be a point of $\mu_{p^r}^{\times}$.
  The $\Z/M\Z$-part of $P$ together with $Q$ constitutes a
  $\Gamma_1(Mp^e)$-structure on $E$ (in particular, the ampleness
  condition holds), and this is an \'etale level structure since the
  cyclic subgroup $\langle Q_0 \rangle$ in $E_0^{\rm{sm}}$ is \'etale.
  Hence, the infinitesimal deformation functor of $(E_0, P_0, Q_0)$ is
  pro-represented by $\mu_{p^r}^{\times}$ over the deformation ring of
  an \'etale $\Gamma_1(Mp^e)$-structure.  For any $R \ge 1$,
  deformation rings for \'etale $\Gamma_1(R)$-structures on polygons
  over $k$ have the form $W(k)[\![x]\!]$ (as is explained near the end
  of the proof of \cite[3.3.1]{conradamgec}, using \cite[II,~1.17]{dr}), so
  not only are we done but in this case the deformation ring for
  $(E_0, P_0, Q_0)$ is the ring $W(k)[\![x]\!][T]/(\Phi_{p^r}(T))$
  that is visibly regular.
  
  {\sc Case 3}: Finally, assume $Q_0$ generates the $p$-part of the
  component group but that $\langle Q_0 \rangle$ is not \'etale (i.e.,
  $Q_0 \in E_0^{\rm{sm}}(k)$ has order strictly less than $p^e$), so
  $p^{e-r}Q_0 \in E_0^{\rm{sm}}(k)$ has order strictly dividing $p^r$.
  Since $\{MP_0, p^{e-r}Q_0\}$ is a Drinfeld $\Z/p^r\Z$-basis of
  $E_0^{\rm{sm}}[p^r]$, the point $MP_0$ must be a $\Z/p^r\Z$-basis
  for $E_0^{\rm{sm}}(k)[p^r]$.  Hence, if we write $P_0 = P'_0 +
  P''_0$ corresponding to the decomposition $\Z/N\Z = (\Z/M\Z) \times
  (\Z/p^r\Z)$ then $P''_0$ has order exactly $p^r$ in
  $E_0^{\rm{sm}}(k)$.  We use $P''_0$ to identify
  $E_0^{\rm{sm}}(k)[p^r]$ with $\Z/p^r\Z$.  It follows that if we make
  the analogous canonical decomposition $P = P' + P''$ for an
  infinitesimal deformation $(E, P, Q)$ of $(E_0, P_0, Q_0)$ then the
  $p$-part $P''$ deforms $P''_0$ and generates an \'etale subgroup of
  $E^{\rm{sm}}$ with order $p^r$.  Thus, $P'$ and $Q$ together
  constitute a (non-\'etale) $\Gamma_1(Mp^e)$-structure on $E$ (in
  particular, the ampleness condition holds), and the data of $P''$
  amounts to a section over $1 \in \Z/p^r\Z$ with respect to the
  unique quotient map $E^{\rm{sm}}[p^r] \twoheadrightarrow \Z/p^r\Z$
  lifting the quotient map $E_0^{\rm{sm}}[p^r] \twoheadrightarrow
  \Z/p^r\Z$ defined by $P''_0$.  Since the specification of a
  $\Z/N\Z$-structure on $E^{\rm{sm}}$ is the ``same'' as the
  specification of a pair consisting of $\Z/M\Z$-structure and a
  $\Z/p^r\Z$-structure \cite[1.7.3]{katzmazur}, we conclude that the
  universal deformation ring of $(E_0, P_0, Q_0)$ classifies the fiber
  over $1 \in \Z/p^r\Z$ in the connected-\'etale sequence for the
  $p^r$-torsion in infinitesimal deformations of the underlying
  $\Gamma_1(Mp^e)$-structure $(E_0, P'_0, Q_0)$.  Universal
  deformation rings for $\Gamma_1(Mp^e)$-structures over $k$ are
  finite flat over $W(k)[\![x]\!]$ (by the proof of
  \cite[3.3.1]{conradamgec}), so we are therefore done.
\end{proof}

\begin{coro}\label{411}  
  The closed substack $\mathscr{M}_{\Gamma_1(N,n)}^{\infty}
  \hookrightarrow \mathscr{M}_{\Gamma_1(N,n)}$ is a relative effective
  Cartier divisor over $\Z$, and it has a reduced generic fiber over
  $\Q$.
\end{coro}

\begin{proof}  The reducedness over $\Q$ is shown in \cite[4.3.2]{conradamgec},
  and the proof works without restriction on $\gcd(N,n)$.  Likewise,
  the proof that $\mathscr{M}_{\Gamma_1(N,n)}^{\infty}$ is a $\Z$-flat
  Cartier divisor is part of \cite[4.1.1(1)]{conradamgec} in case
  $\ord_p(n) \le \ord_p(N)$ for all primes $p|\gcd(N,n)$, but by using
  the above proof of Theorem \ref{331} we see that the method of proof
  works in general.
\end{proof}

Using Lemma \ref{334}, Theorem \ref{331}, and Corollary \ref{411},
Serre's normality criterion can be used to prove normality for
$\mathscr{M}_{\Gamma_1(N,n)}$ in general.  (This is proved in
\cite[4.1.4]{conradamgec} subject to the restrictions on $\gcd(N,n)$ in the
definition of $\Gamma_1(N,n)$-structures in \cite{conradamgec}, but the
argument works in general by using the results that are stated above
without any such restriction on $\gcd(N,n)$.)  However, the proof of
regularity encounters complications at points of a certain locus of
cusps in bad characteristics. This problematic locus is defined as
follows.

\begin{defi}  Let $\mathscr{Z}_{\Gamma_1(N,n)} \hookrightarrow
  \mathscr{M}_{\Gamma_1(N,n)}^{ \infty}$ be the 0-dimensional closed
  substack with reduced structure consisting of geometric points
  $(E_0, P_0, C_0)$ in characteristics $p|\gcd(N,n)$ such that $1 \le
  \ord_p(N) < \ord_p(n)$, $C_0$ is not \'etale, and
  $(N/p^{\ord_p(N)})P_0$ does not generate the $p$-part of the
  component group of $E_0^{\rm{sm}}$.
\end{defi}

Note that if $\ord_p(n) \le \ord_p(N)$ for all primes $p|\gcd(N,n)$
(the situation considered in \cite{conradamgec}) then
$\mathscr{Z}_{\Gamma_1(N,n)}$ is empty; this includes the case of
$\Gamma_1(N)$-structures for any $N$ (take $n = 1$).  In all other
cases it is non-empty.  The geometric points of
$\mathscr{Z}_{\Gamma_1(N,n)}$ correspond to precisely the points in
Case 3 in the proof of Theorem \ref{331}.  The method in \cite{conradamgec}
for analyzing regularity along the cusps assumes
$\mathscr{Z}_{\Gamma_1(N,n)}$ is empty, and by combining it with the
modified arguments in the proof of Theorem \ref{331} (especially the
regularity observation in Case 2) we obtain the following consequence.

\begin{theo}\label{remreg}
  The stack $\mathscr{M}_{\Gamma_1(N,n)}$ is regular outside of the
  closed substack $\mathscr{Z}_{\Gamma_1(N,n)} \subseteq
  \mathscr{M}_{\Gamma_1(N,n)}^{\infty}$.
\end{theo}

\subsection{Applications}\label{appsec}

Before we apply the preceding results, we record a useful lemma.

\begin{lemm}\label{scheme} Let $S$ be a scheme
  and let $\mathscr{X}$ be an Artin stack over $S$.  Assume
  $\mathscr{X}$ is $S$-separated.  The locus of geometric points of
  $\mathscr{X}$ with trivial automorphism group scheme is an open
  substack $\mathscr{U} \subseteq \mathscr{X}$ that is an algebraic
  space.  This algebraic space is a scheme if $\mathscr{X}$ is
  quasi-finite over a separated $S$-scheme.
\end{lemm}

\begin{proof}
  The first part is \cite[2.2.5(2)]{conradamgec}, and the second part
  follows from the general fact that an algebraic space that is
  quasi-finite and separated over a scheme is a scheme
  \cite[Thm.~A.2]{lm}.
\end{proof}

In the setting of Lemma \ref{scheme}, if $\mathscr{X}$ is quasi-finite
over a separated $S$-scheme then we call $\mathscr{U}$ the {\em
  maximal open subscheme} of $\mathscr{X}$.  The case of interest to
us is $\mathscr{X} = \mathscr{M}_{\Gamma_1(N,n)/S}$ over any scheme
$S$.  This is quasi-finite over the $S$-proper stack
$\mathscr{M}_{\Gamma(1)/S}$ via fibral contraction away from the
identity component, and $\mathscr{M}_{\Gamma(1)/S}$ is quasi-finite
over $\mathbf{P}^1_S$ via the $j$-invariant, so $\mathscr{X}$ is
quasi-finite over the separated $S$-scheme $\mathbf{P}^1_S$.

We wish to prove results concerning when certain components of
$\mathscr{M}_{\Gamma_1(N,n)}^{\infty}$ lie in the maximal open
subscheme of $\mathscr{M}_{\Gamma_1(N,n)}$.  To this end, we first
record a general lemma.

\begin{lemm}\label{multfiber} Let $\mathscr{Y}$ be an irreducible
  Artin stack over $\F_p$, and 
  let $\mathscr{C}$ be a finite locally free commutative
  $\mathscr{Y}$-group that is cyclic with order $p^e$.  If
  $\mathscr{C}$ has a multiplicative geometric fiber over
  $\mathscr{Y}$ then all of its geometric fibers are connected.
\end{lemm}

The abstract notion of cyclicity (with no ambient smooth curve group)
is developed in \cite[1.5,~1.9,~1.10]{katzmazur} over arbitrary base schemes,
and the theory carries over when the base is an Artin stack.  We will
only need the lemma for situations that arise within torsion on
generalized elliptic curves (over Artin stacks).

\begin{proof}
  We can assume $e \ge 1$, and we may replace $\mathscr{C}$ with its
  standard subgroup $\mathscr{C}_p$ of order $p$ because it is obvious
  by group theory that a {\em cyclic} group scheme $C$ of $p$-power
  order over an algebraically closed field of characteristic $p$ is
  \'etale if and only if its standard subgroup of order $p$ is
  \'etale.  Hence, we can assume that $\mathscr{C}$ has order $p$.
  Our problem is therefore to rule out the existence of \'etale
  fibers.  By openness of the locus of \'etale fibers and
  irreducibility of $\mathscr{Y}$, if there is an \'etale fiber then
  there is a Zariski-dense open $\mathscr{U} \subseteq \mathscr{Y}$
  over which $\mathscr{C}$ has \'etale fibers.  In particular, there
  is some geometric point $u$ of $\mathscr{U}$ that specializes to the
  geometric point $y \in \mathscr{Y}$ where we assume the fiber is
  multiplicative, so after pullback to a suitable valuation ring we
  get an \'etale group of order $p$ in characteristic $p$ specializing
  to a multiplicative one.  Passing to Cartier duals gives a
  multiplicative group of order $p$ having an \'etale specialization,
  and this is impossible since multiplicative groups of order $p$ in
  characteristic $p$ are not \'etale.
\end{proof}

\begin{theo}\label{app2}
  Let $p$ be a prime, and choose a positive integer $M$ not divisible
  by $p$ such that $M > 2$.  Also fix integers $e, r \ge 0$.  If $e =
  0$ or $r = 0$ then assume $M \ne 4$.  Let $x_0 = (E_0, P_0, C_0)$ be
  a geometric point on the special fiber of the cuspidal substack in
  the proper Artin stack $\mathscr{X} =
  \mathscr{M}_{\Gamma_1(Mp^r,p^e)/\Z_{(p)}}$ over $\Z_{(p)}$, and
  assume that $C_0$ is \'etale.
  
  Let $\mathscr{Y}$ be the irreducible component of $x_0$ in
  $\mathscr{X}_{\F_p}$.  For every geometric cusp $x_1 = (E_1, P_1,
  C_1)$ on $\mathscr{Y}$ the group $C_1$ is \'etale and $x_1$ lies in
  the maximal open subscheme of $\mathscr{X}$.  Moreover, if $x \in
  \mathscr{X}_{\Q}$ is a cusp specializing into $\mathscr{Y}$ then the
  Zariski closure $D$ of $x$ in $\mathscr{X}$ lies in the maximal open
  subscheme and $D$ is Cartier in $\mathscr{X}$.
\end{theo}

The case $e = 2$ is required in the main text.  It is necessary to
avoid the cases $M \le 2$ and $(M,r) = (4,0)$ because in these cases
there are cusps $x_0$ in characteristic $p$ as in the theorem such
that $x_0$ admits nontrivial automorphisms (and so $x_0$ cannot lie in
the maximal open subscheme of $\mathscr{X}$).

\begin{proof}
  We first check that the \'etale assumption at $x_0$ is inherited by
  all geometric cusps $x_1 \in \mathscr{Y}$.  Let
  $(\mathscr{E},\mathscr{P},\mathscr{C})$ be the pullback to
  $\mathscr{Y}$ of the universal family over $\mathscr{X}$.  The group
  $\mathscr{C}$ is cyclic of order $p^e$ with $e \ge 0$, so applying
  Lemma \ref{multfiber} to its Cartier dual gives the result (since at
  a cusp a connected subgroup of $p$-power order must be
  multiplicative).
  
  Now we can rename $x_1$ as $x_0$ without loss of generality, so we
  have to check that $x_0$ lies in the maximal open subscheme of
  $\mathscr{X}$ and that if $x \in \mathscr{X}_{\Q}$ is a geometric
  cusp specializing to $x_0$ then the Zariski closure of $x$ in
  $\mathscr{X}$ is Cartier.  But the \'etale hypothesis on $C_0$
  ensures that $x_0$ is not in the closed substack
  $\mathscr{Z}_{\Gamma_1(Mp^r,p^e)/\Z_{(p)}}$, so by Theorem
  \ref{remreg} the stack $\mathscr{X}$ is regular at $x_0$.  Hence,
  since $\mathscr{X}$ is $\Z_{(p)}$-flat with pure relative dimension
  1 (by Theorem \ref{331}), the desired properties of $D$ at the end
  of the theorem hold once we know that $x_0$ is in the maximal open
  subscheme of $\mathscr{X}$, which is to say that its automorphism
  group scheme $G$ is trivial.  To verify this triviality we will make
  essential use of the property that $C_0$ is \'etale.  Let $k$ be the
  algebraically closed field over which $x_0$ lives.  Since $E_0$ is
  $d$-gon over $k$ for some $d \ge 1$, $G$ is a closed subgroup of the
  automorphism group $\mu_d \rtimes \langle {\rm{inv}} \rangle$ of the
  $d$-gon.  Since $C_0$ is \'etale with order $p^e$ in characteristic
  $p$ it follows that $C_0$ maps isomorphically into the $p$-part of
  the component group of $E_0^{\rm{sm}} = \mathbf{G}_m \times
  (\Z/d\Z)$.  (In particular, $p^e|d$.)  If $R$ is an artin local
  $k$-algebra with residue field $k$ then any choice of generator
  $Q_0$ of $C_0$ must be carried to another generator of $C_0$ by any
  $g \in G(R)$ since $C_0(R) \rightarrow C_0(k)$ is a bijection.  But
  $\mu_d(R)$ acts on $(E_0)_R$ in a manner that preserves the
  components of the smooth locus, and $C_0$ meets each component of
  $E_0^{\rm{sm}}$ in at most one point.  Hence, $G \cap \mu_d$ acts as
  automorphisms of the $\Gamma_1(Mp^e)$-structure on $E_0$ defined by
  $p^rP_0$ and $Q_0$.  Since $Mp^e > 2$ and $Mp^e \ne 4$ (due to the
  cases we are avoiding), such an ample level structure on a $d$-gon
  has trivial automorphism group scheme.  This shows that $G \cap
  \mu_d$ is trivial, so $G$ injects into the group $\Z/2\Z$ of
  automorphisms of the identity component $\mathbf{G}_m$ of
  $E_0^{\rm{sm}}$.  Hence, the contraction operation on $E_0$ away
  from $\langle P_0 \rangle$ is faithful on $G$ since contraction does
  not affect the identity component.  It follows that $G$ is a
  subgroup of the automorphism group of the $\Gamma_1(Mp^r)$-structure
  obtained by contraction away from $\langle P_0 \rangle$.  But $Mp^r
  \not\in \{1,2,4\}$ since we assume $M > 2$ and $(M,r) \ne (4,0)$, so
  $\Gamma_1(Mp^r)$-structures on polygons have trivial automorphism
  functor.  Thus, $G = \{1\}$ as desired.
\end{proof}

We remark that, over the base $\Z_{(p)}$, the results of \S3-4 of
\cite{conradamgec} concerning the properties of the stack $X_1(N,n)$
carry over if $p\nmid n$.  In effect, the hypothesis on $\ord_p(n)$
imposed in \cite{conradamgec} only intervenes in the proofs when $n$
is not invertible on the base.

\bibliographystyle{smfplain}

\end{document}